\documentclass[11pt,leqno]{amsproc}

\usepackage{anysize}
\marginsize{3.5cm}{3.5cm}{2.5cm}{2.5cm}

\usepackage[]{hyperref}
\hypersetup{
    colorlinks=true,       
    linkcolor=red,          
    citecolor=blue,        
    filecolor=magenta,      
    urlcolor=cyan           
}

\usepackage{amsmath}
\usepackage{amsfonts,amssymb}
\usepackage{enumerate}
\usepackage{mathrsfs}
\usepackage{amsthm}

\theoremstyle{plain}
\newtheorem{theo}{Theorem}[section]
\newtheorem{prop}[theo]{Proposition}
\newtheorem{lemm}[theo]{Lemma}
\newtheorem{coro}[theo]{Corollary}

\newtheorem{defi}[theo]{Definition}
\theoremstyle{definition}
\newtheorem{rema}[theo]{Remark}
\newtheorem{nota}[theo]{Notation}
\DeclareMathOperator{\cnx}{div}
\DeclareMathOperator{\cn}{div}
\DeclareMathOperator{\RE}{Re}

\DeclareMathOperator{\supp}{supp}

\DeclareSymbolFont{pletters}{OT1}{cmr}{m}{sl}
\DeclareMathSymbol{s}{\mathalpha}{pletters}{`s}

\def\lDx#1{\langle D_x\rangle^{#1}\,}

\def\B{B }
\def\cnxy{\cn_{x,y}}

\def\da{a}
\def\defn{\mathrel{:=}}
\def\deta{\eta}
\def\dpsi{\psi}

\def\dzeta{\zeta}
\def\dB{B}
\def\dV{V}
\def\Deltax{\Delta}
\def\Deltayx{\Delta_{x,y}}

\def\eps{\varepsilon}
\def\eval{\arrowvert_{y=\eta}}

\def\Hs{\mathcal{H}_\sigma}

\def\la{\left\vert}
\def\lA{\left\Vert}
\def\le{\leq}
\def\les{\lesssim}
\def\leo{}

\def\ma{a}
\def\mez{\frac{1}{2}}
\def\partialx{\nabla}
\def\partialyx{\nabla_{x,y}}
\def\ra{\right\vert}
\def\rA{\right\Vert}
\def\s{\sigma}
\def\tdm{\frac{3}{2}}

\def\xC{\mathbf{C}}
\def\xN{\mathbf{N}}
\def\xR{\mathbf{R}}
\def\xS{\mathbf{S}}

\def\xZ{\mathbf{Z}}

\def\va{\varphi}
\def\K{\mathcal{K}}

\def\mR{\mathcal{R}}
\def\VCe{\lA V\rA_{C^{1+\eps}_*}}
\def\pO{\lA p\rA_{L^\infty}}

\numberwithin{equation}{section}

\title{On the Cauchy problem for gravity water waves}
\author{
T. Alazard, 
N. Burq, 
C. Zuily}
\thanks{
T.A. was supported by the French Agence Nationale de la Recherche, projects ANR-08-JCJC-0132-01 and ANR-08-JCJC-0124-01.
}
\begin{document}


\begin{abstract}
We are interested in the system of gravity water waves equations 
without surface tension. Our purpose is to study the optimal regularity 
thresholds  for the initial conditions. 
In terms of Sobolev embeddings, the initial surfaces we consider turn 
out to be only of~$C^{3/2+\epsilon}$-class for some $\epsilon>0$ and consequently have unbounded curvature, 
while the initial velocities are only Lipschitz. 
We reduce the system using a paradifferential approach. 
\end{abstract}

\maketitle

\section{Introduction}

We are interested in this work in the study of the Cauchy problem 
for the water waves system  in arbitrary dimension, without surface tension. 

An important question in the theory is the possible 
emergence of singularities (see~\cite{CCFGGS, CCFGLF,CoSh,Wu12, CMSW}) and as emphasized by Craig and 
Wayne~\cite{CW}, it is important to decide whether some physical or 
geometric quantities control the equation. In terms of the velocity field, 
a  natural criterion (in view of Cauchy-Lipschitz theorem) is given by the 
Lipschitz regularity threshold. Indeed, this is necessary for the ``fluid particles'' motion 
(i.e.\ the integral curves of the velocity field) to be well-defined. 

In terms of the free boundary,  there is no such natural criterium. In fact, the systematic use 
of the Lagrangian formulation in most previous works~\cite{AmMa,WuInvent,WuJAMS}, and the intensive 
use of Riemannian geometry tools (parallel transport, vector fields,...) by 
Shatah-Zeng~\cite{SZ,SZ2,SZ3}, 
Christodoulou--Lindblad~\cite{ChLi} or Lindblad~\cite{LindbladAnnals}
seem to at least require {\em bounded curvature assumptions} (see also~\cite{CS} where a logarithmic divergence is allowed).
In this direction, the beautiful work by Christodoulou--Lindblad~\cite{ChLi}, gives {\em a priori\/} bounds as long as 
the second fundamental form of the free surface is bounded, 
and the first-order derivatives of the velocity are bounded. This  could lead to the natural conjecture 
that the regularity threshold for the water waves system is indeed 
given by Christodoulou--Lindblad's result and that the domain 
has to be assumed to be essentially~$C^2$. 
Our main contribution in this work is that this is {\em not} the case and that the 
relevant threshold is actually only the Lipschitz regularity of the velocity field. 
Indeed (see Theorem~\ref{theo:Cauchy}), our local existence result involves assumptions which,
  in view of Sobolev embeddings, 
 require only (in terms of H\"older regularity) 
the initial free domain to be~$C^{3/2+\epsilon}$ for some $\epsilon>0$.

As an illustration of the relevance 
of the  analysis of low regularity solutions in a domain with a rough boundary, 
let us mention that in a forthcoming paper, we shall give an application of our analysis to the local Cauchy theory 
of three-dimensional gravity water waves in a canal. This question 
goes back to the work by 
Boussinesq at the beginning of the $20^{\text{th}}$ century (see~\cite{Boussinesq}).

Our analysis require 
the introduction of new techniques and new tools. In~\cite{ABZ1,ABZ2} 
we started a para-differential study of the water waves system in the presence 
of surface tension and were able to prove that the equations 
can be reduced to a simple form
\begin{equation}\label{formepara}
\partial_t u+T_V\cdot \nabla u +iT_\gamma u=f,
\end{equation}
where~$T_V$ is a para-product and~$T_\gamma$ is a para-differential operator of 
order~$3/2$. Here the main step in the proof is to perform the same task without surface tension, 
with~$T_\gamma$ of order~$1/2$. It has to be noticed however that performing our reduction is considerably more 
difficult here than in our previous papers~(\cite{ABZ1,ABZ2}). 
Indeed, in the case with non vanishing surface tension, the natural regularity threshold 
forces the velocity field to be Lipschitz while the domain is actually much smoother ($C^{5/2}$). 
In the present work, 
the velocity field is also Lipschitz, but the domain is merely $C^{3/2}$. 
To overcome these difficulties, we had  to give a micro-local  description (and contraction estimates) 
of the Dirichlet-Neumann operator which is non trivial in the whole range of  $C^s$ domains, $s>1$ 
(see the work by Dahlberg-Kenig~\cite{DK} and Craig-Schanz-Sulem~\cite{CSS} for results 
on the Dirichlet-Neumann operator in Lipschitz domains). We think that this analysis is of independent interest. 

Finally, let us mention that, as we proceed by energy estimates, 
our results are proved in $L^2$-based Sobolev spaces and 
our initial data $(\eta, V)$ which describe 
respectively the initial domain as the graph of the function $\eta$ and 
the trace of the initial velocity on the free surface, 
are assumed to be in $H^{s+\mez}( \xR^d)\times H^s( \xR^d), s> 1+ \frac d 2$. 
The gravity water waves system enjoys a scaling invariance 
for which the critical threshold is $s_c= \mez +\frac d 2$ (in other terms 
our well-posedness result is  $1/2$ above the scaling critical index).

\subsection{Assumptions on the domain}\label{assu:domain}
Hereafter,~$d\ge 1$,~$t$ 
denotes the time variable and~$x\in \xR^d$ and~$y\in \xR$ denote the horizontal and vertical 
spatial variables.
We work in a time-dependent fluid domain~$\Omega$ 
located underneath 
a free surface~$\Sigma$ and moving  
in a fixed container denoted by~$\mathcal{O}$. This fluid domain
$$
\Omega=\{\,(t,x,y)\in [0,T]\times{\mathbf{R}}^d\times\xR \, : \, (x,y) \in \Omega(t)\,\},
$$
is such that, for each time~$t$, one has
$$
\Omega(t)=\left\{ (x,y)\in \mathcal{O} \, :\, y < \eta(t,x)\right\},
$$
where $\eta$ is an unknown function and~$\mathcal{O}$ is a given open domain 
which contains a fixed strip around the free surface 
$$
\Sigma=\{ (t,x,y)\in [0,T]\times{\mathbf{R}}^d\times \xR\,:\, y=\eta(t,x)\}.
$$
This implies that 
there exists~$h>0$ such that, for all~$t\in [0,T]$,
\begin{equation}\label{hypt}
\Omega_{h}(t)\defn \left\{ (x,y)\in {\mathbf{R}}^d\times \xR \, :\, \eta(t,x)-h < y < \eta(t,x)\right\} \subset \Omega(t).
\end{equation}
We also assume that the domain~$\mathcal{O}$ (and hence the domain~$\Omega(t)$) 
is connected.
\begin{rema}
\begin{enumerate}[(i)]
\item  
Two classical examples are given by~$\mathcal{O}=\xR^{d}\times \xR$ (infinite depth case) or 
$\mathcal{O}=\xR^d\times [-1,+\infty)$ (flat bottom). 
Notice that, in the following, no regularity assumption is made on the 
bottom~$\Gamma\defn \partial\mathcal{O}$. 
\item Notice that~$\Gamma$ does not depend on time. However, 
our method applies in the case where the bottom is time dependent 
(with the additional assumption in this case that the bottom is Lipschitz). 
\end{enumerate}
\end{rema}
\subsection{The equations}
Below we use the following notations
$$
\partialx=(\partial_{x_i})_{1\le i\le d},\quad \partialyx =(\partialx,\partial_y), 
\quad \Delta = \sum\nolimits_{{1\le i\le d}} \partial_{x_i}^2,\quad 
\Deltayx = \Delta+\partial_y^2.
$$
We consider an incompressible inviscid liquid, having unit density. 
The equations by which the motion is to be determined are well known. 
Firstly, the Eulerian velocity field~$v\colon \Omega \rightarrow \xR^{d+1}$ 
solves the incompressible Euler equation
\begin{equation}\label{cond1}
\partial_{t} v +v\cdot \partialyx v + \partialyx P = - g e_y ,\quad \cnxy v =0 \quad\text{in }\Omega,
\end{equation}
where~$-ge_y$ is the acceleration of gravity ($g>0$) and 
where the pressure term~$P$ can be recovered from 
the velocity by solving an elliptic equation. 
The problem is then given by three boundary conditions. They are
\begin{equation}\label{systbound}
\left\{\begin{aligned}
&v\cdot n=0 &&\text{on }\Gamma, \\
&\partial_{t} \eta = \sqrt{1+|\partialx \eta|^2}\, v \cdot \nu \quad 
&&\text{on }\Sigma, \\
& P=0
&&\text{on }\Sigma,
\end{aligned}\right.
\end{equation}
where~$n$ and~$\nu$ are the exterior 
unit normals to the bottom~$\Gamma$ and the free surface~$\Sigma(t)$. 
The first condition in~\eqref{systbound} expresses the fact 
that the particles in contact with the rigid bottom remain 
in contact with it. Notice that to fully make sense, 
this condition requires some smoothness on~$\Gamma$, 
but in general, it has a weak variational meaning (see Section~\ref{sec.3}). 
The second condition in~\eqref{systbound} states that the 
free surface moves with the fluid and the last condition 
is a balance of forces across the free surface. Notice that 
the pressure at the upper surface of the fluid may be indeed 
supposed to be zero, provided we 
afterwards add the atmospheric pressure to the pressure so determined. 
The fluid motion is supposed to be irrotational. The velocity field 
is therefore given by~$v=\nabla_{x,y} \phi$ 
for some potential~$\phi\colon \Omega\rightarrow \xR$ satisfying 
$$
\Delta_{x,y}\phi=0\quad\text{in }\Omega, \qquad \partial_n\phi=0\quad\text{on }\Gamma.
$$
Using the Bernoulli integral of the dynamical equations to express the pressure, 
the condition~$P=0$ on the free surface implies that
\begin{equation}\label{syst2}
\left\{
\begin{aligned}
&\partial_{t} \eta = \partial_y \phi-\nabla\eta\cdot\nabla\phi
&&\text{on }\Sigma, \\
& \partial_t \phi+\mez \la \nabla_{x,y}\phi\ra^2+g y  =0
&&\text{on }\Sigma,\\
&\partial_n \phi =0 &&\text{on }\Gamma,
\end{aligned}
\right.
\end{equation}
where recall that~$\nabla=\nabla_x$. 
Many results have been obtained on the Cauchy theory for System~\eqref{syst2}, 
starting from the pioneering works 
of Nalimov~\cite{Nalimov}, Shinbrot~\cite{Shinbrot}, Yoshihara~\cite{Yosihara}, 
Craig~\cite{Craig1985}. In the framework of Sobolev spaces and 
without smallness assumptions on the data, the well-posedness 
of the Cauchy problem was first proved by Wu for the case 
without surface tension (see~\cite{WuInvent,WuJAMS}) 
and by Beyer-G\"unther in~\cite{BG} in the case with surface tension. 
Several extensions of their results have been obtained 
by different methods (see~\cite{CMST,GMS,GMS2,GT,LannesKelvin,MR,Wu09,Wu10,ZhangZhang} 
for recent results and the surveys~\cite{BL,CW,LannesLivre} for more references). 
Here we shall use the Eulerian formulation. Following Zakharov~\cite{Zakharov1968} and 
Craig--Sulem~\cite{CrSu}, 
we reduce the analysis to a system on the free surface 
$\Sigma(t)=\{y= \eta(t,x)\}$. 
If~$\psi$ is defined by 
$$
\psi(t,x)=\phi(t,x,\eta(t,x)),
$$
then~$\phi$ is the unique variational solution of 
$$
\Delta_{x,y} \phi = 0 \text{ in } \Omega, 
\quad \phi\arrowvert_{y=\eta}  = \psi , \quad \partial_n \phi=0 \text{ on }\Gamma.
$$
Define the Dirichlet-Neumann operator by 
\begin{align*}
(G(\eta) \psi)  (t,x)&=\sqrt{1+|\partialx\eta|^2}\,
\partial _n \phi\arrowvert_{y=\eta(t,x)}\\
&=(\partial_y \phi)(t,x,\eta(t,x))-\partialx \eta (t,x)\cdot (\partialx \phi)(t,x,\eta(t,x)).
\end{align*}
For the case with a rough bottom, we recall the precise 
construction later on (see~\S\ref{S:defiDN}).
Now ~$(\eta,\psi)$ solves (see~\cite{CrSu} or~\cite[chapter 1]{LannesLivre} for instance)
\begin{equation}\label{system}
\left\{
\begin{aligned}
&\partial_{t}\eta-G(\eta)\psi=0,\\[0.5ex]
&\partial_{t}\psi+g \eta
+ \frac{1}{2}\la\partialx \psi\ra^2  -\frac{1}{2}
\frac{\bigl(\partialx  \eta\cdot\partialx \psi +G(\eta) \psi \bigr)^2}{1+|\partialx  \eta|^2}
= 0.
\end{aligned}
\right.
\end{equation}

\subsection{The Taylor condition}

Introduce the so-called Taylor coefficient
\begin{equation}\label{defi:ma}
\ma(t,x)=-(\partial_y P)(t,x,\eta(t,x)).
\end{equation}
The stability of the waves is dictated by 
the Taylor sign condition, which is the assumption that there exists a positive 
constant~$c$ such that
\begin{equation}\label{taylor}
\ma(t,x)\ge c >0.
\end{equation}
This assumption is now classical and we refer 
to~\cite{BL,ChLi,CoCoGa,LannesJAMS,WuInvent,WuJAMS} 
for various comments. 
Here we only recall some basic facts. 
First of all, as proved by Wu~(\cite{WuInvent,WuJAMS}), 
this assumption is automatically satisfied in the infinite depth case 
(that is when~$\Gamma=\emptyset$) or 
for flat bottoms (when~$\Gamma=\{ y=-k\}$). Notice 
that the proof remains valid for any~$C^{1,\alpha}$-domain, $0<\alpha<1$ (by using 
the fact that the Hopf Lemma is true for such domains, see~\cite{Safonov} and the references therein). 
There are two other cases where this assumption 
is known to be satisfied. For instance under a smallness assumption. Indeed, if 
$\partial_t \phi=O(\eps^2)$ and~$\nabla_{x,y}\phi=O(\eps)$ then directly from the definition 
of the pressure we have~$P+gy = O(\eps^2)$. 
Secondly, it was proved by Lannes~(\cite{LannesJAMS}) 
that the Taylor's assumption is satisfied 
under a smallness assumption on the curvature of the bottom 
(provided that the bottom is at least~$C^2$). 
However, for general bottom we will assume that~\eqref{taylor} is satisfied at time~$t=0$. 
\subsection{Main result}
We work below with the vertical and horizontal traces of the velocity 
on the free boundary, namely
\begin{gather*}
B\defn (\partial_y \phi)\arrowvert_{y=\eta},\quad 
V \defn (\nabla_x \phi)\arrowvert_{y=\eta}.
\end{gather*}
These can be defined only in terms of~$\eta$ and~$\psi$ by means of the formulas 
\begin{equation}\label{defi:BV}
B= \frac{\partialx \eta \cdot\partialx \psi+ G(\eta)\psi}{1+|\partialx  \eta|^2},
\qquad
V= \partialx \psi -B \partialx\eta.
\end{equation}
Also, recall that the Taylor coefficient~$\ma$ defined in~\eqref{defi:ma} 
can be defined in terms of~$\eta,V,\B,\psi$ only (see Section~\ref{sec.3.2} below).
\begin{theo}
\label{theo:Cauchy}
Let~$d\ge 1$,~$s>1+d/2$ and consider~$(\eta_{0},\psi_{0})$ such that
\begin{enumerate}
\item \label{assu1}
$\eta_0\in H^{s+\mez}({\mathbf{R}}^d),\quad \psi_0\in H^{s+\mez}({\mathbf{R}}^d),
\quad V_0\in H^{s}({\mathbf{R}}^d),\quad B_0\in H^{s}({\mathbf{R}}^d)$,
\item there exists~$h>0$ such that condition \eqref{hypt} holds initially for~$t=0$,
\item there exists a positive constant~$c$ such that, for all~$x\in {\mathbf{R}}^d$, 
$\ma_0(x)\ge c$.
\end{enumerate}
Then there exists~$T>0$ such that 
the Cauchy problem for \eqref{system} 
with initial data $(\eta_{0},\psi_{0})$ has a unique solution 
$
(\eta,\psi)\in C^0\big([0,T];H^{s+\mez}({\mathbf{R}}^d)\times H^{s+\mez}({\mathbf{R}}^d)\big),
$
such that 
\begin{enumerate}
\item we have~$(V,B)\in C^0\big([0,T];H^{s}({\mathbf{R}}^d)\times H^{s}({\mathbf{R}}^d)\big)$,
\item the condition \eqref{hypt} holds for 
$0\le t\le T$, with~$h$ replaced by~$h/2$,
\item for all~$0\le t\le T$ and for all~$x\in {\mathbf{R}}^d$, 
$\ma(t,x)\ge c/2$.
\end{enumerate}
\end{theo}
\begin{rema}
The main novelty is that, in view of Sobolev embeddings, the 
initial surfaces we consider turn out to be only of~$C^{3/2+\epsilon}$-class for some 
$\epsilon>0$ and consequently 
have unbounded curvature.
\end{rema}
\begin{rema}
Assumption \ref{assu1} in the above theorem is automatically satisfied if 
$$
\eta_0\in H^{s+\mez}({\mathbf{R}}^d),
\quad \psi_0\in H^{\mez}({\mathbf{R}}^d),
\quad V_0\in H^{s}({\mathbf{R}}^d),
\quad B_0\in H^{\mez}({\mathbf{R}}^d).
$$

The only point where the estimates depend on~$\psi$ (and not only on~$\eta,V,\B$) 
come from the fact that we consider a general domain without assumption on the bottom. 
Otherwise, 
we shall prove {\em a priori\/} estimates 
for the fluid velocity and not for the fluid potential (notice that 
the fluid potential is defined up to a constant). 
More precisely, 
the $H^{s}$-norm of $\psi$ is used only in the proof of an estimate 
(cf. (4.10)) which holds trivially, say, in infinite depth (since the quantity $\gamma$ 
which is estimated in (4.10) is $0$ in infinite depth). 
We also refer the reader to Theorem 4.35 of Lannes \cite{LannesLivre} for a well-posedness result 
requiring only that $\psi$ belongs to an homogeneous Sobolev space.
\end{rema}

 \subsection{The pressure}\label{sec.3.2}
The purpose of this paragraph is to 
clarify, for low regularity solutions of the water waves system in rough domains, 
the definition of the pressure which is required if one wants to come back 
from solutions to the Zakharov system to solutions to the free boundary Euler equation. 
This definition will also provide the basic {\em a priori} 
estimates which will be later the starting point when 
establishing higher order elliptic regularity estimates required when 
studying the Taylor coefficient~$a= - \partial _y P \mid_{\Sigma}$. 
On a physics point of view, the pressure is the Lagrange multiplier 
which is required by the incompressibility of the fluid 
(preservation of the null divergence condition). 
As a consequence, taking the divergence in~\eqref{cond1}, 
it is natural to define the pressure as a solution of 
\begin{equation}\label{pression}
 \Delta_{x,y} P = - \cn_{x,y}( v\cdot\nabla_{x,y} v), \qquad P \mid_{y=\eta} =0.
 \end{equation}
Notice however that the solution of such a problem may not be 
unique as can be seen in the simple 
case when~$\Omega= (-\infty,0) \times \xR^d$. 
Indeed, if~$P$ is a solution, then~$P +cy$ is another. 
Notice also that if~$P$ satisfies~\eqref{pression}, then 
$$
\Delta_{x,y} \Bigl(P+gy+\frac 1 2  |v |^2\Bigr) =0.
$$
\begin{defi}\label{def.4.1} Let~$(\eta, \psi) \in (W^{1, \infty}\cap H^{1/2}(\xR^d)) \times H^{1/2} ( \xR^d)$. 
Assume that the variational solution (as defined in~$\S\ref{S:defiDN}$) of the equation 
\begin{equation}
\label{varia}
\Delta_{x,y} \phi=0, \quad \phi \mid_{y=\eta} = \psi,
\end{equation}
satisfies 
$$
|\nabla_{x,y} \phi|^2 (x, \eta(x) ) \in H^{1/2} (\xR^d).
$$     
Let~$R$ be the variational solution of 
$$
\Delta_{x,y} R =0 \text{ in } \Omega, \qquad R \mid_{y=\eta} 
=   g \eta+ \frac 1 2  |\nabla_{x,y} \phi|^2 \mid_{y=\eta}.
$$ 
We define the pressure~$P$ in the domain~$\Omega$ by  
$$
P(x,y)\defn R(x,y) -gy - \frac 1 2 |\nabla_{x,y} \phi(x,y)|^2.
$$
\end{defi}
\begin{rema} The main advantage of defining the pressure 
as the solution of a variational problem is that it will satisfy 
automatically an {\em a priori} estimate (the estimate given by the variational theory).
\end{rema}
It remains to link the 
solutions to the Zakharov system to solutions of the free boundary Euler system~\eqref{cond1} 
with boundary conditions~\eqref{systbound}. To do so, 
we proved in~\cite{ABZ-Bertinoro} that if~$(\eta, \psi)$ is a solution of the Zakharov system, 
if we consider the variational solution to~\eqref{varia}, 
then the velocity field~$v= \partialyx \phi$ satisfies~\eqref{cond1}, 
which is of course equivalent to
\begin{equation}\label{egalite}
P = -\partial_t \phi -gy -\mez \vert \nabla_{x,y}\phi \vert^2.
\end{equation}
\begin{theo}[from~\cite{ABZ-Bertinoro}]
Assume that~$(\eta, \psi) \in C([0,T]; H^{s+ \mez}(\xR^d) \times  H^{s+ \mez}(\xR^d) )$, with  
$s> 1/2+d/2$, is a solution of the Zakharov/Craig-Sulem system~\eqref{system}. 
Then the assumptions required to define the pressure are satisfied, 
and~\eqref{egalite} is satisfied, and  the distribution~$\partial_t \phi$ 
is well defined for fixed~$t$ and belongs to the space~$H^{1,0}( \Omega(t))$ 
(see Definition~$\ref{defi:H10}$).  
\end{theo}
\subsection{Plan of the paper}
At first glance, Theorem~\ref{theo:Cauchy} looks very similar to our previous result in presence 
of surface tension~\cite[Theorem 1.1]{ABZ1}. Indeed, the regularity threshold 
exhibited by the velocity field (namely~$V, B \in H^{s}({\mathbf{R}}^d), s>1+d/2$) is the 
same in both results and (as explained above) appears to be the natural one. 
However, an important difference between both cases is that the algebraic 
nature of~\eqref{system} (and its counter-part in presence of surface tension) 
requires that the free domain is~$3/2$ smoother than the velocity field in presence of surface tension 
and only~$1/2$ smoother without surface tension. 
This algebraic rigidity of the system 
implies that in order to lower the regularity threshold to the natural one (Lipschitz velocities), we are 
forced to work with~$C^{3/2}~$ domains (compared to the much smoother~$C^{5/2}$ 
regularity in~\cite{ABZ1}). This in turn poses new challenging questions in the 
study of the Dirichlet--Neumann operator. Indeed, at this level of regularity the 
regularity of the remainder term in the paradifferential description of the 
Dirichlet-Neumann operator~$G(\eta) \psi$ is not given by the regularity 
of the function~$\psi$ itself, but rather by the regularity of the domain. 
This is this phenomenon which forces us to work with the new unknowns 
$V,B$ rather than with~$\psi$. 

In Section~\ref{sec:2}, we wrote a review of paradifferential 
calculus and proved various technical results useful in the article. 
In Section~\ref{sec.3} we  study the Dirichlet-Neumann operator. 
In Section~\ref{S:6}, we symmetrize the system 
and prove {\em a priori} estimates. 
In Section~\ref{S:contraction} we prove the contraction estimates 
required to show uniqueness and stability of solutions. In particular we prove 
a contraction estimate for the difference of two Dirichlet-Neumann operators, 
involving only (in terms of Sobolev embedding) the $C^{\mez }$-norm of the difference of the functions defining 
the domains (see Theorem~\ref{L:p60}), while in 
Section~\ref{S:regul} we prove the existence of solutions by a 
regularization process. 

\smallbreak
\textbf{Acknowledgements.} We would like to thank the referees for their comments which led to a better version of this paper.

\section{Paradifferential calculus}\label{sec:2}
Let us review notations and results about Bony's paradifferential calculus. 
We refer to \cite{Bony,Hormander,MePise,Meyer} for the general theory. 
Here we follow the presentation by M\'etivier in \cite{MePise}.
\subsection{Paradifferential operators}\label{s2}
For~$k\in\xN$, we denote by $W^{k,\infty}({\mathbf{R}}^d)$ the usual Sobolev spaces.
For $\rho= k + \sigma$, $k\in \xN, \sigma \in (0,1)$ denote 
by~$W^{\rho,\infty}({\mathbf{R}}^d)$ 
the space of functions whose derivatives up to order~$k$ are bounded and uniformly H\"older continuous with 
exponent~$\sigma$. 
\begin{defi}
Given~$\rho\in [0, 1]$ and~$m\in\xR$,~$\Gamma_{\rho}^{m}({\mathbf{R}}^d)$ denotes the space of
locally bounded functions~$a(x,\xi)$
on~${\mathbf{R}}^d\times({\mathbf{R}}^d\setminus 0)$,
which are~$C^\infty$ with respect to~$\xi$ for~$\xi\neq 0$ and
such that, for all~$\alpha\in\xN^d$ and all~$\xi\neq 0$, the function
$x\mapsto \partial_\xi^\alpha a(x,\xi)$ belongs to~$W^{\rho,\infty}({\mathbf{R}}^d)$ and there exists a constant
$C_\alpha$ such that,
\begin{equation*}
\forall\la \xi\ra\ge \mez,\quad 
\lA \partial_\xi^\alpha a(\cdot,\xi)\rA_{W^{\rho,\infty}(\xR^d)}\le C_\alpha
(1+\la\xi\ra)^{m-\la\alpha\ra}.
\end{equation*}
\end{defi}
Given a symbol~$a$, we define
the paradifferential operator~$T_a$ by
\begin{equation}\label{eq.para}
\widehat{T_a u}(\xi)=(2\pi)^{-d}\int \chi(\xi-\eta,\eta)\widehat{a}(\xi-\eta,\eta)\psi(\eta)\widehat{u}(\eta)
\, d\eta,
\end{equation}
where
$\widehat{a}(\theta,\xi)=\int e^{-ix\cdot\theta}a(x,\xi)\, dx$
is the Fourier transform of~$a$ with respect to the first variable; 
$\chi$ and~$\psi$ are two fixed~$C^\infty$ functions such that:
\begin{equation}\label{cond.psi}
\psi(\eta)=0\quad \text{for } \la\eta\ra\le \frac{1}{5},\qquad
\psi(\eta)=1\quad \text{for }\la\eta\ra\geq \frac{1}{4},
\end{equation}
and~$\chi(\theta,\eta)$ 
satisfies, for~$0<\eps_1<\eps_2$ small enough,
$$
\chi(\theta,\eta)=1 \quad \text{if}\quad \la\theta\ra\le \eps_1\la \eta\ra,\qquad
\chi(\theta,\eta)=0 \quad \text{if}\quad \la\theta\ra\geq \eps_2\la\eta\ra,
$$
and such that
$$
\forall (\theta,\eta)\,:\qquad \la \partial_\theta^\alpha \partial_\eta^\beta \chi(\theta,\eta)\ra\le 
C_{\alpha,\beta}(1+\la \eta\ra)^{-\la \alpha\ra-\la \beta\ra}.
$$

Since we shall need to work with paraproducts, we chose a cut-off function~$\chi$ such that 
when~$a=a(x)$,~$T_a$ is given by the usual expression in terms of the Littlewood-Paley operators. Namely, 
the function $\chi$ can be constructed as follows. Let~$\kappa\in C^\infty_0({\mathbf{R}}^d)$ be such that
$$
\kappa(\theta)=1\quad \text{for }\la \theta\ra\le 1.1,\qquad 
\kappa(\theta)=0\quad \text{for }\la\theta\ra\ge 1.9.
$$
Then we define
$
\chi(\theta,\eta)=\sum_{k=0}^{+\infty} \kappa_{k-3}(\theta) \varphi_k(\eta)
$,
where
\begin{equation*}
\kappa_k(\theta)=\kappa(2^{-k}\theta)\quad\text{for }k\in \xZ,
\qquad \varphi_0=\kappa_0,\quad\text{ and } 
\quad \varphi_k=\kappa_k-\kappa_{k-1} \quad\text{for }k\ge 1.
\end{equation*}

Given a temperate distribution $u$ and an integer $k$ in $\xN$ we also introduce $S_k u$ 
and $\Delta_k u$ by 
$S_k u=\kappa_k(D_x)u$ and $\Delta_k u=S_k u-S_{k-1}u$ for $k\ge 1$ and $\Delta_0u=S_0u$. 

\subsection{Symbolic calculus}\label{sec.2.2}
We shall use quantitative results from \cite{MePise} about operator norms estimates in symbolic calculus. 
Introduce the following semi-norms.
\begin{defi}\label{defiGmrho}
For~$m\in\xR$,~$\rho\in [0,1]$ and~$a\in \Gamma^m_{\rho}({\mathbf{R}}^d)$, we set
\begin{equation}\label{defi:norms}
M_{\rho}^{m}(a)= 
\sup_{\la\alpha\ra\le 2(d+2) +\rho ~}\sup_{\la\xi\ra \ge 1/2~}
\lA (1+\la\xi\ra)^{\la\alpha\ra-m}\partial_\xi^\alpha a(\cdot,\xi)\rA_{W^{\rho,\infty}({\mathbf{R}}^d)}.
\end{equation}
\end{defi}
\begin{defi}[Zygmund spaces] Consider a Littlewood-Paley decomposition: 
$u=\sum_{q=0}^\infty \Delta_q u$ (with the notations introduced above). 
If~$s$ is any real number, we define the Zygmund class~$C^{s}_*({\mathbf{R}}^d)$ as the 
space of tempered distributions~$u$ such that
$$
\lA u\rA_{C^{s}_*}\defn \sup_q 2^{qs}\lA \Delta_q u\rA_{L^\infty}<+\infty.
$$
\end{defi}
\begin{rema}
Recall that $C^{s}_*({\mathbf{R}}^d)$ is the 
H\"older 
space~$W^{s,\infty}({\mathbf{R}}^d)$ if~$s\in (0,+\infty)\setminus \xN$. 
\end{rema}

\begin{defi}\label{defi:order}
Let~$m\in\xR$.
An operator~$T$ is said to be of  order~$\leo m$ if, for all~$\mu\in\xR$,
it is bounded from~$H^{\mu}$ to~$H^{\mu-m}$. 
\end{defi}

The main features of symbolic calculus for paradifferential operators are given by the following theorem.
\begin{theo}\label{theo:sc0}
Let~$m\in\xR$ and~$\rho\in [0,1]$. 

$(i)$ If~$a \in \Gamma^m_0({\mathbf{R}}^d)$, then~$T_a$ is of order~$\leo m$. 
Moreover, for all~$\mu\in\xR$ there exists a constant~$K$ such that
\begin{equation}\label{esti:quant1}\lA T_a \rA_{H^{\mu}\rightarrow H^{\mu-m}}\le K M_{0}^{m}(a).
\end{equation}
$(ii)$ If~$a\in \Gamma^{m}_{\rho}({\mathbf{R}}^d), b\in \Gamma^{m'}_{\rho}({\mathbf{R}}^d)$ then 
$T_a T_b -T_{a b}$ is of order~$\leo m+m'-\rho$. 
Moreover, for all~$\mu\in\xR$ there exists a constant~$K$ such that
\begin{equation}\label{esti:quant2}
\lA T_a T_b  - T_{a b}   \rA_{H^{\mu}\rightarrow H^{\mu-m-m'+\rho}}
\le 
K M_{\rho}^{m}(a)M_{0}^{m'}(b)+K M_{0}^{m}(a)M_{\rho}^{m'}(b).
\end{equation}
$(iii)$ Let~$a\in \Gamma^{m}_{\rho}({\mathbf{R}}^d)$. Denote by 
$(T_a)^*$ the adjoint operator of~$T_a$ and by~$\overline{a}$ the complex conjugate of~$a$. Then 
$(T_a)^* -T_{\overline{a}}$ is of order~$\leo m-\rho$. 
Moreover, for all~$\mu$ there exists a constant~$K$ such that
\begin{equation}\label{esti:quant3}
\lA (T_a)^*   - T_{\overline{a}}   \rA_{H^{\mu}\rightarrow H^{\mu-m+\rho}}\le 
K M_{\rho}^{m}(a).
\end{equation}
\end{theo}
We shall need in this article to consider paradifferential operators with negative regularity. 
As a consequence, we need to extend our previous definition.
\begin{defi}
For~$m\in \xR$ and~$\rho\in (-\infty, 0)$,~$\Gamma^m_\rho(\xR^d)$ denotes the space of
distributions~$a(x,\xi)$
on~${\mathbf{R}}^d\times(\xR^d\setminus 0)$,
which are~$C^\infty$ with respect to~$\xi$ and 
such that, for all~$\alpha\in\xN^d$ and all~$\xi\neq 0$, the function
$x\mapsto \partial_\xi^\alpha a(x,\xi)$ belongs to $C^\rho_*({\mathbf{R}}^d)$ and there exists a constant
$C_\alpha$ such that,
\begin{equation}
\forall\la \xi\ra\ge \mez,\quad \lA \partial_\xi^\alpha a(\cdot,\xi)\rA_{C^\rho_*}\le C_\alpha
(1+\la\xi\ra)^{m-\la\alpha\ra}.
\end{equation}
For~$a\in \Gamma^m_\rho$, we define 
\begin{equation}
M_{\rho}^{m}(a)= 
\sup_{\la\alpha\ra\le 2(d +2)+\vert \rho \vert   ~}\sup_{\la\xi\ra \ge 1/2~}
\lA (1+\la\xi\ra)^{\la\alpha\ra-m}\partial_\xi^\alpha a(\cdot,\xi)\rA_{C^{\rho}_*({\mathbf{R}}^d)}.
\end{equation}
\end{defi}
\subsection{Paraproducts and product rules}\label{sec.2.3}
If~$a=a(x)$ is a function of~$x$ only, the paradifferential operator~$T_a$ is called a paraproduct. 
A key feature of paraproducts is that one can replace 
nonlinear expressions by paradifferential expressions up to smoothing operators. 
Also, one can define paraproducts~$T_a$ for rough functions~$a$ which do not belong to~$L^\infty(\xR^d)$ but merely to 
$C_*^{-m}({\mathbf{R}}^d)$ with~$m>0$. 
\begin{defi}
Given two functions~$a,b$ defined on~${\mathbf{R}}^d$ we define the remainder 
$$
R(a,u)=au-T_a u-T_u a.
$$
\end{defi}
We record here various estimates about paraproducts (see chapter 2 in~\cite{BCD} or~\cite{Chemin}).
\begin{theo}
\begin{enumerate}[i)]
\item  Let~$\alpha,\beta\in \xR$. If~$\alpha+\beta>0$ then
\begin{align}
&\lA R(a,u) \rA _{H^{\alpha + \beta-\frac{d}{2}}({\mathbf{R}}^d)}
\leq K \lA a \rA _{H^{\alpha}({\mathbf{R}}^d)}\lA u\rA _{H^{\beta}({\mathbf{R}}^d)},\label{Bony} \\ 
&\lA R(a,u) \rA _{H^{\alpha + \beta}({\mathbf{R}}^d)} \leq K \lA a \rA _{C^{\alpha}_*({\mathbf{R}}^d)}\lA u\rA _{H^{\beta}({\mathbf{R}}^d)}.\label{Bony3}
\end{align}
\item   Let~$m>0$ and~$s\in \xR$. Then
\begin{equation}
\lA T_a u\rA_{H^{s-m}}\le K \lA a\rA_{C^{-m}_*}\lA u\rA_{H^{s}}.\label{niS}
\end{equation}
\item Let~$s_0,s_1,s_2$ be such that 
$s_0\le s_2$ and~$s_0 < s_1 +s_2 -\frac{d}{2}$, 
then
\begin{equation}\label{boundpara}
\lA T_a u\rA_{H^{s_0}}\le K \lA a\rA_{H^{s_1}}\lA u\rA_{H^{s_2}}.
\end{equation}
\end{enumerate}
\end{theo}

By combining the two previous points with the embedding~$H^\mu({\mathbf{R}}^d)\subset C^{\mu-d/2}_*({\mathbf{R}}^d)$ (for any~$\mu\in\xR$) 
we immediately obtain the following results. 
\begin{prop}\label{lemPa}
Let~$r,\mu\in \xR$ be such that~$r+\mu>0$. If~$\gamma\in\xR$ satisfies 
$$
\gamma\le r  \quad\text{and}\quad \gamma < r+\mu-\frac{d}{2},
$$
then there exists a constant~$K$ such that, for all 
$a\in H^{r}({\mathbf{R}}^d)$ and all~$u\in H^{\mu}({\mathbf{R}}^d)$, 
$$
\lA au - T_a u\rA_{H^{\gamma}}\le K \lA a\rA_{H^{r}}\lA u\rA_{H^\mu}.
$$
\end{prop} 
\begin{coro}
\begin{enumerate}[i)]
\item \label{it.1} If~$u_j\in H^{s_j}({\mathbf{R}}^d)$ ($j=1,2$) with $s_1+s_2> 0$ then 
\begin{equation}\label{pr}
\lA u_1 u_2 \rA_{H^{s_0}}\le K \lA u_1\rA_{H^{s_1}}\lA u_2\rA_{H^{s_2}},
\end{equation}
if $s_0\le s_j$, $j=1,2$, and $s_0< s_1+s_2-d/2$.

\item \label{it.2} (Tame estimate in Sobolev spaces) If~$s\ge 0$ then
\begin{equation}\label{prS2}
\lA u_1 u_2 \rA_{H^{s}}\le K \bigl(\lA u_1\rA_{H^{s}}\lA u_2\rA_{L^{\infty}}+ \lA u_1\rA_{L^{\infty}} \lA u_2\rA_{H^{s}}\bigr).
\end{equation}

\item \label{it.4} Let~$\mu,m\in \xR$ be such that~$\mu,m> 0$ and~$m\not \in \xN$. Then
\begin{equation}\label{prZ2}
\lA u_1 u_2\rA_{H^\mu}\le K \bigl(\lA u_1\rA_{L^\infty}\lA u_2\rA_{H^\mu}+\lA u_2\rA_{C^{-m}_*}\lA u_1\rA_{H^{\mu+m}}\bigr).
\end{equation}

\item \label{it.6} Let~$s>d/2$ and consider~$F\in C^\infty(\xC^N)$ such that~$F(0)=0$. 
Then there exists a non-decreasing function~$\mathcal{F}\colon\xR_+\rightarrow\xR_+$ 
such that, for any~$U\in H^s({\mathbf{R}}^d)^N$,
\begin{equation}\label{esti:F(u)}
\lA F(U)\rA_{H^s}\le \mathcal{F}\bigl(\lA U\rA_{L^\infty}\bigr)\lA U\rA_{H^s}.
\end{equation}
\end{enumerate}
\end{coro}
\begin{proof}
The first two estimates  are well-known, 
see H\"ormander \cite{Hormander} or Chemin \cite{Chemin}. 
To prove $\ref{it.4}$) 
we write 
$u_1u_2=T_{u_1} u_2+T_{u_2} u_1 +R(u_1,u_2)$ 
and use that
\begin{alignat*}{2}
&\lA T_{u_1} u_2\rA_{H^\mu} \les \lA u_1\rA_{L^\infty}\lA u_2\rA_{H^\mu} \qquad &&(\text{see }\eqref{esti:quant1}),\\
&\lA T_{u_2} u_1\rA_{H^\mu} \les \lA u_2\rA_{C^{-m}_*}\lA u_1\rA_{H^{\mu+m}} 
&&(\text{see }\eqref{niS}),\\
&\lA R(u_1,u_2)\rA_{H^\mu} \les \lA u_2\rA_{C^{-m}_*}\lA u_1\rA_{H^{\mu+m}} \quad
&&(\text{see }\eqref{Bony3}).
\end{alignat*}
Finally, $\ref{it.6}$) is due to Meyer~\cite[Th\'eor\`eme 2.5 and remarque]{Meyer}. 
\end{proof}

Finally, let us finish this section with a generalization of~\eqref{niS} 
\begin{prop}\label{prop.niSbis}
Let~$\rho<0$,~$m\in \xR$ and~$a\in \dot{ \Gamma}^m_\rho$. Then the operator~$T_a$ is of order~$m-\rho$:
\begin{equation}\label{niSbis}
\| T_a \|_{H^s \rightarrow H^{s-(m- \rho)}}\leq C M_{\rho}^{m}(a),\qquad
\| T_a \|_{C^s_* \rightarrow C^{s-(m- \rho)}_*}\leq C M_{\rho}^{m}(a).
\end{equation}
\end{prop}
\begin{proof}
Let us prove the first estimate. The proof of the second is similar. 
Notice that if~$m=0$ and~$a(x, \xi) = a(x)$, then ~\eqref{niSbis} 
is simply ~\eqref{niS}. Furthermore, if~$a(x, \xi) = b(x) p(\xi)$, then
$T_a = T_b \circ p(D)$. 
  As a consequence, using the first step and the fact that   $p(\xi) = \vert \xi \vert^m p\big (\frac{\xi}{\vert \xi \vert}\big) $ we obtain
$$\| T_a \|_{H^s \rightarrow H^{s-(m- \rho)}}\leq C \|b\|_{C^\rho_*} \| p \mid_{\mathbf{S}^{d-1}}\|_{L^\infty}.$$
 In the general case, we can expand, for fixed~$x$,~$a(x,\xi)$ 
 in terms of spherical harmonics. Let~$(\widetilde{h}_\nu)_{\nu\in \xN^*}$ 
 be an orthonormal basis 
of~$L^2(\xS^{d-1})$ consisting of eigenfunctions of the (self-adjoint) Laplace--Beltrami operator, 
$\Delta_\omega=\Delta_{\xS^{d-1}}$ on~$L^2(\xS^{d-1})$, {\em i.e.} 
$
\Delta_\omega \widetilde{h}_\nu 
=\lambda^2_{\nu}\widetilde{h} _\nu.
$
By the Weyl formula, we know that~$\lambda_\nu\sim c \nu^{\frac 1 d}$. Setting 
$$
h_\nu(\xi)= \la \xi\ra^m \widetilde{h}_\nu (\omega),\quad 
\omega=\frac{\xi}{\la\xi\ra},~\xi\neq 0,
$$
we can write
$$
a(x,\xi)=\sum_{\nu\in\xN^*} a_\nu(x) h_\nu(\xi)\quad\text{where}\quad 
a_\nu(x)=\int_{\xS^{d-1}} a(x,\omega) \overline{\widetilde{h}_\nu(\omega)}\, d\omega.
$$
Then $T_a = \sum_{\nu \in \xN^*} T_{a_\nu h_\nu}.$  Now since we have
$$
\lambda_\nu^{2k} a_\nu (x)=\int_{\xS^{d-1}} \Delta_w^k a (x,\omega) \overline{\widetilde{h}_\nu(\omega)}\, d\omega,
$$
taking $k=d$ we obtain  for   all~$\nu \ge 1$, 
\begin{equation*}
\begin{aligned}
 \lA a_\nu \rA_{C^\rho_*} &\le \lambda_\nu^{-2d} 
 \int_{\xS^{d-1}} \vert \Delta_w^d a (x,\omega)\vert \vert \overline{\widetilde{h}_\nu(\omega)} \vert\, d\omega 
 \le  C_1 \lambda_\nu^{-2d} {M}^m_\rho(a) \Vert {\widetilde{h}_\nu} \Vert_{L^2( {\xS^{d-1}})}  ,\\
  &  \le  C_2 \nu ^{-2} {M}^m_\rho(a).
 \end{aligned}
 \end{equation*}
Moreover by the Sobolev embedding we have, with $\eps>0$ small
\begin{equation}
\big\Vert \widetilde{h}_\nu\big\Vert_{L^\infty(S^{d-1})}\le C_3 \lambda_\nu^{\frac{ d-1} 2 + \eps} \leq C_4 \nu^{\frac 1 2 - \frac{ 1} {2d} + \frac{\eps}{d}} \leq C_4 \nu^{\frac 1 2}.
\end{equation}
Therefore by the second step we can conclude that
$$
\| T_a \| _{H^s \rightarrow H^{s-(m- \rho)}}\leq C_5 \sum_{\nu\in \xN^*} \nu^{ - \frac 3 {2} } {M}^m_\rho(a).
$$
This completes the proof.
\end{proof}

We shall also need the following technical result.

\begin{prop}\label{comut}
Set~$ \langle D_x \rangle=(I-\Delta)^{1/2}$. 

\begin{enumerate}[i)]
\item \label{itt.1} Let~$s>\mez + \frac{d}{2}$ and~$\sigma \in \xR$ be such that~$\sigma \leq s$. 
Then there exists~$K>0$ such that for all~$V \in W^{1, \infty}({\mathbf{R}}^d) \cap H^s({\mathbf{R}}^d)$ 
and ~$u \in H^{\sigma- \mez}({\mathbf{R}}^d)$  one has
$$
\Vert [\lDx{\sigma}, V] u \Vert_{L^2({\mathbf{R}}^d)} 
\leq K \big \{ \Vert V \Vert_{W^{1, \infty}({\mathbf{R}}^d)} 
+ \Vert V \Vert_{H^s({\mathbf{R}}^d) }\big \} \Vert u \Vert_{H^{\sigma- \mez}({\mathbf{R}}^d)}.
$$
\item \label{itt.2} Let~$s>1+\frac{d}{2}$ and~$\sigma \in \xR$ be such that~$\sigma \leq s$. 
Then there exists~$K>0$ such that for all~$V \in   H^s({\mathbf{R}}^d)$  and ~$u \in H^{\sigma- 1}({\mathbf{R}}^d)$  one has
$$
\Vert [\lDx{\sigma}, V] u \Vert_{L^2({\mathbf{R}}^d)} \leq K \Vert V \Vert_{H^s({\mathbf{R}}^d) }  \Vert u \Vert_{H^{\sigma- 1}({\mathbf{R}}^d)}.
$$
\end{enumerate}
\end{prop}
\begin{proof}
To prove~$\ref{itt.1}$) we write
\begin{equation*}
\Vert [\lDx{\sigma}, V] u \Vert_{L^2} \leq A+B,\qquad A = \Vert [\lDx{\sigma}, T_V] u \Vert_{L^2},
\quad B = \Vert [\lDx{\sigma}, V-T_V] u \Vert_{L^2}.
\end{equation*}
By \eqref{esti:quant2} we have 
$A \leq K  \Vert V \Vert_{W^{1, \infty}}\Vert u \Vert_{H^{\sigma- 1}}$. 
On the other hand one can write
$$
B \leq 
\Vert  \lDx{\sigma}(V-T_V) u \Vert_{L^2} + \Vert (V-T_V)\lDx{\sigma} u \Vert_{L^2}
= B_1 + B_2.
$$
We use Proposition \ref{lemPa} two times. 
To estimate~$B_1$ we take~$\gamma = \sigma, r = s, \mu = \sigma - \mez.$ 
To estimate~$B_2$ we take~$\gamma =0, r = s, \mu = -\mez~$ and we obtain,
$$
B \leq K \Vert V \Vert_{H^s}  \Vert u \Vert_{H^{\sigma- \mez}}.
$$

To prove~$\ref{itt.2}$),  to estimate~$B_1$ (resp.~$B_2$) we use again 
Proposition \ref{lemPa} 
with $\gamma = \sigma, r = s, \mu = \sigma - 1~$ (resp.~$\gamma = 0, r =s, \mu = -1)$. 
\end{proof}

We shall need well-known estimates on the solutions 
of transport equations.
\begin {prop}\label{estclass}
Let~$I=[0,T]$, $s> 1 + \frac d 2$ and consider the Cauchy problem
\begin{equation}\label{transport}
\left\{
\begin{aligned}
&\partial_t u + V\cdot \nabla u =f, \quad t \in I,\\
&u\arrowvert_{t=0} = u_0.
\end{aligned}
\right.
\end{equation}
We have the following estimates
\begin{equation}\label{eq.i}  
\Vert u(t) \Vert_{L^\infty({\mathbf{R}}^d)} 
\leq \Vert u_0 \Vert_{L^\infty({\mathbf{R}}^d)} 
+ \int_0^t \Vert f(\sigma,\cdot) \Vert_{L^\infty({\mathbf{R}}^d)}d\sigma.
\end{equation}
There exists a non decreasing function  $\mathcal{F}:\xR^+ \to \xR^+$ such that
\begin{equation}
\Vert u(t) \Vert_{L^2({\mathbf{R}}^d)} 
\leq
\mathcal{F}\big(\Vert V \Vert_{L^1(I; W^{1, \infty}({\mathbf{R}}^d))}\big)
\big(\Vert u_0 \Vert_{L^2({\mathbf{R}}^d)} + \int_0^t \Vert f(t',\cdot)\Vert_{L^2({\mathbf{R}}^d)}\,
dt' \big).\label{eq.ii}
\end{equation}
and for any $\sigma\in [0,s]$ 
there exists a non decreasing function~$\mathcal{F}:\xR^+ \to \xR^+$ 
such that 
\begin{equation}\label{eq.iii}
\Vert u(t) \Vert_{H^{\sigma}({\mathbf{R}}^d)}  \\
\leq \mathcal{F} \bigl( \Vert V\Vert_{L^1(I; H^s({\mathbf{R}}^d))}\bigr)
\bigl(\Vert u_0 \Vert_{H^{\sigma}({\mathbf{R}}^d)} 
+ \int_0^t \Vert f(t',\cdot) \Vert_{H^{\sigma}({\mathbf{R}}^d)}\,dt' \big)  .
\end{equation}
\end{prop}

\subsection{Commutation with a vector field}\label{S:Dt} 

We prove in this paragraph 
a commutator estimate between a paradifferential operator~$T_p$ and the 
convective derivative~$\partial_t +V\cdot \nabla$. Inspired by Chemin~\cite{Chemin88b} and Alinhac~\cite{Alinhac88}, 
we prove an estimate which depends on estimates  on~$\partial_t p+V\cdot \nabla p$ and not on~$\nabla_{t,x}p$. 

When~$a$ and~$u$ are
symbols and functions depending on~$t\in I$, we still denote by~$T_a u$ the spatial
paradifferential operator (or paraproduct) such that for all~$t\in I$, 
$(T_a u)(t)=T_{a(t)}u(t)$. Given a symbol~$a=a(t;x,\xi)$ depending on time, we use the notation
$$
\mathcal{M}^m_0(a)\defn \sup_{t\in [0,T]} \sup_{\la\alpha\ra\le   2(d+2)+ \vert \rho \vert}\sup_{\la\xi\ra \ge 1/2~}
\lA (1+\la\xi\ra)^{\la\alpha\ra-m}\partial_\xi^\alpha a(t;\cdot,\xi)\rA_{L^{\infty}({\mathbf{R}}^d)}.
$$
Given a scalar symbol~$p=p(t,x,\xi)$ of order~$m$, 
it follows directly from the symbolic calculus rules for paradifferential operators 
(see~\eqref{esti:quant1} and \eqref{esti:quant2}) that,
\begin{equation*}
\lA \bigl[ T_p , \partial_t +T_V \cdot \nabla\bigr] u\rA_{H^\mu}\le K  \left\{\mathcal{M}^m_0 (\partial_t p)+
\mathcal{M}^m_0 (\nabla p)\lA V\rA_{W^{1,\infty}}\right\}  \lA u\rA_{H^{\mu+m}}.
\end{equation*}
A technical key point in our analysis is that one can replace this estimate by a tame 
estimate which does not involve the first order derivatives of~$p$, but instead 
$\partial_t p+V\cdot \nabla p$.

\begin{lemm}\label{lemm:Dt}
Let~$V\in C^0([0,T];C^{1+\eps}_*({\mathbf{R}}^d))$ for some~$\eps>0$ and 
consider a symbol~$p=p(t,x,\xi)$ which is homogeneous in~$\xi$ of order~$m$. 
Then there exists  $K>0$ (independent of~$p,V$) such that for any~$t \in [0,T]$ 
and any~$u\in C^0([0,T];H^{m}({\mathbf{R}}^d))$,
\begin{multline}\label{comm:Dt}
\lA \bigl[ T_p , \partial_t +T_V \cdot \nabla\bigr] u(t)\rA_{L^2({\mathbf{R}}^d)}\\
\le K  \left\{\mathcal{M}^m_0 (p)\lA V(t)\rA_{C^{1+\eps}_*}
+\mathcal{M}^m_0 (\partial_t p+V\cdot \nabla p)\right\}\lA u(t)\rA_{H^{m}({\mathbf{R}}^d)}.
\end{multline}
\end{lemm}
\begin{proof}
Set~$I=[0,T]$ and denote by~$\mR$ the set of 
continuous operators~$R(t)$ from $H^m({\mathbf{R}}^d)$ to 
$L^2({\mathbf{R}}^d)$ with norm satisfying 
$$
\lA R(t) \rA_{\mathcal{L} (H^{m} (\xR^d), L^2({\mathbf{R}}^d))}
\leq K  \left\{\mathcal{M}^m_0 (p)\lA V(t)\rA_{C^{1+\eps}_*}
+\mathcal{M}^m_0 (\partial_t p+V\cdot \nabla p)  \right\}.
$$
We begin by noticing that it is sufficient to prove that 
\begin{equation}\label{6.13}
\bigl(\partial_t +V \cdot \nabla\bigr) T_p =T_p \bigl(\partial_t +T_V \cdot \nabla\bigr) 
+R,\quad R\in \mR.
\end{equation}
Indeed, by Theorem~$5.2.9$ in \cite{MePise}, we have (for fixed $t$)
$$
\lA ( V-T_V) \cdot \nabla T_p u\rA_{L^2}\les \lA V\rA_{W^{1,\infty}} \lA T_p u\rA_{L^2}\les  \lA V\rA_{W^{1,\infty}} \mathcal{M}_0^m(p) \lA u\rA_{H^m}
$$
  by using the operator norm estimate \eqref{esti:quant1}.
This implies that~$\bigl( V-T_V\bigr) \cdot \nabla T_p \in \mR$.

We split the proof of \eqref{6.13} into three steps. 
By decomposing~$p$ into a sum of spherical harmonics, we shall reduce the analysis 
to establishing \eqref{6.13} for the special case when~$T_p$ is a paraproduct. 
In the first step we prove \eqref{6.13} for~$m=0$ and~$p=p(t,x)$. In the second step we 
prove \eqref{6.13} for~$p=a(t,x)h(\xi)$ where~$h$ is homogeneous in~$\xi$ of order~$m$. 
Then we consider the general case.
\subsubsection*{Step 1: Paraproduct,~$m=0$,~$p=p(t,x)$.} In this case 
 $ \mathcal{M}^0_0 (p) = \| p\|_{L^\infty}.$  
We have
\begin{equation}\label{6.14}
\left\{
\begin{aligned}
&\partial_t T_p u= T_{\partial_t p} u +T_p \partial_t u,\\[0.5ex]
&V\cdot \nabla T_p u =V\cdot T_{\nabla p} u + V T_p \cdot \nabla u=:A+B.
\end{aligned}
\right.
\end{equation}

Decompose~$V=S_{j-3}(V)+S^{j-3}(V)$, with
$$
S_{j-3}(V)=\sum_{k\le j-3} \Delta_k V, \qquad S^{j-3}(V)=\sum_{k\ge j-2} \Delta_k V.
$$
With the choice of cut-off function $\chi$ made in \S\ref{s2}, given two functions 
$u$ and $a$ we have $T_a u=\sum_j S_{j-3}(a)\Delta_j u$ and hence
\begin{equation}\label{6.15}
\left\{
\begin{aligned}
&A=A_1+A_2,\\
&A_1\defn \sum_j S_{j-3}(V)S_{j-3}(\nabla p)\Delta_j u,\qquad A_2\defn \sum_j S^{j-3}(V)S_{j-3}(\nabla p)\Delta_j u.
\end{aligned}
\right.
\end{equation}
Let us consider the term~$A_2$. Since 
$$
\lA S^{j-3}(V)\rA_{L^\infty}\le \sum_{k\ge j-2} \lA \Delta_k V\rA_{L^\infty}
\les \sum_{k\ge j-2} 2^{-k(1+\eps)} \VCe \les 2^{-j(1+\eps)} \VCe
$$
and~$\lA S_{j-3}(\nabla p)\rA_{L^\infty}\les 2^j \pO$, 
we obtain
\begin{equation}\label{6.16}
\lA A_2 \rA_{L^2} \les \sum_j 2^{-j\eps} \lA V\rA_{C^{1+\eps}_*} 
\lA p\rA_{L^\infty} \lA u\rA_{L^2} \les \mathcal{M}^0_0 (p)\lA V\rA_{C^{1+\eps}_*} \lA u\rA_{L^2}.
\end{equation}

We now estimate~$A_1= A_{11}+A_{12}$, with 
\begin{equation}\label{6.17}
\left\{
\begin{aligned}
&A_{11}\defn \sum_j S_{j-3}\bigl\{ S_{j-3}(V) \cdot \nabla p\bigr\} \Delta_j u,\\
&A_{12}\defn \sum_j \bigl\{\bigl[ S^{j-3}(V),S_{j-3}\bigr] \nabla p \bigr\}\Delta_j u.
\end{aligned}
\right.
\end{equation}
Write~$S_{j-3}(V)=V-S^{j-3}(V)$, to obtain
\begin{equation*}
A_{11} = \sum_j S_{j-3}\bigl( V \cdot \nabla p\bigr) \Delta_j u 
- \sum_j S_{j-3}\bigl\{ S^{j-3}(V) \cdot \nabla p\bigr\} \Delta_j u  =T_{V\cdot \nabla p} u +I+II
\end{equation*}
where 
$$
I=-\sum_j \bigl(  \nabla \cdot S_{j-3}\bigl\{ S^{j-3}(V)  p\bigr\}\bigr) \Delta_j u,\quad 
II=\sum_j S_{j-3}\bigl\{ S^{j-3}(\nabla \cdot V) p\bigr\} \Delta_j u.
$$
Then
\begin{equation*}
\begin{aligned}
\lA I\rA_{L^2} &\les \sum_j 2^j \lA S^{j-3}(V) p\rA_{L^\infty}\lA \Delta_j u\rA_{L^2}\\
 &\les \sum_j 2^j 2^{-j(1+\eps)} \VCe \pO \lA u\rA_{L^2} 
 \les \VCe \pO \lA u\rA_{L^2}.
 \end{aligned}
\end{equation*}
Moreover, 
\begin{equation*}
\lA II\rA_{L^2} \les \sum_j \lA S^{j-3}(\nabla V)\rA_{L^\infty} \pO \lA \Delta_j u\rA_{L^2} 
 \les \VCe \pO \lA u\rA_{L^2}.
\end{equation*}
Therefore
\begin{equation}\label{6.18}
A_{11}=T_{V\cdot \nabla p} u +R u,\quad R\in \mR.
\end{equation}
 In order to estimate~$A_{12}$ we note that one can write
 $$ \bigl[ S^{j-3}(V),S_{j-3}\bigr] \nabla p =  \bigl[ S^{j-3}(V),S_{j-3} \nabla \bigr]  p + S_{j-3}\big( S^{j-3}(\nabla V)p\big).$$
 Since $S_{j-3}\partial_k = 2^j\psi_1(2^{j-3}D)$ where $\psi_1 \in C_0^\infty(\xR^d)$ we have
 $$ \Vert  \bigl[ S^{j-3}(V),S_{j-3} \nabla \bigr]  p \Vert_{L^\infty} \leq C 2^{-j\eps}\Vert   V \Vert_{C^{1+\eps}_*} \Vert p \Vert_{L^\infty}.$$
 Moreover we have
 $$ \Vert S_{j-3}\big( S^{j-3}(\nabla V)p\big)\Vert_{L^\infty} \leq C 2^{-j\eps}\Vert  V \Vert_{ C^{1+\eps}_*}  \Vert p \Vert_{L^\infty}.$$

 It follows that~$A_{12}=R u$ with~$R\in \mR$. 
Consequently, we deduce from \eqref{6.17} and \eqref{6.18} that $A_{1}=T_{V\cdot \nabla p} u +R u$ for some $R$ in $\mR$. 
It thus follows from \eqref{6.15} and \eqref{6.16} that
$A=T_{V\cdot \nabla p} u +R u$ with $R\in \mR$.

We estimate now the term~$B$ introduced in \eqref{6.14}. We split this term as follows:
\begin{align*}
B&=V \cdot (T_p \nabla u) = V\cdot  \sum_j  S_{j-3}(p)\nabla \Delta_j u \\ 
&= \sum_j S_{j-3}(V)S_{j-3}( p)\Delta_j \nabla u+\sum_j S^{j-3}(V)S_{j-3}(p)\Delta_j \nabla u
=: B_1+B_2.
\end{align*}
We have
\begin{align*}
\lA B_2\rA_{L^2} &\le \sum_j \lA S^{j-3}(V)\rA_{L^\infty}  \lA S_{j-3}( p)\rA_{L^\infty}  \lA \Delta_j \nabla u\rA_{L^2}\\
&\les \sum_{j} 2^{-j(1+\eps)} \VCe 2^j \pO \lA u\rA_{L^2}.
\end{align*}
and hence~$B_2=R u$ with~$R\in \mR$. To deal with the term~$B_1$, let us introduce
\begin{equation}\label{6.25}
C\defn T_p T_V \cdot \nabla u =\sum_{j}S_{j-3}(p)\Delta_j \sum_k S_{k-3}(V)\cdot \nabla \Delta_k u .
\end{equation}
Since the spectrum of~$S_{k-3}(V)\cdot \nabla \Delta_k u$ is contained in  
$\{ (3/8) 2^k\le \la \xi\ra\le (2+1/8)2^k\}$, the term 
$\Delta_j (S_{k-3}(V)\cdot \nabla \Delta_k u)$ vanishes unless~$\la k-j\ra\le 3$. On the other hand, 
for~$\la k-j\ra\le 3$,~$S_{k-3}(V)-S_{j-3}(V)=\pm \sum_{\ell=-3}^{3}\Delta_{\ell+j} V$, and hence we can write 
$C$ under the form
$$
C=C_1+C_2=C_1+\sum_{j}S_{j-3}(p)\Delta_j \bigl\{ S_{j-3}(V) \cdot \sum_{\la k-j\ra\le 3}  \nabla \Delta_k u\bigr\}
$$
where~$C_1$ is given by 
$$
C_1=\sum_j  S_{j-3}(p) \Delta_j \sum_{i=1}^3 \sum_{\ell=-1}^{i-2} 
\bigl\{ \Delta_{\ell+j}(V) \nabla \Delta_{i+j}(u)-\Delta_{\ell+j-i}(V)\nabla\Delta_{j-i}(u)\bigr\},
$$
so that 
$$
\lA C_1\rA_{L^2} \les \sum_j \pO 2^{-j(1+\eps)}\VCe 2^j \lA u\rA_{L^2},
$$
which implies that~$C_1=Ru$ with~$R\in \mR$. To estimate~$C_2$, as before we write 
$C_2=C_{21}+C_{22}$ where
\begin{align*}
C_{21}&\defn \sum_{j}S_{j-3}(p)\bigl[ \Delta_j ,  S_{j-3}(V)\bigr]  \cdot \sum_{\la k-j\ra\le 3}  \nabla \Delta_k u, ~ \\
C_{22}&\defn \sum_{j}S_{j-3}(p)S_{j-3}(V) \cdot   \Delta_j \sum_{\la k-j\ra\le 3}  \nabla \Delta_k u,
\end{align*}
where (using frequency localization in dyadic annuli and Plancherel formula) 
$$
\lA C_{21}\rA_{L^2}^2 
\les \sum_j \pO^2 2^{-2j}\lA V\rA_{W^{1,\infty}}^2 2^{2j} \sum_{\la k-j\ra\le 3}\lA \Delta_k u\rA_{L^2}^2
\les \pO^2 \VCe^2 \lA u\rA^2_{L^2}.
$$
On the other hand, since~$\Delta_j \sum_{\la k-j\ra\le 3}\Delta_k =\Delta_j$, we have 
$$
C_{22}= \sum_{j}S_{j-3}(V)S_{j-3}(p)\nabla \Delta_j u= B_1.
$$
We thus end up with
\begin{equation}\label{6.30}
B=T_p T_V \cdot \nabla u +R u,\quad R\in \mR.
\end{equation}

It follows from \eqref{6.14} and \eqref{6.30} that
\begin{equation}\label{6.31}
(\partial_t +V\cdot \nabla )T_p u=T_p (\partial_t+T_V \cdot \nabla) u +T_{\partial_t p+V\cdot \nabla p} u+R u,\quad R\in \mR.
\end{equation}
The symbolic calculus shows that~$T_{\partial_t p+V\cdot \nabla p}\in \mR$, 
which proves \eqref{6.13} and concludes the proof of the first step.
\subsubsection*{Step 2 : Higher order paraproducts.} 
We now assume that~$p(t,x,\xi)=a(t,x)h(\xi)$ where~$h(\xi)=\la \xi\ra^m \widetilde{h} (\frac{\xi}{\vert \xi \vert})$ with 
$\widetilde{h} \in C^\infty(\xS^{d-1})$. Then, directly from the definition~\eqref{eq.para}, we have 
$T_p=T_a \psi(D_x) h(D_x)$ where~$\psi$ satisfies \eqref{cond.psi}. We have
$$
\bigl[ T_p , \partial_t +T_V \cdot \nabla\bigr] = \bigl[ T_a , \partial_t +T_V \cdot \nabla\bigr] \psi(D_x) h(D_x)+ T_a \bigl[\psi(D_x) h(D_x) ,  T_V\bigr] \cdot \nabla.
$$
We claim that the norm from~$H^m$ to~$L^2$ of the operator in the left hand side is bounded by
\begin{equation}\label{est:comm}
C \big( \|a\|_{L^\infty} \lA V\rA_{C^{1+\eps}_*} + \|\partial_t a+V\cdot \nabla a\|_{L^\infty} \big) \| \widetilde{h} \|_{H^{d+2}(\xS^{d-1})}.
\end{equation}
To obtain this claim, notice that 
the norm of the first term in the right-hand side is estimated 
by means of the previous step by 
$$
C  (\|a\|_{L^\infty} \lA V\rA_{C^{1+\eps}_*}
+\|\partial_t a+V\cdot \nabla a\|_{L^\infty})\Vert \widetilde{h} \Vert_{L^\infty(\xS^{d-1})}.
$$
To estimate the second term in the right-hand side we use a sharp version of the symbolic calculus 
estimate \eqref{esti:quant2}, see \cite[Theorem 6.1.4]{MePise}, which implies that 
the norm of the second term is bounded by $ C \|a\|_{L^\infty}\lA \nabla V\rA_{L^\infty}
  \| \widetilde{h} \|_{H^{d+2}(\xS^{d-1})}$  . 
The sharp version of symbolic result alluded to above asserts that \eqref{esti:quant2} holds 
with the semi-norm $M_\rho^m$ (cf.\ \eqref{defi:norms}) 
replaced by 
other semi-norms where the supremum on the multi-index $\alpha$ is taken for $\la \alpha\ra$ smaller 
than $d/2$ plus an explicit number. In particular we have the following lemma which suffices to 
complete the 
proof of the claim \eqref{est:comm}.

\begin{lemm}\label{commutateur}
Set  $b(\xi) = \vert \xi \vert^m  \psi(\xi) \widetilde{h}\big(\frac{\xi}{\vert \xi \vert}\big).$  
There exists $C>0$ independent of $V$ and $ \widetilde{h}$ such that  
 $$\lA [b(D), T_V]u \rA_{L^2(\xR^d)} \leq C \Vert \nabla V \Vert_{L^\infty(\xR^d)} \Vert \widetilde{h} \Vert_{H^{d+2}(\xS^{d-1})}Ê\Vert u \Vert_{H^{m-1}(\xR^d)} $$
 for every $u \in H^{m-1}(\xR^d)$.
\end{lemm}
  \subsubsection*{Step 3 : Paradifferential operators.} 
Consider an orthonormal basis $(\widetilde{h} _\nu)_{\nu\in \xN^*}$ 
of~$L^2(\xS^{d-1})$ consisting of eigenfunctions of the (self-adjoint) 
Laplace--Beltrami operator, 
$\Delta_\omega=\Delta_{\xS^{d-1}}$ on~$L^2(\xS^{d-1})$, {\em i.e.\/}~$\Delta_\omega \widetilde{h} _\nu 
=\lambda_{\nu}^2\widetilde{h}_\nu$. By the Weyl formula, we know that~$\lambda_\nu\sim c \nu^{\frac 1 d}$. 
Setting~$h_\nu(\xi)= \la \xi\ra^m \widetilde{h} _\nu (\omega)$, 
$\omega=\xi/\la\xi\ra$,~$\xi\neq 0$, we can write
$$
p(t,x,\xi)=\sum_{\nu\in\xN^*} a_\nu(t,x) h_\nu(\xi)\quad\text{where}\quad 
a_\nu(t,x)=\int_{\xS^{d-1}} p(t,x,\omega) \overline{\widetilde{h}_\nu(\omega)}\, d\omega.
$$
Since
\begin{align*}
  \lambda_\nu^{2k} a_\nu (t,x)&=\int_{\xS^{d-1}} \Delta_w^k p (t,x,\omega) \overline{\widetilde{h}_\nu(\omega)}\, d\omega,\\
\lambda_\nu^{2k}( \partial_t  +V\cdot \nabla)a_\nu (t,x)
&=\int_{\xS^{d-1}} \Delta_w^k (\partial_t  +V\cdot \nabla) p (t,x,\omega) \overline{\widetilde{h}_\nu(\omega)}\, d\omega
 \end{align*}
taking $ k = d+2 $ we deduce
\begin{equation}\label{6.33}
\begin{aligned}
\sup_{t\in I} \lA a_\nu(t, \cdot)\rA_{L^\infty} &\le C \lambda_\nu  ^{-2(d+2)} \mathcal{M}^m_0(p)\\
\sup_{t\in I} \lA ( \partial_t +V\cdot \nabla)a_\nu(t, \cdot)\rA_{L^\infty} &\le C \lambda_\nu  ^{-2(d+2)} \mathcal{M}^m_0 (\partial_t p +V\cdot \nabla  p).
 \end{aligned}
  \end{equation}
Moreover,  there exists a positive constant~$K$ such that, for all~$\nu\ge 1$, 
\begin{equation}\label{6.34}
\big\Vert \widetilde{h}_\nu\big\Vert_{H^{d+2}(\xS^{d-1})} \leq C\lambda_{\nu}^{d+2}.
\end{equation}
Now we can write
$$
\lA \left[ \partial_t+T_V\cdot\partialx , T_p\right] u\rA_{L^2}\le \sum_{\nu\in\xN^*} 
\lA \left[ \partial_t+T_V\cdot\partialx , T_{a_\nu h_\nu}\right] u\rA_{L^2}.
$$
So using  \eqref{est:comm} for every~$\nu\ge 1$ 
and the estimates \eqref{6.33}--\eqref{6.34}, we obtain \eqref{6.13}, since the sum 
$$
\sum_{\nu} \lambda_\nu^{d+2} 
\lambda_\nu^{- 2(d+2)} \sim  \sum_{\nu} \nu^{- 1- \frac{2}{ d}}
$$
is finite.
This completes the proof of the lemma.
\end{proof}

We have also a Sobolev analogue of Lemma \ref{lemm:Dt} which can be proved similarly.
\begin{lemm}\label{lemm:Dt1}
Let~$s> 1 + d/2$ and~$V\in C^0([0,T];H^s({\mathbf{R}}^d))$. 
There exists a positive constant~$K$ such that for any 
symbol~$p=p(t,x,\xi)$ which is homogeneous in~$\xi$ 
of order~$m\in \xR$ and all~$u\in C^0([0,T];H^{s+m}({\mathbf{R}}^d))$,
\begin{multline*}
\lA \bigl[ T_p , \partial_t +T_V \cdot \nabla\bigr] u (t) \rA_{H^s({\mathbf{R}}^d)}\\
\le K  \left\{\mathcal{M}^m_0 (p)\lA V(t)\rA_{H^s}
+\mathcal{M}^m_0 (\partial_t p+V\cdot \nabla p) \right\}   \lA u(t) \rA_{H^{s+m}({\mathbf{R}}^d)}.
\end{multline*}
\end{lemm}

\subsection{Parabolic evolution equation}\label{s:pe}
Consider the evolution equation 
$$
\partial_z w +\la D_x \ra w=0,
$$
where~$z\in\xR$ and~$x\in {\mathbf{R}}^d$. 
By using the Fourier transform, one easily checks that
\begin{equation}\label{w-para}
\sup_{z\in [0,1]} \lA w(z)\rA_{H^r}+
\bigl(\int_0^1 \lA w(z)\rA_{H^{r+\mez}}^2\, dz\bigr)^{\mez}
\le K \lA w(0)\rA_{H^r}.
\end{equation}
The purpose of this section is to prove similar results when the constant coefficient operator~$\la D_x\ra$ is replaced by an elliptic paradifferential operator. 

Given~$I\subset \xR$,~$z_0\in I$ and a function~$\varphi=\varphi(x,z)$ defined on~${\mathbf{R}}^d\times I$, 
we denote by~$\varphi(z_0)$ the function~$x\mapsto \varphi(x,z_0)$. 
For~$I\subset \xR$ and a normed space~$E$,~$\varphi\in C^0_z(I;E)$ 
means that~$z\mapsto \varphi(z)$ is a continuous function from~$I$ to~$E$. 
Similarly, for~$1\le p\le +\infty$, 
$\varphi\in L^p_z(I;E)$ means that 
$z\mapsto\lA \varphi(z)\rA_{E}$ belongs to the Lebesgue space~$L^p(I)$. 

In this section, when~$a$ and~$u$ are
symbols and functions depending on~$z$, we still denote by~$T_a u$ the function defined by 
$(T_a u)(z)=T_{a(z)}u(z)$ where~$z\in I$ is seen as a parameter. We denote by~$\Gamma^m_\rho({\mathbf{R}}^d\times I)$ 
the space of symbols~$a=a(z;x,\xi)$ such that 
$z\mapsto a(z;\cdot)$ is bounded from~$I$ 
into~$\Gamma^m_\rho({\mathbf{R}}^d)$ (see Definition~\ref{defiGmrho}), 
with the semi-norm
\begin{equation}\label{seminorm-bis}
\mathcal{M}^m_\rho(a)=\sup_{z\in I}
\sup_{\la\alpha\ra\le 2(d+2) +\rho  ~}\sup_{\la\xi\ra \ge 1/2~}
\lA (1+\la\xi\ra)^{\la\alpha\ra-m}\partial_\xi^\alpha a(z;\cdot,\xi)\rA_{W^{\rho,\infty}({\mathbf{R}}^d)}.
\end{equation}

Given~$\mu\in\xR$ we define the spaces
\begin{equation}\label{XY}\begin{aligned}
X^\mu(I)&=C^0_z(I;H^{\mu}({\mathbf{R}}^d))\cap L^2_z(I;H^{\mu+\mez}({\mathbf{R}}^d)),\\
Y^\mu(I)&=L^1_z(I;H^{\mu}({\mathbf{R}}^d))+L^2_z(I;H^{\mu-\mez}({\mathbf{R}}^d)).
\end{aligned}
\end{equation}

\begin{prop}\label{prop:max}
Let~$r\in \xR$,~$\rho\in (0,1)$,~$J=[z_0,z_1]\subset\xR$ and let 
$p\in \Gamma^{1}_{\rho}({\mathbf{R}}^d\times J)$ 
satisfying$$
\RE p(z;x,\xi) \geq c \la\xi\ra,
$$
for some positive constant~$c$. Then for any ~$f\in Y^r(J)$ 
and~$w_0\in H^{r}({\mathbf{R}}^d)$, there exists~$w\in X^{r}(J)$ solution of  the parabolic evolution equation
\begin{equation}\label{eqW}
\partial_z w + T_p w =f,\quad w\arrowvert_{z=z_0}=w_0,
\end{equation}
satisfying 
\begin{equation*}
\lA w \rA_{X^{r}(J)}\le K\left\{\lA w_0\rA_{H^{r}}+ \lA f\rA_{Y^{r}(J)}\right\},
\end{equation*}
for some positive constant~$K$ depending only on~$r,\rho,c$ and~$\mathcal{M}^1_\rho(p)$.
 Furthermore, this solution is unique in ~$X^s(J)$ for any~$s\in \xR$. 
 \end{prop}

\begin{proof}
Let~$r\in \xR$. Denote by~$\langle \cdot ,\cdot\rangle_{H^r}$ the scalar product in~$H^r({\mathbf{R}}^d)$ and 
chose~$F_1$ and~$F_2$ such that~$f=F_1+F_2$ with
$$
\lA F_1\rA_{L^1_z(J;H^r)}+ \lA F_2\rA_{L^2(J;H^{r-\mez})}
\leq \lA f\rA_{Y^r(J)} + \delta,\qquad \delta >0.
$$
Let us consider for~$\eps>0$ the equation 
\begin{equation}\label{eqWe}
\partial_z w_\eps + \eps ( - \Delta + \text{Id}) w_\eps+ T_p w_\eps =f,\quad w_\eps \arrowvert_{z=z_0}=w_0.
\end{equation}
Then standard methods in parabolic equations show that for any~$z_1>z_0$, this equation have a unique solution in 
$$C^0([z_0, z_1]; H^r(\xR^d)) \cap L^2((z_0, z_1); H^{r+2}(\xR^d))$$
(here we only used that~$T_p$ is a Sobolev first order operator).
To  let $\eps$ go to zero we need to establish uniform estimates 
with respect to~$\eps$.

Taking the scalar product in~$H^r$, directly from~\eqref{eqWe}, we obtain
\begin{multline*}
\frac{1}{2} \frac{d}{d z} \lA w_\eps(z)\rA_{H^r}^2 + 
\eps \langle (- \Delta + \text{Id}) w_\eps(z) ,w_\eps(z)\rangle_{H^r} 
+  \RE \langle T_{p(z)} w_\eps(z) , w_\eps(z) \rangle_{H^r}
\\  \le 
\lA F_1(z)\rA_{H^r}\lA w_\eps(z)\rA_{H^r}+
\lA F_2(z)\rA_{H^{r-\mez}}\lA w_\eps(z) \rA_{H^{r+\mez}}.
\end{multline*}
It follows from G\aa rding's inequality (see~\cite[Section 6.3.2]{MePise}) 
that there exist two constants~$C_1,C_2>0$ depending only on~$\mathcal{M}^1_\rho(p)$ such that
 for any~$u \in H^r$, 
~$$
\RE\langle T_{p(z)} u(z) , u(z) \rangle_{H^r} \ge C_1 \lA u(z)\rA_{H^{r+\mez}}^2 - C_2\lA u(z)\rA_{H^{r+\frac{ 1- \rho} 2}}^2,
$$
for each fixed~$z\in J$. Therefore, we obtain
\begin{multline*}
\mez \frac{d}{d z} \lA w_\eps(z)\rA_{H^r}^2 
+  \eps \langle (- \Delta + \text{Id}) w_\eps(z) ,w_\eps(z)\rangle_{H^r}  
+C_1 \lA w_\eps(z)\rA_{H^{r+\mez}}^2 \\
\le \lA F_1(z)\rA_{H^r}\lA w_\eps(z)\rA_{H^r}+ 
 \lA F_2(z)\rA_{H^{r-\mez}}\lA w_\eps(z) \rA_{H^{r+\mez}} + C_2 \lA w_\eps(z)\rA_{H^{r+ \frac{ 1- \rho} 2}}^2.
\end{multline*}
Integrating in~$z$ we obtain that, for all~$z\in [z_0,z_1]$,
\begin{multline*}
A(z)\defn\mez\left\{\lA w_\eps(z)\rA_{H^r}^2-\lA w_\eps(z_0)\rA_{H^r}^2\right\} 
+  \eps\int_{z_0}^z\|w_\eps(z')\|_{H^{r+1}}^2 dz'\\
  +C_1\int_{z_0}^z \lA w_\eps(z')\rA_{H^{r+\mez}}^2 dz' 
\end{multline*}
is bounded by
\begin{multline*}
B\defn\lA F_1\rA_{L^1(J;H^r)}\lA w_\eps\rA_{L^\infty(J;H^r)}+ 
\lA F_2\rA_{L^2(J;H^{r-\mez})} \lA w_\eps\rA_{L^2(J;H^{r+\mez})}\\
+C_2\lA w_\eps\rA_{L^2(J;H^{r+\frac {1- \rho} 2})}^2.
\end{multline*}
By standard arguments, it follows that
\begin{multline}\label{eq.2.45}
\lA w_\eps\rA_{L^\infty(J; H^r)}^2 +C_1\lA w_\eps\rA_{L^2(J;H^{r+\frac{ 1} 2})}^2 
\leq \lA w_0\rA_{H^r}^2\\
+ C \bigl( \lA F_1\rA_{L^1(J;H^r)}^2 +\lA F_2\rA_{L^2(J;H^{r-\mez})}^2
+ \lA w_\eps\rA_{L^2(J;H^{r+\frac{ 1- \rho} 2})}^2\bigr).
\end{multline}
Finally, to eliminate the last term 
in the right hand side of ~\eqref{eq.2.45}, one notices 
that the left hand side controls 
by interpolation~$c \lA w_\eps\rA_{L^p(J;H^{r+\frac{ 1- \rho} 2})}^2$, 
for some~$p>2$, hence by H\"older in the~$z$ variable, 
there exists~$\kappa >0$ (depending only on~$p$) such that if ~$|z_0-z_1|\leq \kappa$,  we have 
$$
C \lA w_\eps\rA_{L^2(J;H^{r+\frac{ 1- \rho} 2})}^2
\leq \frac 1 2 \bigl( \lA w_\eps\rA_{L^\infty(J; H^r)}^2 +C_1\lA w_\eps\rA_{L^2(J;H^{r+\frac{ 1} 2})}^2\bigr).
$$
We consequently obtain 
$$
\lA w_\eps\rA_{L^\infty(J; H^r)}^2 +C_1\lA w_\eps\rA_{L^2(J;H^{r+\frac{ 1} 2})}^2 
\leq  2\lA w_0\rA_{H^r}^2+ C \bigl( \lA F_1\rA_{L^1(J;H^r)}^2 +\lA F_2\rA_{L^2(J;H^{r-\mez})}^2\bigr).
$$
We can now iterate the estimate between~$z_0+\kappa$ and~$z_0+ 2 \kappa$, ... 
to get rid of the assumption~$|z_1 - z_0| \leq \kappa~$ 
(and of course the constants will depend on~$z_1$). 

By using the equation, we obtain now that $(w_\eps)$ is bounded in 
$X^r(J)\cap C^1(J; H^{r-2})$. 
It follows from the Banach--Alaoglu theorem that, up to a subsequence, 
$(w_\eps)$ converges in the sense of distributions to $w\in L^\infty(J,H^r) \cap L^2(J,H^{r+\mez})$, which satisfies 
the equation $\partial_z w+T_p w =f$. Then, using the equation, we obtain 
$\partial_z w \in Y^r(J)$ which implies that 
$w\in C^0([z_0,z_1];H^r(\xR^d))$ thanks to the Lemma \ref{lem.3.11} below.
Moreover, by the Ascoli theorem, up to a subsequence, $(w_\eps)$ converges in  
$C^0([z_0,z_1]; H^{r-\mu}_{loc})$ for some $\mu>0$. Since $w_\eps\arrowvert_{z=0}=w_0$ we obtain that 
$w\arrowvert_{z=0}=w_0$, which completes the existence part . 
The proof of uniqueness follows the same steps and we omit it.
\end{proof}
We recall now the following classical   lemma (see \cite{Lions} th.3.1). 
 \begin{lemm}\label{lem.3.11}
Let $I = (-1,0)$ and $s \in \xR$. Let   $u \in L_z^2(I, H^{s+ \mez}(\xR^d))$ such that 
$\partial_z u \in L_z^2(I, H^{s-\mez}(\xR^d)).$ Then $u \in C^0([-1,0], H^{s}(\xR^d))$ 
and there exists an absolute constant $C>0$ such that
$$
\sup_{z \in [-1,0]} \| u(z, \cdot) \|_{ H^{s}({\mathbf{R}}^d))} 
\leq C \bigl(\|u\|_{L^2(I,H^{s+ \mez}(\xR^d))}+  \|\partial_z u\|_{L^2(I,H^{s- \mez}(\xR^d))} \bigr).
$$
\end{lemm}

 \section{The Dirichlet-Neumann operator}\label{sec.3}

We shall prove some  elliptic regularity results using a paradifferential approach.

\subsection{Definition and continuity}\label{S:defiDN}
We begin by recalling  from \cite{ABZ1} the definition of the Dirichlet--Neumann
operator under 
general assumptions on the bottom. 
One of the novelties with respect to our previous work is that we clarify the regularity 
assumptions: assuming only that~$\eta \in W^{1, \infty}({\mathbf{R}}^d)$  
and~$ \psi \in H^{\mez}({\mathbf{R}}^d)$, we show how to define~$G(\eta) \psi$ and 
prove that the map 
$$\psi \in H^\mez({\mathbf{R}}^d) \mapsto G(\eta) \psi \in H^{-\mez} ( {\mathbf{R}}^d)$$ is continuous. 
Our second contribution is to 
prove that the map~$\eta \mapsto G(\eta)$ is Lipschitz (in a proper topology).

The goal is to study the boundary value problem
\begin{equation}\label{v}
\Delta_{x,y}\phi=0\text{ in }\Omega,\quad
\phi\arrowvert_{\Sigma}=f,\quad \partial_n \phi\arrowvert_{\Gamma}=0.
\end{equation}
See \S\ref{assu:domain} for the definitions of $\Omega,\Sigma,\Gamma$. 
Since we make no assumption on~$\Gamma$, the definition of~$\phi$ requires some care.
We recall here the definition of~$\phi$ as given in~\cite{ABZ1}.

\begin{nota}\label{notaD}
Denote by~$\mathscr{D}$ the space of functions~$u\in C^{\infty}(\Omega)$ such that
$\nabla_{x,y}u\in L^2(\Omega)$. We then define~$\mathscr{D}_0$ 
as the subspace of functions~$u\in \mathscr{D}$ such that~$u$ is
equal to~$0$ in a neighborhood of  the top boundary~$\Sigma$.
\end{nota}

\begin{prop}[\protect{\cite[Proposition 2.2]{ABZ1}}]\label{corog}
There exists a positive weight~$g\in L^\infty_{loc} ( \Omega)$, equal to~$1$
near the top boundary of~$\Omega$ and a constant $C>0$ 
such that for all~$u\in \mathscr{D}_0$,
\begin{equation}\label{eq.Poinc}
\int_\Omega g(x,y) |u(x,y)|^2 \,dx dy \leq  C \int_{\Omega} |\partialyx u (x,y)|^2 \,dx dy.
\end{equation}
\end{prop}
\begin{defi}\label{defi:H10}
Denote by~$H^{1,0}( \Omega)$ the space of functions~$u$ on~$\Omega$
such that there exists a sequence~$(u_n)\in \mathscr{D}_0$ such that,
$$ \partialyx u_n  \rightarrow \partialyx u \text{ in } L^2( \Omega, dxdy),
\qquad u_n  \rightarrow  u \text{ in } L^2( \Omega, g(x,y)dxdy).$$
We endow the space~$H^{1,0}(\Omega)$ with the norm 
$\lA u\rA = \lA \partialyx u \rA_{L^2( \Omega)}$.
\end{defi}
Let us recall that the space~$H^{1,0}(\Omega)$ is a Hilbert space (see~\cite{ABZ1}).

We are able now   to define the Dirichlet-Neumann operator.
Let~$f \in H^\mez( {\mathbf{R}}^d)$. We first define an~$H^1$ lifting of~$f$ in~$\Omega.$ 
To do so let~$\chi_0 \in C^\infty(\xR)$ be such that~$\chi_0(z) = 1$ 
if~$z\geq - \mez$ and~$\chi_0(z) = 0$ if~$z \leq -1.$ We set
$$
\psi_1(x,z) = \chi_0(z) e^{z\langle D_x \rangle}f(x), \quad x \in {\mathbf{R}}^d, z\leq 0.
$$
By the usual property of the Poisson kernel we have
$$
\lA \nabla_{x,z} \psi_1 \rA_{L^2([-1,0]\times \xR^d)}\le C \lA f\rA_{H^{\mez}(\xR^d)}.
$$
Then we set
$$
\underline{\psi} (x,y) = \psi_1\big(x, \frac{y-\eta(x)}{h}\big),\quad (x,y)\in \Omega.
$$
This is well defined since~$\Omega \subset \{(x,y): y<\eta(x)\}$. 
Moreover since the bottom~$\Gamma$ is contained in~$\{(x,y): y<\eta(x)-h\}$, 
we see that~$\underline{\psi}$ vanishes identically near~$\Gamma.$ 

Now we have obviously~$\underline{\psi}\arrowvert_\Sigma = f$ and 
since~$\nabla \eta \in   L^\infty({\mathbf{R}}^d)$, 
an easy computation shows that~$\underline{\psi} \in H^1(\Omega)$ and
\begin{equation}\label{est:psi-H1}
\Vert \underline{\psi} \Vert_{H^1(\Omega)} 
\leq K(1+ \Vert \eta \Vert_{W^{1, \infty}})\Vert f\Vert_{H^\mez({\mathbf{R}}^d)}.
\end{equation}
 
Then the map
$$
v\mapsto - \int_{\Omega}  \partialyx  \underline{\psi} \cdot \partialyx v \, dx dy
$$
is a bounded linear form on~$H^{1,0}(\Omega)$. 
It follows from the Riesz theorem that there exists a unique
${u}\in H^{1,0}(\Omega)$ such that
\begin{equation}\label{eq.variabis}
\forall v \in H^{1,0} ( \Omega), \qquad \int_\Omega \partialyx  u \cdot \partialyx v \, dxdy
= - \int_{\Omega}  \partialyx  \underline{\psi} \cdot \partialyx v \, dx dy.
\end{equation}
Then~$u$ is the variational solution to the problem
$$
- \Deltayx u = \Deltayx \underline{\psi}
\quad\text{in }\mathcal{D}'(\Omega), \qquad u \mid_{\Sigma} =0,
\qquad \partial _n u \mid_{\Gamma} = 0,
$$
the latter condition being justified as soon as the bottom~$\Gamma$ is regular enough.

\begin{lemm}
The function~$\phi = u + \underline{\psi}$ constructed 
by this procedure is  independent on the choice
of the lifting function~$\underline{\psi}$ as long as it remains 
bounded in~$H^1(\Omega)$ and vanishes near the bottom.
\end{lemm}
\begin{proof}Consider two functions constructed 
by this procedure,~$\phi_k= u_k + \underline{\psi_k}$,~$k=1,2$. 
Then, by standard density arguments, since~$\underline{\psi_1}-\underline{\psi_2}$ 
vanishes at the top boundary~$\Sigma$ and in a 
neighborhood of the bottom~$\Gamma$, there exists a 
sequence of functions~$\psi_n \in C^\infty_0( \Omega)$ 
supported in a fixed 
Lipschitz domain~$\widetilde{\Omega} \subset \Omega$ tending 
to~$\underline{\psi_1} - \underline{\psi_2}$ in~$H^1_0 (\widetilde{\Omega})$ 
and hence also in~$H^{1,0} (\Omega)$. 
As a consequence,~$\underline{\psi_1} - \underline{\psi_2}\in H^{1,0} ( \Omega)$ 
and the function~$\phi= \phi_1- \phi_2$ is the unique (trivial) solution in~$H^{1,0}(\Omega)$ 
of the equation 
$\Delta_{x,y} \phi =0$ given by the Riesz Theorem.
\end{proof}

\begin{defi}\label{defi:phi}
We shall say that the function~$\phi = u + \underline{\psi}$ constructed by the above 
procedure is the variational solution of~\eqref{v}. It satisfies 
\begin{equation}\label{est:phi}
\int_\Omega \la \nabla_{x,y}\phi\ra^2\, dx dy \le K \lA f\rA_{H^{\mez}({\mathbf{R}}^d)}^2,
\end{equation}
for some constant~$K$ depending only on the Lipschitz norm of~$\eta$. 
\end{defi}
Formally the Dirichlet-Neumann operator is defined by
\begin{equation}\label{def:DN}
  G(\eta)\psi  
 = \sqrt{1+|\nabla \eta|^2}\, \partial_n\phi \big\arrowvert_{y=\eta(x)}
=\big[ \partial_y \phi-\partialx \eta\cdot \partialx \phi \big] \, \big\arrowvert _{y=\eta(x)} .
\end{equation}

\subsubsection{Straightening the free boundary}\label{s:flattening} 

In what follows we shall set 
\begin{equation}\label{lesomega}
\left\{
\begin{aligned}
\Omega_1 &= \{(x,y): x \in {\mathbf{R}}^d, \eta(x)-h<y<\eta(x)\},\\
\Omega_2 &= \{(x,y)\in \mathcal{O}: y\leq\eta(x)-h\},\\
\Omega &= \Omega_1 \cup \Omega_2,
\end{aligned}
\right.
 \end{equation}
where $\mathcal{O}$ has been defined in section \ref{assu:domain}, and 
\begin{equation}\label{omega1}
\left\{
\begin{aligned}
&\widetilde{\Omega}_1= \{(x,z): x \in {\mathbf{R}}^d, z \in I\}, \quad I = (-1,0),\\
&\widetilde{\Omega}_2 = \{(x,z)\in \xR^d \times (-\infty, -1]: (x,z+1+\eta(x)-h)\in \Omega_2\},\\
&\widetilde{\Omega}= \widetilde{\Omega}_1  \cup \widetilde{\Omega}_2.  
\end{aligned}
\right.
\end{equation}

Guided by Lannes~(\cite{LannesJAMS}), 
we consider the map $(x,z) \mapsto  \rho(x,z)$ from $\widetilde{\Omega}$ to $\xR$ defined as follows
\begin{equation}\label{diffeo}
\left\{
\begin{aligned}
&\rho(x,z)=  (1+z)e^{\delta z\langle D_x \rangle }\eta(x) -z \big\{e^{-(1+ z)\delta\langle D_x \rangle }\eta(x) -h\big\}\quad \text{if } (x,z)\in \widetilde{\Omega}_1,\\
&\rho(x,z) = z+1+\eta(x)-h\quad \text{if } (x,z)\in \widetilde{\Omega}_2 
\end{aligned}
\right.
\end{equation}
where $\delta>0$ will be chosen later on.
   \begin{lemm}\label{rho:diffeo}
Assume $\eta \in W^{1,\infty}(\xR^d)$.
\begin{enumerate}
\item There exists $C>0$ such that for every $(x,z) \in \widetilde{\Omega} $ we have
$$ \vert \nabla_{x } \rho (x,z)\vert  \leq C \Vert \eta \Vert_{ W^{1,\infty}(\xR^d)} .$$
\item There exists  $K>0$   such that,  if  $ \delta \Vert \eta \Vert_{ W^{1,\infty}(\xR^d)} \leq \frac{h}{2K}$
 we have
\begin{equation}\label{rhokappa}
  \quad   \text{min }\big(1,\frac{h}{2}\big) \leq \partial_z \rho(x,z) \leq \text{max }(1,  \frac{3h}{2}), \quad \forall (x,z) \in \widetilde{\Omega}.
\end{equation}
\item The map $(x,z) \mapsto (x, \rho(x,z))$ is a Lipschitz diffeomorphism from $\widetilde{\Omega}_1$ to $\Omega_1.$
\end{enumerate}
 \end{lemm}
\begin{proof}
 
 $(1)$  For any $\lambda<0$  the symbol $a(\xi) = e^{\lambda \langle \xi \rangle}$ satisfies 
the estimate $\vert \partial_\xi^\alpha a(\xi) \vert \leq C_\alpha \langle \xi \rangle^{-\vert \alpha \vert}$ 
where $C_\alpha$ is independent on $\lambda$. Therefore its 
Fourier transform in an $L^1(\xR^d)$ function  whose norm is uniformly bounded. 
This implies that $\Vert a(D)f \Vert_{L^\infty(\xR^d)} \leq M \Vert  f \Vert_{L^\infty(\xR^d)} $ 
with $M$ independent of $\lambda.$ This proves our claim.
 
$(2)$ We have in $\widetilde{\Omega}_1$
\begin{equation}\label{dec:rho}
\begin{aligned}
\partial_z \rho  &= h   + \big(e^{\delta z \langle D_x \rangle}\eta  - \eta \big)  +(1+z)\delta e^{\delta z \langle D_x \rangle}\langle D_x \rangle \eta -  \big(e^{-\delta (1+z) \langle D_x \rangle}\eta  - \eta\big)   \\
& + z \delta e^{-\delta (1+z) \langle D_x \rangle} \langle D_x \rangle \eta.
\end{aligned}
\end{equation}
  Since 
$ e^{\delta \lambda \langle D_x \rangle}\eta  - \eta
= \delta \lambda \int_0^1e^{\delta t \lambda \langle D_x \rangle}\langle D_x \rangle \eta \, dt$ 
we can write
\begin{multline*}\Vert e^{\delta z \langle D_x \rangle}\eta  - \eta \Vert_{L^\infty(\xR^d)} 
+ \delta \Vert e^{\delta z \langle D_x \rangle}
\langle D_x \rangle \eta\Vert_ {L^\infty(\xR^d)} + \Vert e^{-\delta (1+z) \langle D_x \rangle}\eta  - \eta \Vert_{L^\infty(\xR^d)}\\ 
+ \delta \Vert e^{-\delta (1+z)\langle D_x \rangle}
\langle D_x \rangle \eta\Vert_ {L^\infty(\xR^d)}
\leq K \delta \Vert \eta \Vert_{W^{1,\infty}(\xR^d)} \leq \frac{h}{2} 
\end{multline*}
which proves $(2)$. 

$(3)$ Since  $ \rho(x,0) = \eta(x), \rho(x,-1) = \eta(x)-h $ our claim follows from $(1)$.
\end{proof}

For later use 
we state the following result whose 
proof is straightforward (recall that the spaces $X^s(I)$ are defined in~\eqref{XY}).
\begin{lemm}\label{rho:diffeo-s}
Let $I=(-1,0)$. Assume $s>\mez + \frac{d}{2}$. There exists $C>0$ such that for every $\eta \in H^{s+\mez}(\xR^d)$ we have
 \begin{equation}\label{eq.rho1}
 \Vert\partial_z \rho -h \Vert_{X^{s- \mez}(I)}
 \leq C \sqrt{\delta}  \|\eta\|_{H^{s+ \mez}( {\mathbf{R}}^d)}, \quad
\|\nabla_x \rho \|_{X^{s- \mez}(I)}   
 \leq \frac{C}{\sqrt{\delta}}   \|\eta\|_{H^{s+ \mez}( {\mathbf{R}}^d)}.
  \end{equation}
  \end{lemm}

Coming back to the case where $\eta$ is a Lipschitz function, 
Lemma \ref{rho:diffeo} shows that the 
map~$(x,z) \mapsto  (x,\rho(x,z))$ is a~Lipschitz-diffeomorphism 
from~$ \widetilde{\Omega} $ to~$\Omega $. 
We denote by~$\kappa$ the inverse map of~$\rho$:
$$
(x,z)\in \widetilde{\Omega} ,\,\,\, (x,\rho(x,z)) 
= (x,y) \Leftrightarrow (x,z) = (x,\kappa(x,y)), \,\,\,(x,y) \in \Omega.
$$ 
Let $\widetilde{\phi} (x,z) = \phi (x, \rho(x,z))$.
Then we have
\begin{equation}\label{Lambda}
\left\{
\begin{aligned} 
&(\partial_y \phi)(x,\rho(x,z)) = (\Lambda_1\widetilde{\phi})(x,z), \quad (\nabla_x\phi)(x,\rho(x,z)) 
= (\Lambda_2\widetilde{\phi})(x,z),\\
&\Lambda_1 = \frac{1}{\partial_z \rho} \partial_z \quad \Lambda_2 = \nabla_x -  \frac{\nabla_x \rho}{\partial_z \rho} \partial_z.
\end{aligned}
\right.
\end{equation}

If $\phi$ is a solution of $\Delta_{x,y}\phi = 0$ in $\Omega$ then $\widetilde{\phi}$ satisfies
$$
(\Lambda_1^2 +\Lambda_2^2)\widetilde{\phi} = 0\quad \text{in }  \widetilde{\Omega}.
$$
This yields
\begin{equation}\label{dnint1a}
(a \partial_z^2  +\Delta_x   + b \cdot\nabla_x\partial_z    - c \partial_z )\widetilde{\phi}=0,
\end{equation}
where 
\begin{equation}\label{dnint1bis}
a\defn \frac{1+|\nabla_x  \rho |^2}{(\partial_z\rho)^2},\quad b\defn  
-2 \frac{\nabla_x\rho}{\partial_z\rho} ,\quad 
c\defn \frac{1}{\partial_z\rho}\bigl( a \partial_z^2 \rho +\Delta_x \rho 
+ b \cdot\nabla_x \partial_z \rho\bigr).
\end{equation}
It will be convenient to have a constant coefficient in front of~$\partial_z^2 \widetilde{\phi}$. 
Dividing~\eqref{dnint1a} by~$a$ we obtain
\begin{equation}\label{dnint1}
(\partial_z^2  +\alpha\Delta_x  + \beta \cdot\nabla_x\partial_z   - \gamma \partial_z )\widetilde{\phi} =0,
\end{equation}
where 
\begin{equation}\label{alpha}
\alpha\defn \frac{(\partial_z\rho)^2}{1+|\nabla_x   \rho |^2},\quad 
\beta\defn  -2 \frac{\partial_z \rho \nabla_x \rho}{1+|\nabla_x  \rho |^2} ,\quad 
\gamma \defn \frac{1}{\partial_z\rho}\bigl(  \partial_z^2 \rho 
+\alpha\Delta_x \rho + \beta \cdot \nabla_x \partial_z \rho\bigr).
\end{equation}

In the  coordinates $(x,z)$, according to \eqref{def:DN} we have
\begin{equation}\label{def2:DN}
G(\eta)\psi = \widetilde{U}\arrowvert_{z=0}, \quad \widetilde{U}= \Lambda_1 \widetilde{\phi} - \nabla_x \rho\cdot \Lambda_2 \widetilde{\phi}.
\end{equation}
The following remark will be useful in the sequel. We have
\begin{equation}\label{dzU}
\partial_z\widetilde{U} = - \nabla_x \cdot((\partial_z \rho) \Lambda_2 \widetilde{\phi}).
\end{equation}
Indeed we can write
\begin{equation*}
\begin{aligned}
\partial_z \widetilde{U} &= \partial_z \Lambda_1 \widetilde{\phi} - 
\nabla_x \partial_z \rho \cdot \Lambda_2 \widetilde{\phi} - \nabla_x \rho  \cdot \partial_z \Lambda_2 \widetilde{\phi} \\
&= (\partial_z \rho) \Lambda_1^2 \widetilde{\phi} -  
\nabla_x \partial_z \rho \cdot \Lambda_2 \widetilde{\phi} + (\partial_z \rho) (\Lambda_2 - \nabla_x) \Lambda_2 \widetilde{\phi}\\
&= (\partial_z \rho)( \Lambda_1^2 +  \Lambda_2^2) \widetilde{\phi} -\nabla_x \cdot\big((\partial_z \rho) \Lambda_2 \widetilde{\phi} \big).
\end{aligned}
\end{equation*}
Since $ (\Lambda_1^2 +  \Lambda_2^2) \widetilde{\phi} =0$ we obtain  \eqref{dzU}.

\subsubsection{Continuity of the Dirichlet-Neumann operator}
\begin{theo}\label{th.23}
Let $\eta \in W^{1,\infty}(\xR^d).$ Then for all $f \in H^\mez(\xR^d),$ 
$$
G(\eta)f = \sqrt{1+ \vert \nabla \eta \vert^2}\partial_n \phi\arrowvert_{y= \eta(x)} = [\partial_y - \nabla \eta \cdot \nabla \phi] \arrowvert_{y=\eta(x)}
$$
is well defined in $H^{-\mez}(\xR^d)$. 
Furthermore there exists $\mathcal{F}: \xR^+ \to \xR^+$ 
non decreasing independent of $\eta,f$ such that   
$$ \| G(\eta) f \|_{H^{-\mez} ({\mathbf{R}}^d)} \leq \mathcal{F} 
( \|\eta\|_{W^{1,\infty}(\xR^d)}) \|f \|_{H^\mez ({\mathbf{R}}^d)}.
$$
\end{theo}
\begin{proof}
With the preceding notations we have $G(\eta) f = \big(\Lambda_1 \widetilde{\phi} - \nabla_x \rho \cdot \Lambda_2\widetilde{\phi}\big) \arrowvert_{z=0}.$ Let us set $\widetilde{U} = \Lambda_1 \widetilde{\phi} - \nabla_x \rho \cdot \Lambda_2\widetilde{\phi}.$ 
Then the theorem will follow from Lemma \ref{lem.3.11} with $s = -\mez$ 
and the following inequality where $I=(-1,0)$
\begin{equation*}
 \Vert \widetilde{U} \Vert_{L^2(I, L^2(\xR^d))} + \Vert \partial_z\widetilde{U} \Vert_{L^2(I, H^{-1}(\xR^d))} \leq \mathcal{F}(\Vert \eta \Vert_{W^{1,\infty}(\xR^d)})\Vert f \Vert_{H^\mez(\xR^d)}. 
\end{equation*}
The estimate on $\widetilde{U}$ is a consequence of \eqref{est:phi}, \eqref{Lambda} 
and the estimate $(1)$ in Lemma \ref{rho:diffeo} on $\nabla_x\rho$. 
To estimate $\partial_z \widetilde{U}$ we use \eqref{dzU} and 
the estimate on $\partial_z\rho $ in Lemma~\ref{rho:diffeo}.
\end{proof}

Notice that, as a by product of the previous proof, 
in the system of coordinates~$(x,z)$, 
the variational solution of~\eqref{v},~$\widetilde{\phi}$, 
satisfies
\begin{equation}\label{eq.Xs}
\widetilde{\phi} \in C^0_z([-1,0]; H^\mez ({\mathbf{R}}^d)) 
\cap C^1_z([-1,0]; H^{-\mez} ( {\mathbf{R}}^d)).
\end{equation}

We also state a second basic  strong continuity of the Dirichlet-Neumann operator.

\begin{theo}\label{th.25}
There exists a non decreasing function $\mathcal{F}\colon\xR_+\rightarrow \xR_+$ such that, 
for all $\eta_j\in W^{1,\infty}({\mathbf{R}}^d)$,~$j=1,2$ and all $f\in H^{\mez} ({\mathbf{R}}^d)$, 
$$
\big\Vert \bigl(G(\eta_1)- G(\eta_2)\bigr) f \big\Vert_{H^{-\mez} }
\leq \mathcal{F} \bigl(\|(\eta_1,\eta_2)\|_{W^{1,\infty}\times W^{1,\infty}}\bigr)
\|\eta_1 - \eta_2\|_{W^{1,\infty}} \|f \|_{H^\mez}.
$$
\end{theo}
\begin{proof} 
We use the notations introduced in \S\ref{s:flattening}. 
Namely, for~$j=1,2$, $x\in\xR^d, z\in I:=(-1,0)$ we introduce 
$\rho_j(x,z)$ and~$v_j(x,z)$ defined by~\eqref{diffeo},   
\begin{equation*}
\begin{aligned}
\rho_j(x,z)&=(1+z)(e^{\delta  z\langle D_x\rangle}\eta_j)(x)-z\big\{e^{-(1+ z)\delta  \langle D_x \rangle }\eta_j(x) -h\big\},\\
\rho_j(x,z)&= z+1+\eta_j -h,\quad \text{if } (x,z)\in \widetilde{\Omega}_2,
\end{aligned}
\end{equation*}
where $\delta \Vert \eta_j \Vert_{W^{1,\infty}(\xR^d)}$ is small enough for $j = 1,2.$

Notice that we have the following estimates (see Lemma \ref{rho:diffeo})
\begin{equation}\label{est-rho}
\left\{
\begin{aligned}
&(i) \quad \partial_z \rho_j \geq \min(1,\frac{h}{2}), \quad (x,z)\in \widetilde{\Omega},\\
& (ii) \quad \Vert \nabla_{x,z}\rho_j \Vert_{L^\infty(\widetilde{\Omega})}
\leq C(1+ \Vert \eta_j\Vert_{W^{1,\infty}(\xR^d)})\\
&(iii) \quad \Vert \nabla_{x,z}(\rho_1 -\rho_2)\Vert_{L^\infty(I,L^\infty(\xR^d))}
\leq C\Vert \eta_1-\eta_2\Vert_{W^{1,\infty}(\xR^d)}.
 \end{aligned}
\right.
 \end{equation}
Recall also that we have set
\begin{equation}\label{Lambda-2}
\Lambda_1^i = \frac{1}{\partial_z \rho_i}\partial_z, \quad \Lambda_2^i 
= \nabla_x - \frac{\nabla_x \rho_i}{\partial_z \rho_i}\partial_z.
\end{equation}
It follows from  \eqref{est-rho} that for $k=1,2$ we have with $W^{1,\infty} = W^{1,\infty} (\xR^d),$
\begin{equation}\label{diff-lambda}
\left\{
 \begin{aligned}
 &(i) \quad \Lambda^1_k - \Lambda^2_k = \beta_k \partial_z, 
 \quad \text{with }\supp\beta_k \subset \xR^d\times I, \\
 &(ii) \quad \Vert \beta_k \Vert_{L^\infty(I\times \xR^d)}
 \leq\mathcal{F}(\Vert (\eta_1, \eta_2)\Vert_{W^{1,\infty} \times W^{1,\infty} }) 
 \Vert \eta_1-\eta_2\Vert_{W^{1,\infty} }
 \end{aligned}
 \right.
 \end{equation}

Then we set $\widetilde{\phi}_j(x,z) = \phi_j(x,\rho_j(x,z))$ (where $\Delta_{x,y}\phi_j =0$ in $\Omega_j, \phi_j\arrowvert_{\Sigma_j} = f $) 
and we recall (see \eqref{def2:DN}) that
\begin{equation}\label{Getaj}
 G(\eta_j)f = U_j\arrowvert_{z=0}, \quad U_j= \Lambda_1^j \widetilde{\phi}_j - \nabla_x \rho_j\cdot \Lambda_2^j \widetilde{\phi}_j.
\end{equation}
\begin{lemm}\label{dem:estvar}
Set $I=(-1,0)$, $v = \widetilde{\phi}_1 - \widetilde{\phi}_2,$ and $ \Lambda^j = (\Lambda^j_1, \Lambda^j_2) $. 
There exists a non decreasing function $\mathcal{F}:\xR^+\to \xR^+$ such that 
\begin{equation}\label{est2:var}
\Vert \Lambda^j v \Vert_{L^2(I;L^2(\xR^d))}
\leq \mathcal{F}( \Vert(\eta_1, \eta_2)\Vert_{W^{1,\infty} \times W^{1,\infty} })
\Vert \eta_1-\eta_2\Vert_{W^{1,\infty} }
\Vert f\|_{H^{\mez}}.
\end{equation}
\end{lemm}
Let us show how this Lemma implies 
Theorem \ref{th.25}. According to \eqref{Getaj} we have
\begin{equation}\label{formeU}
\begin{aligned}
 U_1-U_2 &= (1) +(2) +(3)+ (4)+ (5) \quad \text{where}\\
(1) &= \Lambda_1^1 v, \quad (2) 
= (\Lambda_1^1 -\Lambda_1^2)\widetilde{\phi}_2, \quad (3) =  - \nabla_x (\rho_1 -   \rho_2) \Lambda_2^1 \widetilde{\phi}_1\\
(4)&= - (\nabla_x \rho_2) \Lambda_2^1v, \quad (5) =   - (\nabla_x \rho_2)(\Lambda_2^1- \Lambda_2^2)\widetilde{\phi}_2.
\end{aligned}
 \end{equation}
The $L^2(I, L^2(\xR^d))$ norms of $(1)$ and $(4)$ are estimated 
using \eqref{est2:var} and \eqref{est-rho}. Also, the $L^2(I, L^2(\xR^d))$ norms of $(2)$ and $(5)$ are estimated 
by the right hand side of \eqref{est2:var} using \eqref{diff-lambda} and \eqref{est:phi}. 
Eventually the $L^2(I, L^2(\xR^d))$ norm of $(3)$ is also estimated by the right hand side 
of \eqref{est2:var} using \eqref{est-rho} $(iii)$ and \eqref{est:phi}. It follows that 
\begin{equation}\label{est:Uj}
\Vert U_1-U_2 \Vert_{L^2(I,L^2 )} 
\leq \mathcal{F}( \Vert(\eta_1, \eta_2)\Vert_{W^{1,\infty} \times W^{1,\infty}})
\|\eta_1 - \eta_2\|_{W^{1, \infty}}  \| f\|_{H^{\mez}}.
\end{equation}
Now according to \eqref{dzU} we have
\begin{equation}\label{formedzU}
\partial_z(U_1-U_2)
= -\nabla_x\cdot\big(  \partial_z(\rho_1- \rho_2 )\Lambda_2^1 \widetilde{\phi}_1 
+ (\partial_z\rho_2)(\Lambda_2^1-\Lambda_2^2) \widetilde{\phi}_1 +(\partial_z\rho_2)\Lambda_2^2v\big).
\end{equation}
Therefore using the same estimates as above we see easily that
\begin{equation}\label{est:dzUj}
\Vert  \partial_z(U_1-U_2)\Vert_{L^2(I,H^{-1} )}
\leq \mathcal{F}( \Vert(\eta_1, \eta_2)\Vert_{W^{1,\infty} \times W^{1,\infty} }) \|\eta_1 - \eta_2\|_{W^{1, \infty}}  \| f\|_{H^{\mez}}.
\end{equation}
Then Theorem \ref{th.25} follows from \eqref{est:Uj}, \eqref{est:dzUj} and Lemma \ref{lem.3.11}.
\end{proof}
\begin{proof}[Proof of Lemma \ref{dem:estvar}]

We use the variational characterization of the solutions $u_i$. 
First of all we notice that $\widetilde{\phi}_1- \widetilde{\phi}_2 = \widetilde{u}_1 - \widetilde{u}_2=:v$. 
Now setting $X=(x,z)$ and recalling that $ \Lambda^i=(\Lambda^i_1, \Lambda^i_2),$         we have
\begin{equation}\label{equi}
\int_{\widetilde{\Omega}}\Lambda ^i\widetilde{u}_i\cdot\Lambda ^i\theta \, J_i\, dX 
=-\int_{\widetilde{\Omega}}\Lambda ^i \underline{\widetilde{f}} \cdot\Lambda ^i\theta \, J_i\, dX 
\end{equation}
for all $\theta \in H^{1,0}(\widetilde{\Omega}),$ where $J_i = \vert \partial_z \rho_i \vert.$
 
Taking the difference between the 
two equations \eqref{equi}, using \eqref{est-rho} 
and setting $\theta = v = \widetilde{u}_1 -\widetilde{u}_2$  one can find a positive constant $C$ such that 
\begin{equation*}
\int_{\widetilde{\Omega}}\vert \Lambda ^1v\vert^2\, dX \leq C\sum_{k=1}^6 A_k,
\end{equation*}
where
\begin{equation*}
\left\{
\begin{aligned}
A_1&=   \int_{\widetilde{\Omega}}\vert (\Lambda^1 - \Lambda^2) \widetilde{u}_2 \vert \vert \Lambda^1 v\vert \,J_1\,  dX, \qquad   
&&A_2=  \int_{\widetilde{\Omega}}\vert (\Lambda^1 - \Lambda^2) v \vert \vert \Lambda^2 \widetilde{u}_2 \vert \,J_1 dX,  \\
A_3 &= \int_{\widetilde{\Omega}}\vert \Lambda^2\widetilde{u}_2\vert \vert  \Lambda^2v \vert \,\vert J_1-J_2\vert \, dX, \qquad 
&&A_4 = \int_{\widetilde{\Omega}}\vert (\Lambda^1 - \Lambda^2) \underline{\widetilde{f}}  \vert \vert \Lambda^1 v\vert \,J_1\, dX, \\
A_5 &=  \int_{\widetilde{\Omega}}\vert (\Lambda^1 - \Lambda^2) v \vert \vert \Lambda^2 \underline{\widetilde{f}}  \vert \,J_1\, dX,
\qquad 
&&A_6 =  \int_{\widetilde{\Omega}}\vert \Lambda^2  \widetilde{f} \vert \vert  \Lambda^2v \vert \,\vert J_1-J_2\vert\, dX.
\end{aligned}
\right.
\end{equation*}
Using \eqref{diff-lambda}, \eqref{est:phi}, \eqref{est-rho} we can write
\begin{equation}\label{est:A1}  
\begin{aligned}
\vert A_1 \vert &\leq \Vert \beta \Vert_{L^{\infty}(I\times \xR^d)}
\Vert J_1 \Vert_{L^{\infty}(I\times \xR^d)}\Vert \partial_z \widetilde{u}_2\Vert_{L^{2}(I\times \xR^d)}
\Vert \Lambda^1 v\Vert_{L^{2}(\widetilde{\Omega})}\\
& \leq \mathcal{F}( \| (\eta_1, \eta_2)\|_{W^{1, \infty}\times W^{1, \infty}})
\|\eta_1 - \eta_2\|_{W^{1, \infty}}  \| f\|_{H^{\mez}}\Vert \Lambda^1 v\Vert_{L^{2}(\widetilde{\Omega})}.
\end{aligned}
\end{equation}
Since $\Lambda_j^1 - \Lambda_j^2 = \frac{\beta_j}{\partial_z \rho_1}\Lambda_1^1$ 
the term $A_2$ can be bounded by the right hand side of \eqref{est:A1}.

Now we have $\Vert J_1 - J_2 \Vert _{L^{\infty}(I\times \xR^d)} \leq C\Vert \eta_1 -\eta_2 \Vert_{W^{1,\infty}(\xR^d)}$ and 
$$
\Vert \Lambda^2 v\Vert_{L^{2}(\widetilde{\Omega})}
\leq \mathcal{F}( \| (\eta_1, \eta_2)\|_{W^{1, \infty}\times W^{1, \infty}}) \Vert \Lambda^1v \Vert _{L^{2}(\widetilde{\Omega})}.
$$
So using \eqref{est:phi} we see that the term $A_3$ 
can be also estimated by the right hand side of \eqref{est:A1}. 
To estimate the terms $A_4$ to $A_6$ we use the same arguments and also \eqref{est:psi-H1}. 
This completes the proof.
\end{proof}


Let us finish this definition section by recalling 
also the following result which is a consequence of~\cite[Lemma 2.9]{ABZ1}.
\begin{lemm}\label{regell}
Assume that~$-\mez \leq a<b \leq -\frac{1}{5}$  then the strip 
$S_{a,b}= \{ (x,y) \in \mathbf{R}^{d+1}\,:\,ah<y- \eta(x)< bh\}$ 
is included in~$\Omega$ and for any~$k \ge 1$, there exists~$C>0$ such that 
$$
\lA\phi\rA_{H^{k}(S_{a,b})}\leq C \|\psi\|_{H^\mez ({\mathbf{R}}^d)}.
$$
\end{lemm}

\subsection{Paralinearization of the Dirichlet-Neumann operator}
In the case of smooth domains, it is known that, modulo a smoothing operator, $G(\eta)$ 
is a pseudo-differential operator with principal symbol given by
\begin{equation}\label{eq.lambda}
\lambda(x,\xi)\defn\sqrt{(1+\la\partialx\eta(x)\ra^2)\la\xi\ra^2-(\partialx\eta(x)\cdot\xi)^2}.
\end{equation}
Notice that~$\lambda$ is well-defined for any~$C^1$ function~$\eta$. 
The main result of this section allow 
to compare~$G(\eta)$ to the paradifferential operator~$T_\lambda$ when~$\eta$ 
has limited regularity. Namely we want to estimate 
the operator
$$
R(\eta)=G(\eta)-T_\lambda.
$$
Such an analysis was at the heart of our previous work~\cite{AM}~\cite[Proposition 3.14]{ABZ1} 
for ``smooth domains" ($\eta \in H^{s+\mez}, s> 2+ \frac d 2$). 
Here we are able to lower the regularity thresholds down to $s>\mez+\frac d 2$. 
The following results, 
which we think are of independent interest, complement  
previous estimates about the Dirichlet-Neumann operator 
by Craig-Schanz-Sulem~\cite{CSS}, 
Beyer--G\"unther~\cite{BG}, Wu~\cite{WuInvent,WuJAMS}, Lannes~\cite{LannesJAMS}.
\begin{theo}\label{coro:estiDN}
Let~$d\ge 1$,~$s>\mez+\frac{d}{2}$ and~$\frac 1 2 \leq \sigma \leq s+ \frac 1 2$. 
Then there exists a non-decreasing function~$\mathcal{F}\colon\xR_+\rightarrow\xR_+$ 
such that, for all 
$\eta\in H^{s+\mez}({\mathbf{R}}^d)$ and all~$f\in H^{\sigma}({\mathbf{R}}^d)$, we have 
$G(\eta)f\in H^{\sigma-1}({\mathbf{R}}^d)$, together with the estimate
\begin{equation}\label{ests+2}
\lA G(\eta)f \rA_{H^{\sigma-1}({\mathbf{R}}^d)}\le 
\mathcal{F}\bigl(\| \eta \|_{H^{s+\mez}({\mathbf{R}}^d)}\bigr)\lA f\rA_{H^{\sigma}({\mathbf{R}}^d)}.
\end{equation}
\end{theo}
We also prove error estimates. 
\begin{prop}\label{coro:paraDN1}
Let~$d\ge 1$ and~$s>\mez+\frac{d}{2}$. For any~$\frac{1}{2}\leq \sigma \leq s$ and any
$$
0<\eps\leq \mez, \qquad \eps< s-\mez-\frac{d}{2},
$$
there exists a non-decreasing function~$\mathcal{F}\colon\xR_+\rightarrow\xR_+$ such that 
$R(\eta)f\defn G(\eta)f-T_\lambda f$ satisfies
\begin{equation*}
\lA R(\eta)f\rA_{H^{\sigma-1+\eps}({\mathbf{R}}^d)}\le \mathcal{F} \bigl(\| \eta \|_{H^{s+\mez}({\mathbf{R}}^d)}\bigr)\lA f\rA_{H^{\sigma}({\mathbf{R}}^d)}.
\end{equation*}
\end{prop}

To prove Theorem~\ref{coro:estiDN} and Proposition~\ref{coro:paraDN1}, 
we shall use the diffeomorphism $\rho$ defined by \eqref{diffeo}. 
Then recall from \eqref{dnint1} that the function 
$\widetilde{\phi}(x,z)=\phi(x,\rho(x,z))$ satisfies 
\begin{equation}\label{dnint1-b}
(\partial_z^2  +\alpha\Delta_x  + \beta \cdot\nabla_x\partial_z   - \gamma \partial_z )\widetilde{\phi} =0,\qquad 
\widetilde{\phi} \arrowvert_{z=0}=\phi\arrowvert_{y=\eta(x)}=
f,
\end{equation}
where the coefficients are given by \eqref{alpha}, and
\begin{equation*}
G(\eta) f= (\Lambda_1 \widetilde{\phi}  - \nabla \rho\cdot\Lambda_2 \widetilde{\phi} ) \arrowvert_{z=0}= \big(\frac{1 +|\partialx \rho |^2}{\partial_z \rho} \partial_z \widetilde{\phi}   - \partialx \rho\cdot\partialx \widetilde{\phi} \big)
\big\arrowvert_{z=0}.
\end{equation*}

We conclude this paragraph by stating 
elliptic estimates for the solutions of \eqref{dnint1-b}.
For later purpose, we will consider the non-homogeneous case. 
This yields no new difficulty and will be useful later to estimate the pressure (see Section~\ref{S:tameP}). 
We thus consider the problem
\begin{equation}\label{dnint1F}
\partial_z^2 v +\alpha\Delta v + \beta \cdot\partialx\partial_z v  
- \gamma \partial_z v=F_0,\quad v\arrowvert_{z=0}=f,
\end{equation}
where~$f=f(x)$ and~$F_0=F_0(x,z)$ are given functions. 
Recall that for $\mu\in\xR$, the spaces~$X^\mu(I), Y^\mu(I)$ are defined by (see~\eqref{XY}): 
\begin{equation}\begin{aligned}
X^\mu(I)&=C^0_z(I;H^{\mu}({\mathbf{R}}^d))\cap L^2_z(I;H^{\mu+\mez}({\mathbf{R}}^d)),\\
Y^\mu(I)&=L^1_z(I;H^{\mu}({\mathbf{R}}^d))+L^2_z(I;H^{\mu-\mez}({\mathbf{R}}^d)).
\end{aligned}
\end{equation}
Recall that~$H^\sigma({\mathbf{R}}^d)$ 
is an algebra for $\sigma>d/2$ and so is~$C^0(I; H^\sigma({\mathbf{R}}^d))$. Also, using the tame estimate~\eqref{prS2}, we obtain the following
\begin{lemm}\label{lem.alg} Assume that~$\sigma>\frac d 2~$. 
Then the space~$X^\sigma(I)$ is an algebra. Moreover if $F:\xC^N \to \xC$ is a $C^\infty$-bounded function such that $F(0) =0$ one can find  non decreasing functions $\mathcal{F}, \mathcal{F}_1$ from $\xR^+$ to $\xR^+$ such that
$$
\Vert F(U)\Vert_{X^\sigma(I)}
\leq \mathcal{F}(\Vert U\Vert_{L^\infty(I\times \xR^d)}) \Vert U\Vert_{X^\sigma(I)} \leq \mathcal{F}_1(\Vert U\Vert_{ X^\sigma(I)}).
$$
\end{lemm}

With these notations, we want to estimate 
the~$X^\sigma$-norm of~$\nabla_{x,z}v$ in terms of the~$H^{\sigma+1}$-norm 
of the data and the~$Y^\sigma$-norm of the source term. 
An important point is that we need to consider the case of rough coefficients. 
In this section we only assume that 
$\eta\in H^{s+\mez}({\mathbf{R}}^d)$ for some~$s>1/2+d/2$.  
An interesting point is that we shall prove elliptic estimates 
as well as elliptic regularity (in other words, we do not prove only {\em a priori} estimates). 
Our only assumption is that~$v$ is given by a variational problem, 
so that
\begin{equation}\label{assu:var}
\lA \nabla_{x,z} v\rA_{X^{-\mez}([-1,0])}
<+\infty.
\end{equation}
\begin{rema}\label{rema:321}
In the case where~$v(x,z)=\widetilde{\phi}(x,z)=\phi(x,\rho(x,z))$ with~$\phi$ the variational solution of
$$
\Delta_{x,y}\phi=0, \quad \phi\arrowvert_{y=\eta}=f, \quad \partial_n\phi=0 \text{ on }\Gamma,
$$
then~\eqref{eq.Xs} together with \eqref{est:phi} shows that~$v$ satisfies the assumption~\eqref{assu:var}.
\end{rema}
\begin{prop}\label{p1}
Let~$d\ge 1$ and 
$$
s> \mez + \frac d 2,\quad - \frac 1 2 \leq \sigma \leq s - \frac 1 2.
$$ 
Consider ~$f\in H^{\sigma+1}({\mathbf{R}}^d)$, 
$F_0\in Y^\sigma([-1,0])$ and~$v$ satisfying the assumption~\eqref{assu:var} 
solution to~\eqref{dnint1F}. 
Then for any~$z_0\in (-1,0)$, 
$
\nabla_{x,z}v \in X^{\sigma}([z_0,0])$,
and
$$
\lA \nabla_{x,z} v\rA_{X^{\sigma}([z_0,0])}
\le \mathcal{F}(\| \eta \|_{H^{s+\mez}})\left\{\lA f\rA_{H^{\sigma+1}}+\lA F_0\rA_{Y^\sigma([-1,0])}
+ \lA \nabla_{x,z} v\rA_{X^{-\mez}([-1,0])} \right\},
$$
for some non-decreasing function~$\mathcal{F}\colon\xR_+\rightarrow\xR_+$ depending only on~$\sigma$. 
\end{prop}
To prove Proposition~\ref{p1} we shall proceed by induction on the regularity~$\sigma$. 
\begin{defi}
Given~$\sigma$ such that 
$-1/2\le \sigma\le s-1/2$, we say that the property~$\Hs$ is satisfied 
if for any interval~$I\Subset (-1,0]$,
\begin{equation}\label{av}
\lA \nabla_{x,z} v\rA_{X^\sigma(I)}\le
\mathcal{F}(\| \eta \|_{H^{s+\mez}})\left\{\lA f\rA_{H^{\sigma+1}}+\lA F_0\rA_{Y^\sigma([-1,0])} + \lA \nabla_{x,z} v\rA_{X^{-\mez}([-1, 0])} \right\}  ,
\end{equation}
for some non-decreasing function~$\mathcal{F}\colon\xR_+\rightarrow\xR_+$ 
depending only on~$I$ and $\sigma$.
\end{defi}

With this definition, note that Assumption~\eqref{assu:var} 
means that property~$\mathcal{H}_{-1/2}$ is satisfied. 
Consequently, Proposition~\ref{p1} is 
an immediate consequence of the following proposition 
which will be proved in Sections~\ref{S:NE} (see \S\ref{S:PPLind}).
\begin{prop}\label{Linduction}Let~$s>\frac 1 2 + \frac d 2$. For any~$\eps$ such that
\begin{equation}\label{cond:eps}
0<\eps\leq \mez, \qquad \eps< s-\mez-\frac{d}{2},
\end{equation}
if~$\Hs$ is satisfied for some ~$-1/2\le \sigma\le s-1/2-\eps$, then~$\mathcal{H}_{\sigma+\eps}$ is satisfied.
\end{prop}
\subsection{Nonlinear estimates}\label{S:NE}
Let us fix~$\eps$ satisfying \eqref{cond:eps},~$\sigma$ such that
$$
-\mez\le \sigma\le s-\mez-\eps
$$
and assume that~$\Hs$ is satisfied. 
We begin by estimating the coefficients~$\alpha,\beta,\gamma$ in~\eqref{alpha} in terms of~$\| \eta \|_{H^{s+\mez}}$.
\begin{lemm}\label{lem.3.25} 
Let $J= [-1,0]$ and $s> \mez+ \frac d 2$. We have
\begin{equation}
\lA \alpha - h^2\rA_{X^{s-\mez}(J)}+
\lA \beta\rA_{X^{s-\mez}(J)}+\lA \gamma\rA_{X^{s-\tdm}(J)}\le \mathcal{F}(\| \eta \|_{H^{s+\mez}}).
\label{regv}
\end{equation}
 
\end{lemm}
\begin{proof}
 
According to~\eqref{eq.rho1}, we can write
$$(\partial_z \rho)^2 = h^2  +G \quad  \text{with }  \Vert G \Vert_{X^{s-\mez}(I)} \leq \mathcal{F}(\Vert \eta \Vert_{H^{s+\mez}(\xR^d)}).$$
Noticing that $\frac 1 {1+ |\nabla\rho|^2} = 1 - \frac {|\nabla\rho|^2} {1+ |\nabla\rho|^2}$ we obtain
$$
\alpha -  h^2  = - h^2 \frac {|\nabla\rho|^2} {1+ |\nabla\rho|^2} +G -G \frac {|\nabla\rho|^2} {1+ |\nabla\rho|^2}
$$
and we use Lemma \ref{lem.alg} with $\sigma = s-\mez$ 
together with \eqref{eq.rho1}. 
The estimates for $\beta$ and $\gamma$ are proved along the same lines.
\end{proof}

\begin{lemm}\label{UV}
There exists a constant~$K$ such that for all~$I\subset [-1,0]$,
\begin{equation}\label{gpzv}
\lA F_1\rA_{Y^{\sigma+\eps}(I)}
\le K 
\lA \gamma \rA_{X^{s-\tdm}(I)}\lA \partial_z v\rA_{X^\sigma(I)},
\end{equation}
where $F_1= \gamma \partial_z v$.
\end{lemm}
\begin{proof}
We shall prove that, on the one hand, if~$-1/2\le \sigma\le s-1-\eps$ then
\begin{equation}\label{gpzv1}
\lA \gamma \partial_z v\rA_{L^1(I;H^{\s+\eps})} \les 
\lA \gamma \rA_{L^2(I;H^{s-1})}\lA \partial_z v\rA_{L^2(I;H^{\s+\mez})},
\end{equation}
and on the other hand, 
if~$s-1-\eps\le \sigma\le s-\mez-\eps$ then 
\begin{equation}\label{gpzv2}
\lA \gamma \partial_z v\rA_{L^2(I;H^{\s-\mez+\eps})} \les 
\lA \gamma \rA_{L^2(I;H^{s-1})}\lA \partial_z v\rA_{L^\infty(I;H^{\s})}.
\end{equation}
Since~$s>\eps+\mez+d/2$, if~$-1/2\le \sigma\le s-1-\eps$ then
$$
s-1+\sigma+\mez >0,~ \sigma+\eps\le \sigma+\mez, 
~  \sigma+\eps\le s-1,~
 \sigma+\eps < s-1+\sigma+\mez-\frac{d}{2},
$$
and hence the product rule in Sobolev spaces~\eqref{pr} implies that
$$
\lA \gamma(z) \partial_z v(z)\rA_{H^{\s+\eps}} 
\les \lA \gamma(z) \rA_{H^{s-1}}\lA \partial_z v(z)\rA_{H^{\s+\mez}}.
$$
Integrating in~$z$ and using the Cauchy-Schwarz inequality, we obtain~\eqref{gpzv1}.
On the other hand, if~$-\eps\le \sigma\le s-\mez-\eps$ then one easily checks that
$$
s-1+\sigma>0,\quad \sigma-\mez+\eps\le \sigma, 
\quad  \sigma-\mez+\eps\le s-1,\quad 
 \sigma-\mez+\eps < s-1+\sigma-\frac{d}{2},
$$
and hence the product rule~\eqref{pr} implies that
$$
\lA \gamma (z)\partial_z v(z)\rA_{H^{\s-\mez+\eps}} 
\les \lA \gamma (z)\rA_{H^{s-1}}\lA \partial_z v(z)\rA_{H^{\s}}.
$$
Taking the~$L^2$-norm in~$z$, we obtain~\eqref{gpzv2}.
\end{proof}

Our next step is to replace the multiplication by~$\alpha$ (resp.\ $\beta$) by the paramultiplication 
by~$T_\alpha$ (resp.\ $T_\beta$).
\begin{lemm}\label{l3.8}
There exists a constant~$K$ such that for all~$I\subset [-1,0]$,~$v$
satisfies the paradifferential equation
\begin{equation}\label{2m.11}
\partial_z^2 v  +T_{\alpha}\Deltax v +T_{\beta}\cdot \partialx \partial_z v =F_0+F_1+F_2,
\end{equation}
for some remainder 
\begin{equation}\label{F2e}
F_2= (T_\alpha-\alpha)\Delta v+(T_\beta-\beta)\cdot\partialx\partial_z v
\end{equation} satisfying 
\begin{equation}\label{F2}
\lA F_2\rA_{Y^{\sigma+\eps}(I)}
\le K
\left\{1+\lA \alpha - h^2\rA_{X^{s-\mez}([-1,0])}+
\lA \beta\rA_{X^{s-\mez}([-1,0])}\right\}\lA \nabla_{x,z} v\rA_{X^\sigma(I)}.
\end{equation}
\end{lemm}
\begin{proof}
According to Proposition~\ref{lemPa}, we have 
$$
\lA au - T_a u\rA_{H^{\gamma}}\les \lA a\rA_{H^{r}}\lA u\rA_{H^\mu},
$$
 provided that~$r,\mu,\gamma\in \xR$ satisfy
\begin{equation}\label{criteria1}
r+\mu>0,\quad \gamma\le r  \quad\text{and } \gamma < r+\mu-\frac{d}{2}.
\end{equation}

Since~$s>\eps+1/2+d/2$, if~$-1/2\le \sigma\le s-\mez-\eps$ then
$$
s+\sigma-\mez >0,\quad \sigma+\eps\le s, \quad
 \sigma+\eps < s+\sigma-\mez-\frac{d}{2},
$$
and hence \eqref{criteria1} applies with
$$
\gamma =\sigma+\eps,\quad r =s ,\quad \mu=\sigma-\mez.
$$
This implies that if~$-1/2\le \sigma\le s-\mez-\eps$ then
\begin{equation}\label{F2-1}
\begin{aligned}
&\lA (T_\alpha-\alpha)\Delta  v\rA_{L^1(I;H^{\s+\eps})} \les 
\Bigl(1+\lA \alpha - \frac{h^2}{16}\rA_{L^2(I;H^{s})}\Bigr)\lA \Delta v\rA_{L^2(I;H^{\s-\mez})},\\
&\lA (T_\beta-\beta)\nabla\partial_z v  \rA_{L^1(I;H^{\s+\eps})} \les 
\lA \beta \rA_{L^2(I;H^{s})}\lA \nabla\partial_z v\rA_{L^2(I;H^{\s-\mez})},
\end{aligned}
\end{equation}
which yields
$$
\lA F_2\rA_{Y^{\sigma+\eps}(I)}\le \lA F_2\rA_{L^1(I;H^{\s+\eps})} \les 
\Bigl\{1+\lA \alpha - h^2 \rA_{X^{s-\mez}(I)}+\lA \beta \rA_{X^{s-\mez}(I)}\Bigr\}\lA \nabla_{x,z} v\rA_{X^{\sigma}(I)}.
$$
This concludes the proof.
\end{proof}

Our next task is to perform a decoupling into a forward and a backward parabolic evolution equations. 
Recall that by assumption~$\eta\in H^{s+\mez}(\xR^d)$ with~$s>\eps+1/2+d/2$. 
In particular,~$\eta\in C_*^{1+\eps}(\xR^d)$.

\begin{lemm}\label{lemm:total}
There exist two symbols 
$a,A$ in~$\Gamma^{1}_{\eps}({\mathbf{R}}^d\times [-1,0])$ and a remainder~$F_3$ 
such that,
\begin{equation}\label{2m2.b}
( \partial_z - T_a) (\partial_z - T_A)v =F_0+F_1+F_2+F_3,
\end{equation}
with
\begin{equation}\label{boundsaA}
\mathcal{M}^1_{\eps}(a)+ \mathcal{M}^1_\eps(A)\le 
\mathcal{F}\bigl( \| \eta \|_{H^{s+\mez}}\bigr),
\end{equation}
and
\begin{equation*}
\lA F_3\rA_{L^2(I;H^{\sigma-\mez+\eps})}\le \mathcal{F}\bigl(\| \eta \|_{H^{s+\mez}}\bigr)\lA 
\nabla_{x,z} v\rA_{L^2(I;H^{\sigma+\mez})},
\end{equation*}
for some non-decreasing function~$\mathcal{F}\colon\xR_+\rightarrow\xR_+$.
\end{lemm}
\begin{proof}
We seek~$a,A$ satisfying
\begin{align*}
a(z;x,\xi)A(z;x,\xi)=-\alpha(x,z) \la \xi\ra^2,\quad a(z;x,\xi)+A(z;x,\xi)=-i\beta(x,z)\cdot\xi.
\end{align*}
We thus set
\begin{equation}\label{aA}
a = \frac{1}{2}\bigl( -i  \beta\cdot \xi
-   \sqrt{ 4\alpha \la \xi \ra^2 -  (\beta \cdot \xi)^2}\bigr), \qquad A  = \frac{1}{2}\bigl( -i  \beta\cdot \xi
+   \sqrt{ 4\alpha \la \xi \ra^2 -  (\beta \cdot \xi)^2}\bigr).
\end{equation}
Directly from the definition of~$\alpha$,~$\beta$~\eqref{alpha}, note that
$$
\exists c>0; \quad \sqrt{ 4\alpha \la \xi \ra^2 -  (\beta \cdot \xi)^2}  \geq  c\la\xi\ra.
$$
According to~\eqref{regv} the symbols~$a,A$ belong to~$\Gamma^{1}_{\eps}({\mathbf{R}}^d\times [-1,0])$ 
and they satisfy the bound~\eqref{boundsaA}. 
Therefore, we have
\begin{equation}\label{F3}
(\partial_z - T_a) ( \partial_z - T_A)v= \partial_z^2v - T_{\beta}\nabla \partial_zv + T_{\alpha} \Delta v + F_3, \qquad F_3=R_0v + R_1v, 
\end{equation}
where
$$
R_0(z)\defn  T_{a(z)} T_{A(z)}   - T_\alpha\Delta, \quad
R_1(z)\defn - T_{\partial_z A(z)} .
$$
According to Theorem~\ref{theo:sc0}, applied with~$\rho=\eps$, 
$R_0(z)$ is of order $2- \eps$, uniformly in $z\in [-1,0]$. On the other hand, since
$$
\partial_z \rho \in L^\infty((-1,0); H^{s-\mez}), \qquad \partial_z^2 \rho \in L^\infty((-1,0); H^{s-\tdm}),
$$
according to \eqref{pr} we have
$$
\partial_z \alpha, \partial_z \beta \in L^\infty((-1,0); H^{s- \tdm}) \subset L^\infty((-1,0); C_*^{\eps -1}).
$$
Therefore $\partial_z A \in \Gamma^1_{\eps-1}({\mathbf{R}}^d\times [-1,0])$. 
As a consequence, using  Proposition~\ref{prop.niSbis}, we get that  
$R_1(z)$ is also of order $2- \eps$. We end up with
$$
\sup_{z\in [-1,0]}\lA R_0(z)\rA_{H^{\mu+2-\eps}\rightarrow H^{\mu}}
+\lA R_1(z)\rA_{H^{\mu+2-\eps}\rightarrow H^{\mu}}
\le \mathcal{F}\bigl( \| \eta \|_{H^{s+\mez}}\bigr).
$$

Now we notice that, given any symbol~$p$ and any function~$u$, by definition of paradifferential operators 
we have~$T_p u=T_p (1-\Psi(D_x))u$ 
for any Fourier multiplier~$(I-\Psi(D_x))$ such that~$\Psi(\xi)=0$ for~$\la \xi\ra\ge 1/2$. 
This means that we can replace~$\lA  v(z)\rA_{H^{\s+\tdm}}$ by~$\lA \partialx v(z)\rA_{H^{\s+\mez}}$. 
We thus obtain the desired result from Lemma~\ref{l3.8}.
\end{proof}

\subsubsection{Proof of Proposition~\ref{Linduction}}\label{S:PPLind}

We shall apply Proposition~\ref{prop:max}  twice. 
At first we apply it to the {\em forward} parabolic evolution equation 
$\partial_z u - T_a u=F$ (by definition~$\RE (-a) \ge c\la \xi\ra$). 
This requires an initial data on~$z=-1$ that might be chosen to be~$0$ by using a cut-off function, up to shrinking the interval~$I$. 
Next we apply it to the {\em backward\/} parabolic evolution equation 
$\partial_z u -T_A u =F$ (by definition~$\RE A \ge c\la \xi\ra$). 
This requires an initial data on~$z=0$ (which is given by our assumption on~$f$) and this requires 
also an estimate for the remainder term~$F$ which is given by means of the first step.

Suppose that~$\mathcal{H}_\sigma$ is satisfied and let~$I_0=[\zeta_0,0]$ with $\zeta_0\in (-1,0)$. Then 
\begin{equation*}
\lA \nabla_{x,z} v\rA_{X^\sigma(I_0)}\le
\mathcal{F}(\| \eta \|_{H^{s+\mez}})\big\{\lA f\rA_{H^{\sigma+1}}+\lA F_0\rA_{Y^\sigma([-1,0])}
+\lA \nabla_{x,z} v\rA_{X^{-\mez}([-1,0])} \big\}.
\end{equation*}
We shall prove that, for any~$0>\zeta_1> \zeta_0$, 
\begin{equation}\label{Indseps}
\begin{aligned}
\lA \nabla_{x,z} v\rA_{X^{\sigma+\eps}([\zeta_1,0])}\le
\mathcal{F}(\| \eta \|_{H^{s+\mez}})\big\{&\lA f\rA_{H^{\sigma+1+\eps}}+\lA F_0\rA_{Y^{\sigma+\eps}([-1,0])}\\
&\quad + \lA \nabla_{x,z} v\rA_{X^{-\mez}([-1,0])} \big\}.
\end{aligned}
\end{equation}

Introduce a cutoff function~$\chi$ such that 
$\chi\mid_{\zeta< \zeta_0} =0$, $\chi\mid_{\zeta> \zeta_1}=1$ . 
Set~$ w \defn \chi (z) (\partial_z - T_{A})v$. 
It follows from~\eqref{2m2.b} for~$v$ that $\partial_z w  -T_{a} w = F'$,
where
$$
F'=\chi(z)(F_0+F_1+F_2+F_3)+\chi'(z)(\partial_z - T_{A})v.
$$
We have already estimated~$F_1,F_2,F_3$ and~$F_0$ is given. 
We now turn to an estimate for~$(\partial_z-T_A)v$. 
According to~\eqref{esti:quant1} and~\eqref{boundsaA}, we have
\begin{equation*}
\lA T_{A}v\rA_{L^2(I_0;H^{\sigma+\mez})}\le \mathcal{F}(\| \eta \|_{H^{s+\mez}}) 
\lA \partialx v\rA_{L^2(I_0;H^{\sigma+\mez})} \le \mathcal{F}(\| \eta \|_{H^{s+\mez}})\lA \nabla_{x,z}v\rA_{X^{\sigma}(I_0)},
\end{equation*}
which implies that (since $L^2(I_0;H^{\sigma+\mez})\subset Y^{\sigma+1}(I_0)$) 
$$
\lA T_{A}v\rA_{Y^{\sigma+1}(I_0)}
\le \mathcal{F}(\| \eta \|_{H^{s+\mez}}) 
\lA \partialx v\rA_{L^\infty(I_0;H^{\sigma})}\le \mathcal{F}(\| \eta \|_{H^{s+\mez}})
\lA \nabla_{x,z}v\rA_{X^{\sigma}(I_0)}.
$$
Consequently
\begin{equation}
\lA\chi'(z)(\partial_z - T_{A})v\rA_{Y^{\sigma+1}(I_0)}\le \mathcal{F}(\| \eta \|_{H^{s+\mez}})
\lA \nabla_{x,z}v\rA_{X^{\sigma}(I_0)}\label{W-full-2}.
\end{equation}
Using the previous estimates for $F_1,F_2,F_3$, it follows 
from \eqref{W-full-2} that
\begin{equation}
\lA F'\rA_{Y^{\sigma+\eps}(I_0)}\le 
\mathcal{F}(\| \eta \|_{H^{s+\mez}})\lA \nabla_{x,z}v\rA_{X^{\sigma}(I_0)}
+\lA F_0\rA_{Y^{\sigma+\eps}(I_0)}.\label{F-full}
\end{equation}
Since~$w(x,z_0)=0$ 
and since~$a\in \Gamma^1_{\eps}$ 
satisfies~$\RE (-a(x,\xi))\ge c \la \xi\ra$, 
by using Proposition~\ref{prop:max} applied with~$J=I_0$, 
$\rho=\eps$ and~$r=\sigma+\eps$, we have 
\begin{equation*}
\lA w\rA_{X^{\sigma+\eps}(I_0)}\le \mathcal{F}(\| \eta \|_{H^{s+\mez}})
\lA F'\rA_{Y^{\sigma+\eps}(I_0)},
\end{equation*} 
and hence, using 
\eqref{F-full},
\begin{equation}\label{w-1}
\lA w\rA_{X^{\sigma+\eps}(I_0)}
\le \mathcal{F}(\| \eta \|_{H^{s+\mez}})\big\{\lA \nabla_{x,z}v\rA_{X^{\sigma}(I_0)}+\lA F_0\rA_{Y^{\sigma+\eps}(I_0)}\big\}.
\end{equation}

Now notice that on~$I_1\defn [\zeta_1,0]$ we have~$\chi=1$ so that
\begin{equation*}
\partial_z v  - T_{A} v = w \quad \text{for }z\in I_1.
\end{equation*}
Therefore the function~$\widetilde{v}$ defined by~$\widetilde v(x,z)=v(x,-z)$ satisfies
\begin{equation*}
\partial_z \widetilde v  +T_{\widetilde{A}} \widetilde v =- \widetilde w\quad \text{for }z\in \widetilde{I}_1=[0,-\zeta_1],
\end{equation*}
with obvious notations for~$\widetilde w$ and~$\widetilde A$. 
By using Proposition~\ref{prop:max} with~$J=\widetilde{I}_1$, noticing that 
$\widetilde v\arrowvert_{z=0}=v\arrowvert_{z=0}=f$, 
we obtain that
$$
\lA \widetilde v\rA_{X^{\sigma+1+\eps}(\widetilde{I}_1)}
\le \mathcal{F}(\| \eta \|_{H^{s+\mez}})\bigl(\lA f\rA_{H^{\sigma+1+\eps}}
+\lA \widetilde w\rA_{Y^{\sigma+1+\eps}(\widetilde{I}_1)}
\bigr).
$$
Using the obvious estimate
$$
\lA \widetilde w\rA_{Y^{\sigma+1+\eps}(\widetilde{I}_1)}
=\lA w\rA_{Y^{\sigma+1+\eps}(I_1)}\le
\lA w\rA_{L^2_z(I_1;H^{\sigma+\mez+\eps})}\le 
 \lA w\rA_{X^{\sigma+\eps}(I_1)},
$$
it follows from~\eqref{w-1} that
$$
\lA v\rA_{X^{\sigma+1+\eps}(I_1)}\\
\le 
\mathcal{F}(\| \eta \|_{H^{s+\mez}}) \bigl(\lA f\rA_{H^{\sigma+1+\eps}}
+\lA \nabla_{x,z}v\rA_{X^{\sigma}(I_0)}
+\lA F_0\rA_{Y^{\sigma+\eps}(I_0)}\bigr).
$$
We easily estimate~$\partial_z v$ directly from~$\partial_z v  = T_{A} v + w$ 
(by using \eqref{w-1} and the fact that~$T_A$ is an operator of order~$1$). 
This completes the proof of~\eqref{Indseps}. 

This proves that if~$\mathcal{H}_\sigma$ is satisfied then~$\mathcal{H}_{\sigma+\eps}$
is satisfied and hence concludes the proof of Proposition~\ref{Linduction} 
(and hence the proof of Proposition~\ref{p1}).

\subsubsection{Proof of Theorem~\ref{coro:estiDN}}\label{s:estiDN}
Let~$v$ be the solution of~\eqref{dnint1} with data~$v\arrowvert_{z=0}=f$. 
By definition of the Dirichlet--Neumann operator we have
\begin{equation}\label{r3.13}
G(\eta) f= \frac{1 +|\partialx \rho |^2}{\partial_z \rho} \partial_z v  -   \partialx \rho  \cdot\partialx v
\big\arrowvert_{z=0}.
\end{equation}
Now, by applying Proposition~\ref{p1} with~$F_0=0$ and Remark~\ref{rema:321}, 
we find that if~$v$ solves \eqref{dnint1}, then for any~$I \Subset  (-1, 0]$,  
\begin{equation}\label{oubli}
\lA \nabla_{x,z} v\rA_{X^{\sigma-1}(I)}
\le \mathcal{F}(\| \eta \|_{H^{s+\mez}})\lA f\rA_{H^\sigma}.
\end{equation}
According to~\eqref{eq.rho1} and~\eqref{prS2},
we obtain that
$$
\lA \frac{1 +|\partialx \rho |^2}{\partial_z \rho} \partial_z v  - 
\partialx \rho\cdot\partialx v\rA_{C^0 ([z_0,0];H^{\sigma-1})}
\le \mathcal{F}(\| \eta \|_{H^{s+\mez}})\lA f\rA_{H^\sigma}.
$$ 
As a result, taking the trace on~$z=0$ immediately implies the desired result~\eqref{ests+2}.

\subsubsection{Proof of Proposition~\ref{coro:paraDN1}}\label{s:paraDN1}
Let~$1/2\le \sigma_0 \leq s$. It follows from~\eqref{w-1} applied with~$\sigma=\sigma_0-1$ and~$F_0=0$ that
\begin{equation*}
\lA \chi(z)( \partial_z v-T_A v) \rA_{X^{\sigma_0-1+\eps}(I_0)}\le 
\mathcal{F}(\| \eta \|_{H^{s+\mez}})\big\{\lA \nabla_{x,z}v\rA_{X^{\sigma_0-1}(I_0)}\big\},
\end{equation*}
for some cut-off function~$\chi$ such that~$\chi(0)=1$. 
By using Proposition~\ref{p1}, we thus obtain
\begin{equation}\label{3.56ter}
\lA \partial_z v-T_{A}v \arrowvert_{z=0} \rA_{H^{\sigma_0-1+\eps}}\le 
\mathcal{F}\bigl( \| \eta \|_{H^{s+\mez}}\bigr)\lA f\rA_{H^{\sigma_0}}.
\end{equation}
The previous estimate allows us to express the ``normal'' derivative~$\partial_z v$ in 
terms of the tangential derivatives. Which is the main step to 
paralinearize the Dirichlet-Neumann operator. 

Now, as mentioned above, by definition of~$v$,
\begin{equation*}
G(\eta) f= \frac{1 +|\partialx \rho |^2}{\partial_z \rho} \partial_z v  - \partialx \rho \cdot\partialx v
\big\arrowvert_{z=0}.
\end{equation*}
Set
$$
\zeta_1 \defn  \frac{1 +|\partialx \rho |^2}{\partial_z \rho},\quad \zeta_2 \defn \partialx \rho .
$$
According to~\eqref{eq.rho1}, 
\begin{equation}\label{3.56bis}
\lA \zeta_1-\frac{1}{h}\rA_{C^0_z([-1,0];H^{s-\mez}_x)}+
\lA \zeta_2\rA_{C^0_z([-1,0];H^{s-\mez}_x)}\le \mathcal{F}(\| \eta \|_{H^{s+\mez}}).
\end{equation}
Let 
$$
R'=\zeta_1\partial_z v
-\zeta_2 \cdot\partialx v -(T_{\zeta_1}  \partial_z v  
- T_{\zeta_2}\partialx v).
$$
Since~$\eps \le \frac 1 2$ and $\eps<s- \frac 1 2 - \frac d 2$, 
we verify that Proposition~\ref{lemPa} applies with 
$$
\gamma=\sigma_0-1+\eps,\quad r=s-\mez,\quad \mu=\sigma_0-1,
$$
which, according to~\eqref{3.56bis} and \eqref{oubli}, implies 
$$
\lA R'\rA_{C^0(I;H^{\sigma_0-1+\eps})}\le 
\mathcal{F}\bigl( \| \eta \|_{H^{s+\mez}}\bigr)\lA f\rA_{H^{\sigma_0}}.
$$
Furthermore, according to~\eqref{3.56ter} and~\eqref{3.56bis}, we obtain
\begin{equation*}
T_{\zeta_1}  \partial_z v  
- T_{\zeta_2}\partialx v\big\arrowvert_{z=0}
-(T_{\zeta_1}   T_A  v -T_{i\zeta_2\cdot \xi}  v \big\arrowvert_{z=0} )=R'',
\end{equation*}
with 
~$$\lA R''\rA_{H^{\sigma_0-1+\eps}}\leq \mathcal{F}\bigl( \| \eta \|_{H^{s+\mez}}\bigr)\lA f\rA_{H^{\sigma_0}}.
$$ Finally,  thanks to \eqref{esti:quant2},~\eqref{3.56bis} and~\eqref{boundsaA}, we have
\begin{align*}
\lA 
T_{\zeta_1(z)}   T_{A(z)} -T_{\zeta_1(z) A(z)}\rA_{ H^{\sigma_0}\rightarrow H^{\sigma_0-\mez}}
&\les \lA \zeta_1(z)\rA_{W^{\eps,\infty}} \mathcal{M}^1_{\eps}(A) 
\le \mathcal{F}(\| \eta \|_{H^{s+\mez}}),
\end{align*}
and hence
$$
G(\eta)f=T_{\zeta_1 A}  v -T_{i\zeta_2\cdot \xi}  v \big\arrowvert_{z=0} +R(\eta)f
$$
where 
$$
\lA R(\eta)f \rA_{H^{\sigma_0-1+\eps}}
\le \mathcal{F}(\| \eta \|_{H^{s+\mez}})  \lA f\rA_{H^{\sigma_0}}.
$$
Let 
\begin{equation*}
\lambda=\frac{1 +|\partialx \rho |^2}{\partial_z \rho} A -i  \partialx \rho
\cdot  \xi\big\arrowvert_{z=0}= \sqrt{(1+|\partialx \eta(x)|^2)| \xi |^2
-\bigl(\partialx \eta (x)\cdot \xi\bigr)^2}.
\end{equation*}
Then 
$$
G(\eta)f=T_{\lambda} f +R(\eta)f,
$$
which concludes the  proof of Proposition~\ref{coro:paraDN1}.

\section{A priori estimates in Sobolev spaces}\label{S:6}
In this section, we shall prove {\em a priori} estimates on smooth solutions on a fixed time interval $[0,T]$. 
Recall that the system reads
\begin{equation}\label{systema}
\left\{
\begin{aligned}
&\partial_{t}\eta-G(\eta)\psi=0,\\[0.5ex]
&\partial_{t}\psi+g \eta
+ \frac{1}{2}\la\partialx \psi\ra^2  -\frac{1}{2}
\frac{\bigl(\partialx  \eta\cdot\partialx \psi +G(\eta) \psi \bigr)^2}{1+|\partialx  \eta|^2}
= 0.
\end{aligned}
\right.
\end{equation}
As already mentioned, we work with the unknowns~$B=B(\eta,\psi)$ and~$V=V(\eta,\psi)$ defined by 
\begin{equation}\label{first:4}
B\defn \frac{\partialx \eta\cdot\partialx\psi+ G(\eta)\psi}{1+|\partialx  \eta|^2},
\qquad 
V\defn 
\partialx \psi -B\partialx\eta.
\end{equation}
It follows from Theorem~\ref{coro:estiDN} that, 
for all~$s>1+d/2$ and all~$(\eta,\psi)\in H^{s+\mez}$,~$\B$ and~$V$ are well defined and belong to 
$H^{s-\mez}$. 
Moreover, we shall prove that if they belong initially to~$H^{s}$ then this regularity is propagated by the equation. 
We shall prove estimates 
in terms of
\begin{equation}\label{defi:MsZr}
\begin{aligned}
M_s(T)&\defn \sup_{\tau\in [0,T]}
\lA (\psi(\tau),B(\tau),V(\tau),\eta(\tau))\rA_{H^{s+\mez}\times H^s\times H^s\times H^{s+\mez}},\\
M_{s,0}&\defn \lA (\psi(0),B(0),V(0),\eta(0))\rA_{H^{s+\mez}\times H^s\times H^s\times H^{s+\mez}}.
\end{aligned}
\end{equation}

The main result of this section is the following proposition.
\begin{prop}\label{prop:apriori}
Let~$d\ge 1$ and consider $s>1+\frac{d}{2}$. 
Consider a fluid domain such that, there exists~$h>0$ such that for all~$t\in [0,T]$,
\begin{equation}\label{hypt2}
\left\{ (x,y)\in {\mathbf{R}}^d\times \xR \, :\, \eta(t,x)-h < y < \eta(t,x)\right\} \subset \Omega(t).
\end{equation}
Assume that for any~$t \in [0,T]$, 
$$
\ma (t,x)\ge c_0,
$$
for some given positive constant~$c_0$. 
Then, there exists a non-decreasing 
function~$\mathcal{F}\colon \xR^+\rightarrow\xR^+$ such that, for all~$T\in (0,1]$ 
and all smooth solution 
$(\eta,\psi)$ of \eqref{systema} defined on the time interval~$[0,T]$, there holds
\begin{equation}\label{apriori}
M_s(T)\le \mathcal{F}\bigl(\mathcal{F}(M_{s,0})+T\mathcal{F}\bigl(M_s(T)
\bigr)\bigr).
\end{equation}
\end{prop}
\begin{rema}
The assumption \eqref{hypt2} holds provided that it holds initially at time $0$ and
$\| \eta - \eta\mid_{t=0}\|_{H^{s+ \mez}}\leq \epsilon$, 
for some small enough positive constant~$\epsilon$.
\end{rema}

\subsection{A new formulation}\label{formulation}
Since we consider low regularity solutions, 
various 
cancellations have to be used. 
We found that these cancellations are most easily seen by working with the incompressible Euler equation directly, and hence 
we do not use the Zakharov formulation. This means that 
we begin with a new formulation of the water waves system which involves the following unknowns
\begin{equation}\label{defi:unknowns}
\zeta = \partialx \eta, \quad 
\B=\partial_y\phi \eval ,  \quad 
V=\nabla_x \phi\eval , \quad \ma=-\partial_y P \eval ,
\end{equation}
where recall that~$\phi$ is the velocity potential and the pressure~$P=P(t,x,y)$ is given by
\begin{equation}\label{defi:P}
-P=\partial_t \phi+\mez \la \nabla_{x,y}\phi\ra^2+gy.
\end{equation}

\begin{prop}\label{prop:newS}Let~$s>\mez +\frac{d}{2}$. We have 
\begin{align}
(\partial_{t}+V\cdot\partialx)\B&=\ma-g,\label{eq:B}\\
(\partial_t+V\cdot\partialx)V+\ma\zeta&=0,\label{eq:V}\\
(\partial_{t}+V\cdot\partialx)\zeta&=G(\eta)V+ \zeta G(\eta)\B +\gamma,\label{eq:zeta}
\end{align}
where the remainder term~$\gamma=\gamma(\eta,\psi,V)$ satisfies the following estimate : 
\begin{equation}\label{p42.1}
\lA \gamma\rA_{H^{s-\mez}}\le \mathcal{F}(\lA (\eta,\psi,V)\rA_{H^{s+\mez}\times H^s\times H^s}).
\end{equation}
\end{prop}
\begin{rema}
In the case~$\Gamma = \emptyset$, one can see that at least formally~$\gamma =0$. 
\end{rema}
\begin{proof}
Directly from the equality $\partial_t\eta=G(\eta)\psi$ and the definition of $B$ and $V$ (see \eqref{first:4}), we have
\begin{equation}\label{eq:etaB}
 \partial_t \eta +V\cdot \nabla\eta = B.
\end{equation}
Then, for any function~$f=f(t,x,y)$, by using the chain rule, we check that, with $\nabla=\nabla_x$,
\begin{align*}
(\partial_t + V\cdot \nabla) ( f \arrowvert_{y=\eta(t,x)})
&=(\partial_t + V\cdot \nabla) f(t,x,\eta(t,x))\\
&=\bigl[\partial_t f +\nabla \phi\cdot\nabla f 
+\partial_y f (\partial_t \eta + V\cdot \nabla \eta)\bigr] 
\big\arrowvert_{y=\eta(t,x)}\\
&=\bigl[ (\partial_t +\nabla_{x,y} \phi \cdot \nabla_{x,y}) f \bigr] \big\arrowvert_{y=\eta(t,x)},
\end{align*}
where we used \eqref{eq:etaB} as well as the identity $B=\partial_y \phi \big\arrowvert_{y=\eta(t,x)}$. 
Applying~$\partial_y~$ to \eqref{defi:P}, this identity yields~\eqref{eq:B}. 
On the other hand, applying $\partial_{x_k}$ to \eqref{defi:P}, the previous identity gives
$$
(\partial_t+V\cdot\nabla)V+ (\nabla P)\arrowvert_{y=\eta}=0.
$$
Since~$P\arrowvert_{y=\eta}=0$, we have
$$
0=\nabla( P\arrowvert_{y=\eta} )
= (\nabla P)\arrowvert_{y=\eta}+(\partial_y P)\arrowvert_{y=\eta}\nabla \eta,
$$
which yields \eqref{eq:V}.

To derive equation \eqref{eq:zeta} on~$\zeta\defn \nabla \eta$ we 
start from 
$$
\partial_t \eta = \B-V\cdot \partialx \eta.
$$
Differentiating with respect to~$x_i$ (for~$i=1,\ldots,d$) we find that~$\partial_i \eta=\partial_{x_i}\eta$ satisfies
\begin{equation}\label{identity:zeta}
(\partial_{t}+V\cdot\partialx)\partial_i \eta=\partial_i \B -\sum_{j=1}^d \partial_i V_j \partial_j \eta.
\end{equation}
Directly from the definitions of~$B$ and~$V$ 
($B=\partial_y\phi\arrowvert_{y=\eta}$,~$V=\nabla\phi\arrowvert_{y=\eta}$), and using the chain rule, 
we compute that
\begin{align*}
\partial_i B - \sum_{j=1}^d \partial_i V_j \partial_j\eta 
 &= \big[\partial_i \partial_y \phi +\partial_i \eta \partial_y^2\phi \big] \big\arrowvert_{y=\eta} 
 -\sum_{j=1}^d 
 \partial_j \eta \big[\partial_i \partial_j \phi 
 +\partial_i \eta \partial_j\partial_y\phi \big] \big\arrowvert_{y=\eta} \\
&= \big[\partial_y \partial_i  \phi 
- \sum_{j=1}^d \partial_j \eta \partial_i \partial_j \phi\big] \big\arrowvert_{y=\eta} 
+\partial_i \eta \big[ \partial_y^2 \phi -\sum_{j=1}^d 
\partial_j \eta \partial_j\partial_y\phi \big] \big\arrowvert_{y=\eta} .
\end{align*}
Therefore
\begin{equation}
(\partial_{t}+V\cdot\partialx)\partial_i \eta =  \big[\partial_y \partial_i  \phi 
-  \partialx \eta \cdot \partialx \partial_i \phi\big] \big\arrowvert_{y=\eta} 
+\partial_i \eta \big[ \partial_y (\partial_y \phi) -\partialx \eta \cdot \partialx\partial_y\phi \big] 
\big\arrowvert_{y=\eta} .
\end{equation}
 
 Let $\theta_i$ be the variational solution of the problem 
$$
\Delta_{x,y}\theta_i =0 \text{ in }\Omega, \quad \theta_i \arrowvert_{y=\eta} = V_i, 
\quad \partial_n \theta_i =0 \text{ on }\Gamma.
$$
 Then
$$
G(\eta)V_i = 
\sqrt {1+ \vert \nabla \eta \vert^2} \frac{\partial \theta_i}{\partial n}\arrowvert_{y=\eta} 
= (\partial_y\theta_i - \nabla \eta \cdot \nabla\theta_i)\arrowvert_{y=\eta}.
$$
Then we write
\begin{equation}\label{Ri}
 (\partial_y - \nabla \eta \cdot \nabla)\partial_i \phi \arrowvert_{y=\eta}
 = G(\eta)  V_i +R_i, \quad \text{where} 
 \quad   R_i = (\partial_y - \nabla \eta \cdot \nabla)(\partial_i \phi - \theta_i)\arrowvert_{y=\eta}.
\end{equation}
Notice that 
$$\Delta_{x,y}\partial_i \phi =0 \text{ in }\Omega, \quad \partial_i \phi\arrowvert_{y=\eta} = V_i.$$
However, in general  $ \partial_n \partial_i \phi \neq 0 \text{ on }\Gamma.$ Our goal is to show that $\partial_i \phi - \theta_i$ has a better regularity.
 Due to the presence of the bottom we have to localize the problem near~$\Sigma.$

Let~$\chi_0 \in C^\infty(\xR)$,~$\eta_1\in H^\infty({\mathbf{R}}^d)$ 
be such that~$\chi_0(z) = 1$ if~$z\geq 0 , \chi_0(z) =0$ if~$z \leq -\frac{1}{4}$ and
$$
\eta(x) - \frac{h}{4} \leq \eta_1(x) \leq \eta(x) - \frac{h}{5}. 
$$
Set 
$$
U_i(x,y) = \chi_0\big(\frac{y- \eta_1(x)}{h}\big)(\partial_i \phi - \theta_i)(x,y).
$$
We see easily that 
$R_i = (\partial_y - \nabla \eta \cdot \nabla)U_i\arrowvert_{y=\eta}$. 
Moreover~$U_i$ satisfies the equation
\begin{equation}
 \Delta_{x,y} U_i   = \big[ \Delta_{x,y},  \chi_0\big(\frac{y- \eta_1(x)}{h}\big) \big](\partial_i \phi - \theta_i):= F_i 
\end{equation}
and, with a slight change of notation, we have
\begin{equation}
\supp F_i \subset  S_{\mez,\frac{1}{5}}:= \big \{(x,y): x\in {\mathbf{R}}^d, \eta(x) - \frac{h}{2} \leq y \leq \eta(x) - \frac{h}{5} \big \}.
\end{equation}
 Moreover by ellipticity (see  Lemma~\ref{regell}) we have for all~$\alpha \in \xN^{d+1}$, 
\begin{equation}\label{regFi}
   \|D_{x,y}^\alpha  F_i \|_{ L^\infty(S_{\mez,\frac{1}{5}} )
   \cap L^2(S_{\mez,\frac{1}{5}} )} \leq C_\alpha  \| (V , B)\|_{H^{\mez }\times H^{\mez}}.
\end{equation}
Now we set~$  y = \rho(x,z) =  (1+z)e^{\delta z\langle D_x \rangle }\eta(x) -z \big\{e^{\delta(1+ z)\langle D_x \rangle }\eta(x) -h \big\} ~$ 
and~$\widetilde{g_i}(x,z) = g_i(x, \rho(x,z)).$
Since we have taken~$\delta \Vert \eta \Vert_{H^{s+\mez}}$ small,  
it is easy to see that on the image of~$  S_{\mez,\frac{1}{5}}$ one has~$ -h\leq z \leq -\frac{h}{10}.$ 
Now, according to section~\ref{s:flattening},~$\widetilde{U_i}$ is a solution of the problem 
$$
(\partial_z^2 + \alpha \Delta
+ \beta \cdot \nabla \partial_z - \gamma \partial_z) \widetilde{U_i} 
= \frac{(\partial_z \rho)^2}{1+ \vert \nabla \rho \vert^2} \widetilde{F_i}.
$$
Due to the exponential smoothing and to \eqref{regFi}, on the support of~$\widetilde{F_i}$ 
the right hand side of the above equation belongs in fact to~$C^0_z( (-h,0); H^\infty({\mathbf{R}}^d)).$ 
In particular we can apply Proposition~\ref{p1} with~$f=0.$ 
It follows that
$$
\|\nabla_{x,z} \widetilde{U_i}\|_{C^0([z_0,0]; H^{s-\mez}({\mathbf{R}}^d))}
\le \mathcal{F}(\| \eta \|_{H^{s+\mez}})\bigl(\| \widetilde F_i\|_{Y^\sigma((-1,0))}
+ \| \nabla_{x,z} \widetilde U_i\|_{X^{-\mez}((-1,0))} \bigr)~.
$$
Notice that according to the constructions of variational solutions 
the norm of~$ \widetilde U_i$ in $X^{-\mez}((-1,0))$ 
is bounded by 
$$
\mathcal{F}(\Vert \eta \Vert_{H^{s + \mez}}) \bigl(\|\psi\|_{H^\mez}+ \|V_i\|_{H^{\mez}}\bigr).
$$
Since
$$
R_i 
=\Big[ \big( \frac{1+ \vert \nabla \eta \vert^2}{1+\delta\langle D_x \rangle \eta} \partial_z 
- \nabla \eta \cdot \nabla  \big)\widetilde{U_i}\Big] \Big \arrowvert_{z=0},
$$
we deduce that 
\begin{equation*}
\Vert R_i \Vert_{H^{s-\mez}} \leq \mathcal{F}(\Vert \eta \Vert_{H^{s + \mez}})
\big( \Vert \psi \Vert_{H^\mez} + \Vert V_i \Vert_{H^\mez} \big) \\
\leq \mathcal{F}\big(\Vert \eta \Vert_{H^{s + \mez}}, \Vert \psi \Vert_{H^{s}}, \Vert V \Vert_{H^{s }} \big),
\end{equation*}
since~$s>\mez + \frac{d}{2}.$ 
We use exactly the same argument to show that
\begin{equation}
(\partial_y - \nabla \eta \cdot \nabla)\partial_y \phi \arrowvert_{y=\eta}= G(\eta)B +R_0,  
\end{equation}
where~$R_0$ satisfies the same estimate as~$R_i$. 
This completes the proof. 
\end{proof}
Following the same lines, we have the following relation between~$V$ and~$B$.
\begin{prop}\label{cancellation}
Let~$s> \mez + \frac d 2$. Then we have~$G(\eta){B}=-\cnx V +\gamma$ where 
$$
\|\gamma \|_{H^{s- \mez}} \leq \mathcal{F}( \|( \eta,  V, B)\|_{H^{s+ \mez} \times H^{\mez}\times H^\mez}).
$$
\end{prop}
\begin{proof}
Recall that, by definition, $B=\partial_y\phi\arrowvert_{y=\eta}$ and $V=\partialx\phi\arrowvert_{y=\eta}$. 
Let $\theta$ be the variational solution to the problem
$$
\Deltayx \theta=0,\quad \theta\arrowvert_{y=\eta}=B,\quad \partial_n \theta\arrowvert_{\Gamma}=0.
$$
Then $G(\eta)B=( \partial_y \theta-\nabla\eta\cdot\nabla\theta)\big\arrowvert_{y=\eta}$. 
Now let $\widetilde{\theta}=\partial_y\phi$. We claim that 
$$
( \partial_y\widetilde{\theta}-\nabla\eta\cdot\nabla\widetilde{\theta})\big\arrowvert_{y=\eta}=-\cnx V.
$$
Indeed, on the one hand we have
$$
( \partial_y\widetilde{\theta}-\nabla\eta\cdot\nabla\widetilde{\theta})=
\partial_y^2\phi-\nabla\eta\cdot \nabla\partial_y \phi,
$$
and on the other hand
$$
\cnx V= \sum_{1\le i\le d}\partial_{x_i}V=\Bigl(\sum_{1\le i\le d} \partial_i^2\phi+\nabla\eta\cdot \partial_y\phi\Bigr)\Big\arrowvert_{y=\eta}.
$$
Then our claim follows from the fact that $\Deltayx\phi=0$. Now we have
$$
\Deltayx(\theta-\widetilde{\theta})=0,\quad (\theta-\widetilde{\theta})\arrowvert_{y=\eta}=0,
$$
so, as in the proof of Proposition~\ref{prop:newS}, we deduce from Proposition~\ref{p1} that
$$
\lA (\partial_y-\nabla\eta\cdot\nabla)(\theta-\widetilde{\theta})\rA_{H^{s-\mez}}\le 
\mathcal{F}( \|( \eta,  V, B)\|_{H^{s+ \mez} \times H^{\mez}\times H^\mez}),
$$
which is the desired result.
\end{proof}

\subsection{Estimates for the Taylor coefficient}\label{S:tameP}
In this paragraph, we prove several estimates for the Taylor 
coefficient.
\begin{prop}\label{prop:ma}
Let~$d\ge 1$ and $s>1+\frac{d}{2}$. 
There exists a non-decreasing 
function~$\mathcal{F}\colon\xR^{+}\rightarrow \xR^+$ such that, for all~$t\in [0,T]$,
\begin{equation}
\lA \ma(t)-g\rA_{H^{s-\mez}}\le 
\mathcal{F}\bigl( \lA  (\eta, \psi,V,B)(t)\rA_{H^{s+\mez}\times H^{s+\mez}\times 
H^s\times H^s}\bigr).\label{esti:a1}
\end{equation}
For $0<\eps <s-1-d/2$, there exists a non-decreasing 
function $\mathcal{F}$ such that, 
\begin{equation}
\lA (\partial_t \ma +V\cdot \partialx \ma)(t)\rA_{C^{\eps}}
\le \mathcal{F}\bigl( \lA  (\eta, \psi,V,B)(t)\rA_{H^{s+\mez}\times H^{s+\mez}\times 
H^s\times H^s}\bigr) 
.\label{esti:a2}
\end{equation}
\end{prop}

Recall that 
$\ma=-\partial_y P\arrowvert_{y=\eta}$ 
where 
$$
P=P(t,x,y)=-\bigl(\partial_t \phi+\mez \la \nabla_{x}\phi\ra^2+\mez (\partial_y\phi)^2+gy \bigr).
$$
The basic idea is that one should be able to easily estimate~$P$ since 
it satisfies an elliptic equation. Indeed, 
since~$\Delta_{x,y}\phi=0$, we have
$$
\Delta_{x,y}P= -\la \nabla_{x,y}^2\phi\ra^2.
$$
Moreover, by assumption we have~$P=0$ on the free surface~$\{y=\eta(t,x)\}$. 
Yet, this requires some preparation because, as we shall see, the regularity of~$P$ is {\em not\/} 
given by the right-hand side in the elliptic equation above. Instead the regularity of~$P$ 
is limited by the regularity of the domain (i.e.\ the regularity of the function~$\eta$).

Hereafter, since the time variable is fixed, we shall skip it. 
We use the change of variables~$(x,z)\mapsto (x,\rho(x,z))$ introduced in~$\S$\ref{s:flattening}. 
Introduce~$\varphi$ and~$\wp$ given by
$$
\varphi(x,z)=\phi(x,\rho(x,z)),\quad \wp(x,z)=P(x,\rho(x,z))+g\rho(x,z),
$$
and notice that 
$$
\ma-g = - \frac 1 { \partial_z \rho} \partial_z  \wp\mid_{z=0}.
$$
The first elementary step is to compute the equation satisfied
by the new unknown~$v$ in~$\{z<0\}$ as well as the boundary conditions
on~$\{z=0\}$. Set (see \eqref{Lambda})
$$
\Lambda=(\Lambda_1,\Lambda_2),\quad 
\Lambda_1=\frac{1}{\partial_z\rho}\partial_z, \quad
\Lambda_2 = \nabla -\frac{\partialx \rho}{\partial_z \rho}\partial_z.
$$
We find that
\begin{align*}
& (\Lambda_1^2 + \Lambda_2^2)\va=0\text{ in }\quad - 1 <z<0,\\
&  (\Lambda_1^2 + \Lambda_2^2)\wp=-\la \Lambda^2 \va\ra^2 \text{ in }\quad - 1 < z<0,\\
&  (\Lambda_1^2 + \Lambda_2^2)\rho=0  \text{ in }\quad z<0,\\ 
\intertext{together with the boundary conditions}
&\wp=g\eta ,\quad \Lambda_1\wp=g -\ma \quad\text{on }\, z=0,\\
&\Lambda_2\varphi =V ,\quad \Lambda_1 \varphi=B,\quad\text{on }\, z=0.
\end{align*}
According to~\eqref{est:phi} and Remark~\ref{rema:321}, we have the {\em a priori estimate}
$$
\| \nabla_{x,z} \varphi \|_{X^{-\mez} (-\frac h 2, 0) }\leq \mathcal{F} ( \|\eta\|_{H^{s+ \mez}} )\|\psi\|_{H^{\mez}},
$$
while according to Proposition~\ref{p1}
\begin{equation}\label{eq4.20}
\|\nabla_{x,z}\wp\|_{X^{-\mez} (-1, 0) }\leq 
\mathcal{F}  \Bigl( \| \mathcal{R}\|_{X^{\mez} (-1, 0) } +  \| | \nabla \varphi|^2\|_{ X^{\mez} (-1, 0) }\Bigr)
\leq \mathcal{F} \bigl( \|(\eta, \psi)\|_{H^{s+ \mez}}\bigr),
\end{equation}
where $\mathcal{R}(x,z)= R(x, \rho(x,z))$ and $R$ is defined in ~Definition~\ref{def.4.1}. 

Expanding~$\Lambda_1^2 + \Lambda_2^2$, we thus find that~$\wp$ solves 
\begin{equation}
\begin{aligned}
&\partial_z^2 \wp +\alpha\Delta \wp + \beta \cdot\partialx\partial_z \wp  - \gamma \partial_z \wp=F_0(x,z)&& \text{for }z<0,\\
&\wp=g\eta && \text{on }z=0,
\end{aligned}
\end{equation}
where~$\alpha,\beta,\gamma$ are as above (see~\eqref{dnint1bis}) and where
\begin{equation}\label{defi:F0Lambda}
F_0=-\alpha \la \Lambda^2 \va\ra^2. 
\end{equation}
Our first task is to estimate the source term~$F_0$. 
\begin{lemm}\label{es:F0p}
Let~$d\ge 1$ and $s>1+d/2$. 
Then there exists~$z_0 <0$ such that 
\begin{equation*}
\lA F_0\rA_{L^1([z_0, 0];H^{s-\mez})}
\le \mathcal{F}\bigl(\lA (\eta,\psi,V,B)\rA_{H^{s+\mez}\times H^{s+\mez}\times H^s\times H^s}\bigr).
\end{equation*}
\end{lemm}
\begin{proof}
Since~$[\Lambda_1, \Lambda_2]=0$ we have 
$$
(\Lambda_1^2 + \Lambda_2^2)\Lambda_2\va=0,
\quad (\Lambda_1^2 + \Lambda_2^2)\Lambda_1\va=0.
$$
Since~$\Lambda_2 \varphi\arrowvert_{z=0}=V$ 
and~$\Lambda_1 \varphi\arrowvert_{z=0}=B$, it follows from 
Proposition~\ref{p1} (and Theorem~\ref{th.23} 
which guarantees that~$\nabla_{x,z}\varphi\in X^{-\mez} (z_0, 0)$) 
that
\begin{equation*}
\lA \nabla_{x,z} \Lambda_j \varphi\rA_{X^{s-1}([z_0,0])}
\le \mathcal{F}\bigl(\lA (\eta,\psi,V,B)\rA_{H^{s+\mez}\times H^{s+\mez}\times H^s\times H^s}\bigr).
\end{equation*}
By using the   estimate~\eqref{eq.rho1}
\begin{equation*}
\lA \nabla_{x}\rho \rA_{C^0([z_0, 0];H^{s-\mez})}
+\lA \partial_z \rho-  h  \rA_{C^0([z_0,0]; H^{s-\mez})}
\les \|\eta \|_{H^{s+\mez}},
\end{equation*}
and the product rule in Sobolev spaces, we obtain
\begin{equation}\label{D2var}
\lA \Lambda_j \Lambda_k\varphi\rA_{L^2([z_0,0];H^{s-\mez})}
\le \mathcal{F}\bigl(\lA (\eta,\psi,V,B)\rA_{H^{s+\mez}\times H^{s+\mez}\times H^s\times H^s}\bigr).
\end{equation}
Since~$H^{s-\mez}$ is an algebra, according to Lemma~\ref{lem.3.25}, we obtain 
\begin{multline}\label{D2varbis}
\lA F_0\rA_{L^1([z_0,0];H^{s-\mez})}\les 
\left(1+\lA \alpha- h^2 \rA_{C^0([z_0,0];H^{s-\mez})}\right) 
\lA\Lambda_j \Lambda_k\varphi\rA_{L^2([z_0,0];H^{s-\mez})} ^2 \\
\le \mathcal{F}\bigl(\lA (\eta,\psi,V,B)\rA_{H^{s+\mez}\times H^{s+\mez}\times H^s\times H^s}\bigr).
\end{multline}
This completes the proof.
\end{proof}
It follows from Lemma~\ref{es:F0p} 
and Proposition~\ref{p1} applied with~$\sigma=s-1/2$ that there exists~$z_0$ such that
\begin{equation}\label{ep1}
\lA \nabla_{x,z} \wp\rA_{X^{s-\mez}([z_0,0])}
\le \mathcal{F}\bigl(\lA (\eta, \psi)\rA_{H^{s+\mez}},\lA F_0\rA_{L^1([z_0,0];H^{s-\mez})}
\bigr),
\end{equation}
where we used the estimate~\eqref{eq4.20}. 
According to~\eqref{D2varbis}, this implies that
\begin{equation}\label{inwpCSob}
\lA \nabla_{x,z} \wp\rA_{X^{s-\mez}([z_0,0])}
\le \mathcal{F}\bigl(\lA (\eta,\psi,V,B)\rA_{H^{s+\mez}\times H^{s+\mez}\times H^s\times H^s}\bigr),
\end{equation}
which in turn implies that
~$\lA \ma-g\rA_{H^{s-\mez}}$ is bounded by a constant 
depending only on 
$\lA(\eta, \psi)\rA_{H^{s+\mez}}$ and~$\lA (V,B)\rA_{H^s}$. 
\subsection{Paralinearization of the system}
Introduce 
\begin{equation}\label{eq.U}
U=V+T_\zeta B.
\end{equation}
To clarify notations, let us mention that the~$i$th component ($i=1,\ldots,d$) 
of this vector valued unknown satisfies 
$U_i=V_i+T_{\partial_i\eta}B$. The new unknown~$U$ is related to 
what is called the good unknown of Alinhac in~\cite{AM,ABZ1,Alipara,Ali}. 

To estimate~$(U,\zeta)$ in Sobolev spaces, we want to estimate 
$(\lDx{s}U,\lDx{s-\mez}\zeta)$ in~$L^\infty([0,T];L^2\times L^2)$ 
where~$\lDx{}\defn (I-\Delta)^{1/2}$. 
However, for technical reasons, 
instead of working with~$(\lDx{s}U,\lDx{s-\mez}\zeta)$, 
it is more convenient to work with
\begin{equation}\label{defiUszetas}
\begin{aligned}
&U_s \defn \lDx{s} V+T_\zeta \lDx{s}\B,\\[0.5ex]
&\zeta_s\defn \lDx{s}\zeta.
\end{aligned}
\end{equation}

\begin{prop}\label{prop:csystem2}
Under the assumptions of Proposition~\ref{prop:apriori}, 
there exists a non decreasing function~$\mathcal{F}$ such that
\begin{align}
&(\partial_t+T_V\cdot\partialx)U_s+T_\ma\zeta_s=f_1,\label{premiere}\\
&(\partial_{t}+T_V\cdot\partialx)\zeta_s=T_\lambda U_s +f_2,\label{seconde}
\end{align}
where recall that~$\lambda$ is the symbol
$$
\lambda(t;x,\xi)\defn\sqrt{(1+\la\partialx\eta(t,x)\ra^2)\la\xi\ra^2-(\partialx\eta(t,x)\cdot\xi)^2},
$$
and where, for each time~$t\in [0,T]$, 
\begin{equation}\label{p49:estf}
\lA (f_1(t),f_2(t))\rA_{L^2\times H^{-\mez}}
\le \mathcal{F}\bigl( \lA \eta(t)\rA_{H^{s+\mez}},\lA (V,B)(t)\rA_{H^s}\bigr).
\end{equation}
\end{prop}
\begin{proof}
The proof is based on 
the paralinearization of the Dirichlet-Neumann operator 
(see~Proposition~\ref{coro:paraDN1}), 
the Bony's paralinearization formula for a product, 
some simple computations and the commutator estimate proved in Section~\ref{S:Dt}.

\noindent {\sc{Step }1:} Paralinearization of the equation 
$$
(\partial_t+V\cdot\partialx)V+\ma\zeta=0.
$$

We begin by proving 
\begin{lemm}We have  
\begin{equation}\label{parastep1a}
\begin{aligned}
&(\partial_t+T_V\cdot\partialx)V+T_{\ma}\zeta+T_\zeta (\partial_t+T_V\cdot\partialx)B=h_1 
\quad\text{with}\\
&\lA h_1\rA_{H^s}
\le \mathcal{F}\bigl( \| \eta \|_{H^{s+\mez}},\lA (V,B)\rA_{H^s}\bigr).
\end{aligned}
\end{equation}
\end{lemm}
\begin{proof}
 Using \eqref{Bony3} and \eqref{esti:quant1} we have 
$V\cdot \nabla V=T_V\cdot \nabla V +A_1$ where~$A_1=\sum_j T_{\partial_j V} V_j +R(\partial_j V, V_j)$ satisfies
\begin{align*}
\lA A_1\rA_{H^s} \les \lA \nabla V\rA_{L^\infty} \lA V\rA_{H^s}.
\end{align*}
Similarly,~$(\ma-g)\zeta=T_{\ma-g} \zeta+T_{\zeta}(\ma-g) +R(\zeta,\ma-g)$ where 
\begin{equation}\label{Rzetaa}
\lA R(\zeta,\ma-g)\rA_{H^{s}}\les \lA\zeta\rA_{H^{s-1/2}}\lA \ma-g\rA_{H^{s-1/2}},
\end{equation}
and where~$\lA \ma-g\rA_{H^{s-1/2}}$ is estimated by means of \eqref{esti:a1}. 

Since~$T_\zeta g=0$, by 
replacing~$a$ by~$g+(\partial_t\B+V\cdot\nabla \B)$ we obtain 
\begin{align*}
T_\zeta \ma &=T_\zeta (\partial_t\B+V\cdot\nabla \B)\\
&=T_\zeta (\partial_t\B+T_V\cdot\nabla \B) +T_\zeta(V-T_V)\cdot \nabla \B.
\end{align*}
As in the analysis of~$A_1$ above, we have
$$
\lA (V-T_V)\cdot \nabla\B \rA_{H^s} \les  \lA \nabla \B\rA_{L^\infty} \lA V\rA_{H^s}.
$$
Now we use~$\lA T_\zeta\rA_{H^s\rightarrow H^s}\les \lA \zeta\rA_{L^\infty}\les \| \eta \|_{H^{s+1/2}}$ (since~$s+1/2>1+d/2$) to obtain 
$$
\lA T_\zeta (V-T_V)\cdot \nabla\B \rA_{H^s} \les  \| \eta \|_{H^{s+\mez}} \lA \nabla \B\rA_{L^\infty} \lA V\rA_{H^s}.
$$
By Sobolev injection, this proves~\eqref{parastep1a}. 
\end{proof}

\noindent {\sc{Step }2.} We now commute~\eqref{parastep1a} with~$\lDx{s}=(I-\Delta)^{s/2}$. 
The paradifferential rule~\eqref{esti:quant2} implies that 
\begin{align*}
&\lA \left[ T_\ma ,\lDx{s}\right]  \rA_{H^{s-1/2}\rightarrow L^2}\les \lA \ma \rA_{W^{1/2,\infty}}\les 1+\lA \ma-g \rA_{H^{s-1/2}},\\
&\lA \left[ T_\zeta , \lDx{s}\right]  \rA_{H^{s-1/2}\rightarrow L^2}
\les \lA \zeta \rA_{W^{1/2,\infty}}\les \lA \zeta \rA_{H^{s-1/2}},\\
&\lA \left[ T_V\cdot \nabla , \lDx{s}\right]  \rA_{H^{s}\rightarrow L^2}
\les \lA V \rA_{W^{1,\infty}}\les \lA V \rA_{H^{s}}.
\end{align*}
Consequently, it easily follows from~\eqref{esti:a2} and~\eqref{parastep1a} that
\begin{equation*}
(\partial_t+T_V\cdot\partialx) \lDx{s} V
+T_{\ma} \lDx{s} \zeta+T_\zeta (\partial_t+T_V\cdot\partialx) \lDx{s} B=h_2
\end{equation*}
for some remainder~$h_2$ satisfying
$\lA h_2\rA_{L^2}
\le \mathcal{F}\bigl( \| \eta \|_{H^{s+\mez}},\lA (V,B)\rA_{H^s}\bigr)$. 

On the other hand, Lemma~\ref{lemm:Dt} implies that 
$$
\lA  [T_\zeta,\partial_t+T_V\cdot \nabla] \lDx{s} \B (t) \rA_{L^2}
\le
\mathcal{F}\bigl( \lA \eta(t)\rA_{H^{s+\mez}},\lA (V,B)(t)\rA_{H^s}\bigr).
$$
Here we have used the fact that the~$L^\infty$ norm of~$\partial_t \zeta+V\cdot\nabla \zeta$ is, since~$s>\frac 1 2 + \frac d 2$,  estimated 
by means of the identity~\eqref{identity:zeta}: 
\begin{align*}
\lA \partial_t \zeta+V\cdot\nabla \zeta \rA_{L^\infty}
&\les \lA \nabla \B\rA_{L^\infty}+\lA \zeta\rA_{L^\infty}\lA \nabla V \rA_{L^\infty}\\
&\les \lA \nabla \B\rA_{L^\infty}+\| \eta \|_{H^{s+\mez}}\lA \nabla V \rA_{L^\infty}.
\end{align*}

By combining the previous results we obtain
$$
(\partial_t +T_V\cdot\nabla )\bigl( \lDx{s} V+T_\zeta  \lDx{s}\B\bigr)+T_\ma  \lDx{s}\zeta 
=f_1
$$
where~$f_1$ satisfies the desired estimate~\eqref{p49:estf}. 

\noindent {\sc{Step }3.} Paralinearization of the equation 
$$
(\partial_{t}+V\cdot\partialx)\zeta=G(\eta)V+ \zeta G(\eta)\B+ \gamma.
$$
Writing~$(V-T_V)\cdot \nabla\zeta=T_{\nabla\zeta}\cdot V +\sum_{j=1}^d R(\partial_j\zeta,V_j)$ and using \eqref{Bony3} and 
\eqref{niS}, we obtain
\begin{equation}
\lA (V-T_V)\cdot \nabla\zeta \rA_{H^{s-\mez}} \les  \lA \nabla \zeta\rA_{C^{-\mez}_*} \lA V\rA_{H^s}\les 
\| \eta \|_{C^{\tdm}_*}\lA V \rA_{H^s}.
\end{equation}The key step is to paralinearize~$G(\eta)V+ \zeta G(\eta)\B$. 
This is where we use the analysis performed in the previous Section. 
By definition of~$R(\eta)=G(\eta)-T_{\lambda}$ we have
\begin{equation*}
G(\eta)V+ \zeta G(\eta)\B = T_\lambda U  +F_2(\eta,V,B),
\end{equation*}
where
\begin{equation}\label{sy:F1}
F_2=[T_\zeta,T_\lambda]B+R(\eta)V+\zeta R(\eta)B+(\zeta-T_\zeta)T_\lambda B.
\end{equation}
The commutator~$ [T_\zeta,T_\lambda]B$ is estimated by means of \eqref{esti:quant2} which implies that
$$
\lA [T_\zeta,T_\lambda]B\rA_{H^{s-\mez}}\les \left\{M_0^0(\zeta)M_{1/2}^1(\lambda)
+M_{1/2}^0(\zeta)M_{0}^1(\lambda)\right\} \lA B\rA_{H^s}.
$$
Since 
$M_{1/2}^0(\zeta)+M_{1/2}^1(\lambda)\le \K\bigl( \| \eta \|_{H^{s+\mez}}\bigr)$ 
we conclude that 
$$
\lA [T_\zeta,T_\lambda]B\rA_{H^{s-\mez}}
\le \K\bigl( \| \eta \|_{H^{s+\mez}},\lA B\rA_{H^s}\bigr).
$$
Moving to the estimate of the second and third terms 
in the right-hand side of~\eqref{sy:F1}, we use Proposition~\ref{coro:paraDN1} to obtain 
that the~$H^{s-\mez}$-norm of~$R(\eta)V$ and~$R(\eta)B$ satisfy
$$
\lA R(\eta(t))V (t)\rA_{H^{s-\mez}}+\lA R(\eta(t))\B(t)\rA_{H^{s-\mez}}\\
\le 
\mathcal{F}\bigl( \lA \eta(t)\rA_{H^{s+\mez}},\lA (V,\B)(t)\rA_{H^s}\bigr).
$$
Since 
$H^{s-\mez}$ is an algebra, the term~$\zeta R(\eta)\B$ satisfies the same estimate as~$R(\eta)\B$ does. It remains only to estimate 
$(\zeta-T_\zeta)T_\lambda \B$. To do so we write
$$
(\zeta-T_\zeta)T_\lambda \B= T_{T_\lambda \B} \zeta+R(\zeta,T_\lambda \B).
$$
Thus \eqref{Bony3} (applied with~$\alpha=0$ and~$\beta=s-1/2$) implies that
$$
\lA (\zeta-T_\zeta)T_\lambda \B\rA_{H^{s-\mez}}\les \lA T_\lambda \B\rA_{C^0_*}\lA \zeta\rA_{H^{s-\mez}}.
$$
Using \eqref{esti:quant1} this yields
$$
\lA (\zeta-T_\zeta)T_\lambda B\rA_{H^{s-\mez}}\les M_0^1(\lambda) \lA B\rA_{C^1_*}\lA \zeta\rA_{H^{s-\mez}}.
$$
We thus end up with 
$\lA F_2\rA_{H^{s-\mez}}\le \mathcal{F}\bigl( \| \eta \|_{H^{s+\mez}},\lA (V,B)\rA_{H^s}\bigr)$.

By combining the previous results, we obtain
\begin{equation}\label{parastep3}
(\partial_{t}+T_V\cdot\partialx)\zeta=T_\lambda U +h_3,
\end{equation}
where 
$\lA h_3\rA_{H^{s-\mez}}\le \mathcal{F}\bigl( \| \eta \|_{H^{s+\mez}},\lA (V,B)\rA_{H^s}\bigr)$. 
As in the second step, by commuting the equation~\eqref{parastep3} 
with~$\lDx{s}$ we obtain the desired result~\eqref{seconde}, which 
concludes the proof. 
\end{proof}

\subsection{Symmetrization of the equations}
We shall use Proposition~\ref{prop:csystem2}. 
To prove an~$L^2$ estimate for System~\eqref{premiere}--\eqref{seconde}, we begin by performing a 
symmetrization of the non-diagonal part. Here we use in an essential way the fact that 
the Taylor coefficient~$\ma$ is a positive function. 
Again, let us mention that this assumption 
is automatically satisfied 
for infinitely deep fluid domain: this result was first proved by Wu~(see~\cite{WuInvent,WuJAMS}) and one can check 
that the proof remains valid for any~$C^{1,\alpha}$-domain, with $0<\alpha<1$.

\begin{prop}\label{psym}
Introduce the symbols
$$
\gamma=\sqrt{\ma \lambda},\quad q= \sqrt{\frac{\ma}{\lambda}},
$$
and set~$\theta_s=T_{q} \zeta_s$. Then 

\begin{align}
&\partial_{t} U_s+T_{V}\cdot\partialx  U_s  + T_\gamma \theta_s = F_1,\label{systreduit1}
\\
&\partial_{t}\theta_s+T_{V}\cdot\partialx \theta_s - T_\gamma U_s =F_2,
\label{systreduit2}
\end{align}
for some source terms~$F_1,F_2$ satisfying 
$$
\lA (F_{1}(t),F_{2}(t))\rA_{L^{2}\times L^{2}}\\
\le \mathcal{F}\bigl( \lA \eta(t)\rA_{H^{s+\mez}},\lA (V,B)(t)\rA_{H^s}\bigr).
$$
\end{prop}
\begin{proof}
Directly from~\eqref{premiere}--\eqref{seconde}, we obtain~\eqref{systreduit1}--\eqref{systreduit2} with
\begin{align*}
F_1&\defn f_1+\bigl( T_{\gamma}T_q -T_{\ma}\bigr)\zeta_s,\\
F_2&\defn T_q f_2
+\bigl(T_q T_\lambda-T_{\gamma}\bigr)U_s -\left[ T_q , \partial_t +T_V\cdot \nabla\right]
\zeta_s.
\end{align*}
The commutator between~$T_q$ and~$\partial_t +T_V \cdot \nabla$ is estimated by means of 
Lemma~\ref{lemm:Dt}:
\begin{multline}\label{comm:Dtbis}
\lA \bigl[ T_q , \partial_t +T_V \cdot \nabla\bigr] \zeta_s\rA_{L^2({\mathbf{R}}^d)}\\
\le K  \left\{\mathcal{M}^{-\mez}_0 (q)\lA V\rA_{C^{1+\eps}_*}
+\mathcal{M}^{-\mez}_0 (\partial_t q+V\cdot \nabla q) 
\right\}
\times 
 \lA \zeta_s\rA_{H^{-\mez}({\mathbf{R}}^d)}.
\end{multline}$T_q f_2$ is estimated by means of \eqref{esti:quant1}. 
The key point is to estimate~$( T_{\gamma}T_q -T_{\ma})\zeta_s$ and 
$(T_q T_\lambda-T_{\gamma})U_s$. 
Since~$ \gamma q =\ma$, the operator~$T_\gamma T_q -T_{\ma}$ if of order~$-1/2$ 
since~$\gamma$ is a symbol of order~$1/2$,~$q$ is of order~$-1/2$, 
and since these symbols are~$C^{1/2}$ in~$x$. Similarly, since~$q \lambda =\gamma$, the operator 
$T_q T_\lambda-T_{\gamma}$ is of order~$0$. 
More precisely, 
by using the estimate \eqref{esti:quant2} for symbolic calculus, we obtain
\begin{align*}
&\lA T_\gamma T_q -T_\ma \rA_{H^{-\mez}\rightarrow L^{2}}\les 
M^{1/2}_{1/2}(\gamma)M^{-1/2}_{0}(q)+M^{1/2}_{0}(\gamma)M^{-1/2}_{1/2}(q),\\
&\lA T_q T_\lambda -T_\gamma \rA_{L^{2}\rightarrow L^{2}}\les 
M^{-1/2}_{1/2}(q)M^{1}_{0}(\lambda)+M^{-1/2}_{0}(q)M^{1}_{1/2}(\lambda).
\end{align*}
The above semi-norms are easily estimated by means of the~$C^{1/2}$ norms of~$\zeta=\nabla \eta$ 
and~$\ma$ (given by the Sobolev injection and Proposition~\ref{prop:ma}).
\end{proof}
 
We are now in position to prove an~$L^2$ estimate for~$(U_s,\theta_s)$. 
\begin{lemm}\label{estimationdeUs}
There exists a non-decreasing function~$\mathcal{F}$ such that
\begin{equation}\label{esti:4.46}
\lA U_s\rA_{L^\infty([0,T];L^{2})}+
\lA \theta_s\rA_{L^\infty([0,T];L^{2})}
\le \mathcal{F}(M_{s,0})+T\mathcal{F}(M_s(T)).
\end{equation}
\end{lemm}
\begin{rema}
The fact that this 
implies corresponding estimates for the Sobolev norms 
of~$\eta,\psi,V,B$ is explained below in~\S\ref{S:eVB}. 
\end{rema}
\begin{proof}
Multiply \eqref{systreduit1}Ê by~$U_s$ 
and \eqref{systreduit2}Ê by~$\theta_s$ and integrate in space to obtain
$$ 
\frac{d}{dt}\left\{\lA U_s(t)\rA_{L^2}^2 +\lA \theta_s(t)\rA_{L^2}^2\right\}
+(I)+(II)=(III),
$$
where
\begin{align*}
(I)&\defn \langle T_{V(t)} \cdot \nabla U_s(t), U_s(t)\rangle +\langle T_{V(t)} \cdot \nabla \theta_s(t), \theta_s(t)\rangle ,\\
(II)&\defn \langle T_{\gamma(t)} \theta_s(t) , U_s(t)\rangle 
-\langle T_{\gamma(t)} U_s (t), \theta_s(t)\rangle ,\\
(III)&\defn \langle F_1(t) ,U_s(t)\rangle +\langle F_2(t), \theta_s(t)\rangle.
\end{align*}
Then the key points are that (see point~$(iii)$ in Theorem~\ref{theo:sc0})
$$
\lA (T_{V(t)}\cdot \partialx)^* + T_{V(t)}\cdot \partialx \rA_{L^2\rightarrow L^2}
\les \lA V(t)\rA_{W^{1,\infty}},
$$
and 
$$
\lA T_{\gamma(t)} -(T_{\gamma(t)})^*\rA_{L^2\rightarrow L^2}
\les M^{1/2}_{1/2}(\gamma(t)).
$$
We then easily obtain~\eqref{esti:4.46}. 
\end{proof}

\subsection{Back to estimates for the original unknowns}\label{S:eVB}

Up to now, we only estimated~$(U_s,\theta_s)$ in~$L^\infty([0,T];L^2\times L^2)$. In 
this section, we shall show how we can recover estimates for the original unknowns 
$(\eta,\psi,V,\B)$ in~$L^\infty([0,T];H^{s+\mez}\times H^{s+\mez}\times H^s\times H^{s})$. 
Recall that the functions~$U_s$ and~$\theta_s$ are obtained 
from~$(\eta,V,\B)$ through:
\begin{align*}
&U_s \defn \lDx{s} V+T_\zeta \lDx{s}\B,\\[0.5ex]
&\theta_s\defn T_{\sqrt{\ma/\lambda}}\lDx{s}\nabla \eta.
\end{align*}
The analysis is in four steps:
\begin{enumerate}[(i)]
\item We first prove some estimates for~$(B,V,\eta)$ and the Taylor coefficient~$\ma$ 
in some lower order norms.
\item Then, by using the previous estimate of~$\theta_s$, we show how to recover an 
estimate of the~$L^\infty([0,T],H^{s+\mez})$-norm of~$\eta$.
\item Once~$\eta$ is estimated in~$L^\infty([0,T],H^{s+\mez})$, by using the 
estimate for~$U_s$, we estimate~$(B,V)$ in~$L^\infty([0,T];H^s)$. 
Here we make an essential use of our first result on the paralinearization of the Dirichlet-Neumann operator 
(see Proposition~\ref{coro:paraDN1}). Namely, we use the fact 
that one can paralinearize the Dirichlet-Neumann operator for any domain 
whose boundary is in~$H^{\mu}$ for some~$\mu>1+d/2$.
\item The desired estimate for~$\psi$ follows directly from the 
previous estimates for $\eta,V,\B$, the 
identity~$\nabla\psi=V+\B\nabla \eta$ and the fact that one 
easily obtain an~$L^{\infty}([0,T];L^2)$-estimate for~$\psi$.
\end{enumerate}
 We begin with the following 
lemma.
\begin{lemm}\label{ecBV}
There exists a non-decreasing function~$\mathcal{F}$ such that,
\begin{equation}\label{ecBVBV}
\| \eta \|_{L^\infty([0,T];H^{s})}+
\lA (B,V)\rA_{L^\infty([0,T]; H^{s-\mez})}
\le \mathcal{F}(M_{s,0})+T\mathcal{F}(M_s(T)).
\end{equation}
and, for any~$0<\eps<s-1-d/2$, 
\begin{equation}\label{ecBVma}
\lA \ma\rA_{L^\infty([0,T];C_*^{\eps})}\le \mathcal{F}(M_{s,0})
+\sqrt{T}\mathcal{F}(M_s(T)).
\end{equation}
\end{lemm}
\begin{proof}
The proof is based on classical Sobolev or H\"older estimates 
for the solution~$w$ of a transport equation of the form 
$\partial_t w +V\cdot \nabla w =F$ (these estimates are recalled in 
Proposition~\ref{estclass}). 

We begin by proving \eqref{ecBVBV}. It follows from \eqref{eq:B}, \eqref{eq:V} and \eqref{eq:etaB} 
that 
$$
\partial_t \eta+V\cdot\partialx \eta=F(\eta),\quad
\partial_t V+V\cdot\partialx V=F(V),
\quad 
\partial_t B+V\cdot\partialx B=F(B)
$$
with $F(\eta)=B$, $F(V)=-\ma \zeta$ and $F(B)=\ma -g$. Now 
by using the estimate~\eqref{esti:a1} for the $H^{s-\mez}$-norm 
of~$\ma-g$ together with the product rule in Sobolev spaces, we have
$$
\| F(\eta)(t)\| _{H^{s}} +\lA F(V)(t)\rA_{H^{s-\mez}}+
\lA F(B)(t)\rA_{H^{s-\mez}}\le 
\mathcal{F}\bigl( \lA \eta(t)\rA_{H^{s+\mez}},\lA (V,B)(t)\rA_{H^s}\bigr),
$$
for some non decreasing function $\mathcal{F}$. 
Using the obvious estimate 
$\lA h\rA_{L^1([0,T])}\le T\lA h\rA_{L^\infty([0,T])}$, we deduce that
$$
\| F(\eta)\| _{L^1([0,T];H^{s})} +\lA F(V)\rA_{L^1([0,T];H^{s-\mez})}+
\lA F(B)\rA_{L^1([0,T];H^{s-\mez})}
\le T\mathcal{F}(M_s(T)).
$$
The desired estimate for $\eta$ (resp.\ for $(V,B)$) 
then follows from the estimate~\eqref{eq.iii} in 
Proposition \ref{estclass} applied with~$\sigma = s $ (resp.\ $\sigma=s-\mez$). 

It remains to prove \eqref{ecBVma}. We shall estimate the 
$L^\infty_t(L^\infty_x)$-norm of $\dot{a}\defn \langle D_x\rangle^{\eps} a$. 
To do so, using the estimate \eqref{eq.i}, as above one reduces the proof 
to proving that 
the $L^1_t(L^\infty_x)$-norm of $(\partial_t+V\cdot\partialx)\dot{a}$ is bounded 
by $T\mathcal{F}(M_s(T))$. Again, it is sufficient to prove that 
the $L^\infty_t(L^\infty_x)$-norm of $(\partial_t+V\cdot\partialx)\dot{a}$ is bounded 
by $\mathcal{F}(M_s(T))$. 
To prove this estimate, write
$$
(\partial_t+V\cdot\partialx)\dot{a}=
\langle D_x\rangle^{\eps}(\partial_t+V\cdot\partialx)a
+\bigl[ V, \langle D_x\rangle^{\eps}\bigr]\cdot\partialx a.
$$ 
The first term is estimated by means of \eqref{esti:a2}. 
The desired estimate for the second term follows from the bound
$$
\Vert [\lDx{\eps}, V] a \Vert_{L^\infty(\xR^d)}
\leq K \| V\|_{H^s(\xR^d)} \|a \|_{L^\infty(\xR^d)},
$$
which follows from paraproduct estimates in Zygmund spaces 
(see \cite{BCD}). 
\end{proof}

\begin{lemm}\label{ecBVeta}
There exists a non-decreasing function~$\mathcal{F}$ such that
\begin{equation}
\| \eta \|_{L^\infty([0,T];H^{s+\mez})}
\le \mathcal{F}\bigl(\mathcal{F}(M_{s,0})+T\mathcal{F}\bigl(M_s(T)\bigr)\bigr).
\end{equation}
\end{lemm}
\begin{proof}
Chose~$\eps$ and an integer~$N$ such that 
$$
0<\eps<s-1-\frac{d}{2},\qquad (N+1)\eps >\mez.
$$
Set~$R=I-T_{1/q}T_q$ to obtain
$$
\zeta_s= T_{1/q}T_q \zeta_s +R \zeta_s,
$$
where recall that~$\zeta_s=\lDx{s} \zeta$. 
Consequently,
$$
\zeta_s =(I+R+\cdots +R^N) T_{1/q}T_q \zeta_s +R^{N+1}\zeta_s.
$$
By definition of~$q=\sqrt{\ma/\lambda}$, 
Theorem~\ref{theo:sc0} implies that, for all~$\mu\in\xR$, there exists a non-decreasing function~$\mathcal{F}$ 
depending only on~$\eps$ and~$\inf_{(t,x)\in [0,T]\times {\mathbf{R}}^d} \ma(t,x)>0$ such that, 
$$
\lA R(t)\rA_{H^\mu \rightarrow H^{\mu+\eps}}\le \mathcal{F}\bigl(\lA \ma(t) \rA_{C_*^{\eps}},
\lA \eta(t)\rA_{C_*^{1+\eps}}\bigr),
$$
and 
$$
\lA T_{1/q(t)}\rA_{H^{\mu+1/2} \rightarrow H^{\mu}}\le \mathcal{F}\bigl(\lA \eta(t)\rA_{W^{1,\infty}}\bigr).
$$
Therefore
$$
\lA \nabla \eta\rA_{H^{s-\mez}}=\lA \zeta_s\rA_{H^{-\mez}}
\le \mathcal{F}\bigl(\lA \ma\rA_{C_*^{\eps}},
\| \eta \|_{C_*^{1+\eps}}\bigr)
\left\{ \lA T_q\zeta_s\rA_{L^2}+\lA \zeta_s\rA_{H^{-1}}\right\}.
$$
Now it follows from Lemma~\ref{ecBV} and the Sobolev embedding that
$$
\lA \ma\rA_{L^\infty([0,T];C_*^{\eps})}+
\| \eta \|_{L^\infty([0,T];C_*^{1+\eps})}+\lA \zeta_s\rA_{L^\infty([0,T];H^{-1})}
\le \mathcal{F}(M_{s,0})+T\mathcal{F}(M_s(T)).
$$
On the other hand, 
it follows from~Lemma~\ref{estimationdeUs} that
$$
\lA T_q \zeta_s\rA_{L^\infty([0,T];L^{2})}
\le \mathcal{F}(M_{s,0})+T\mathcal{F}(M_s(T)).
$$
This implies the desired result.
\end{proof}

It remains only to estimate~$(V,B)$. 
\begin{lemm}\label{ecBVs}
There exists a non-decreasing function~$\mathcal{F}$ such that
\begin{equation}
\lA (V,\B)\rA_{L^\infty([0,T];H^{s})}
\le \mathcal{F}\bigl(\mathcal{F}(M_{s,0})+T\mathcal{F}\bigl(M_s(T)\bigr)\bigr).
\end{equation}
\end{lemm}
\begin{proof}
The proof is based on the relation between~$V$ and~$B$ given by Proposition~\ref{cancellation}.

\noindent{\sc Step} 1. Recall that $U=V+T_\zeta B$. 
We begin by proving that there exists a non-decreasing function~$\mathcal{F}$ such that
\begin{equation}\label{p50:1}
\lA U \rA_{L^\infty([0,T];H^{s})}
\le \mathcal{F}\bigl(\mathcal{F}(M_{s,0})+T\mathcal{F}\bigl(M_s(T)\bigr)\bigr).
\end{equation}
To see this, write
$$
\lDx{s} U = U_s +\bigl[ \lDx{s},T_\zeta\bigr] \B
$$
and use Theorem~\ref{theo:sc0} to obtain
$$
\lA \bigl[ \lDx{s},T_\zeta\bigr] \B\rA_{L^{2}}
\les \lA \zeta\rA_{C_*^{\mez}}\lA \B\rA_{H^{s-\mez}}.
$$
Since, by assumption,~$s>1+d/2$ we have 
$\lA \zeta\rA_{C_*^{\mez}}
\les \lA \zeta\rA_{H^{s-\mez}}\le \| \eta \|_{H^{s+\mez}}$ 
and hence
$$
\lA U\rA_{H^{s}}\les \lA U_s\rA_{L^{2}}
+ \| \eta \|_{H^{s+\mez}} \lA \B\rA_{H^{s-\mez}}.
$$
The three terms in the right-hand side of the above inequalities have been already estimated 
(see Lemma~\ref{estimationdeUs} for~$U_s$, Lemma~\ref{ecBV} for~$B$ and 
Lemma~\ref{ecBVeta} for~$\eta$). This proves~\eqref{p50:1}.

\noindent{\sc Step} 2. 
Taking the divergence in $U=V+T_{\zeta}B$, 
we get according to Proposition~\ref{cancellation}, Lemma~\ref{ecBV} and Lemma~\ref{ecBVeta}:
\begin{align*}
\cn U &= \cn V + \cn T_\zeta B=\cn V +T_{\cn \zeta} B+T_\zeta\cdot \nabla B\\
&= -G(\eta) B+T_{i \zeta \cdot \xi+ \cn \zeta}B + \gamma\\
&=-T_\lambda B +R(\eta)B+T_{i \zeta \cdot \xi+ \cn \zeta}B + \gamma\\
&=T_q B +R(\eta)B+T_{\cn \zeta}B + \gamma
\end{align*}
where, by notation,
$$
q\defn -\lambda+i\zeta\cdot\xi,
$$
and
$$
\|\gamma \|_{H^{s- \mez}} \leq \mathcal{F}( \|( \eta,  V, B)\|_{H^{s+ \mez} \times H^{\mez}\times H^\mez}).
$$
According to Proposition~\ref{coro:paraDN1} (with~$\mu= s - \mez$)  and Lemma~\ref{ecBV}, we deduce 
\begin{equation}\label{eq.blabla}
 T_{q} B = \cn U - T_{\cn \zeta} B - R(\eta) B-\gamma.
 \end{equation}
Now write
$$
B=T_{\frac{1}{q}}T_q B +\Bigl(I-T_{\frac{1}{q}}T_q\Bigr) B
$$
to obtain from~\eqref{eq.blabla}
$$
B=T_{\frac{1}{q}}  \cn U -T_{\frac{1}{q}} \gamma+R_{-\epsilon} B
$$
where
\begin{equation}\label{defi:Repsilon}
R_{-\epsilon}\defn T_{\frac{1}{q}} \Bigl( -T_{\cn \zeta}-R(\eta)\Bigr)+\Bigl(I-T_{\frac{1}{q}}T_q\Bigr).
\end{equation}
Notice now that, according to Lemma~\ref{ecBVeta}, 
we control~$\cn \zeta= \Delta \eta$ in~$H^{s- \frac 3 2 }\subset C^{-\mez}_*$ 
(since~$s> 1 + \frac d 2$) so $T_{\cn \zeta}$ is an operator of order~$\mez$. Finally,
$q=- \lambda + i \zeta \cdot \xi \in \Gamma^1_{1/2}$ with 
$M^1_{1/2}(q)\leq C( \|\eta\|_{H^{s+ \mez}})$. 
Moreover, $q^{-1}$ is of order $-1$ and we have
$$
M^1_{1/2}\bigl(q^{-1}\bigr)\leq C( \|\eta\|_{H^{s+ \mez}}).
$$
Consequently, according to~\eqref{esti:quant1} and \eqref{esti:quant2}, the operator 
$R_{-\epsilon}$ given by~\eqref{defi:Repsilon} is of order $-\mez$. 
applying $ T_{(- \lambda + i \zeta \cdot \xi )^-1}$ to~\eqref{eq.blabla}, we get 
$$ B = W+ R_{-\epsilon} B~$$
where $W\defn T_{\frac{1}{q}}  \cn U -T_{\frac{1}{q}} \gamma$ satisfies
$$
\| W \|_{H^{s} }\leq \mathcal{F}\bigl(\mathcal{F}(M_{s,0})
+T\mathcal{F}\bigl(M_s(T)\bigr)\bigr).
$$
Since $R_{-\epsilon}$ is an operator of order~$-\mez$ and 
since we have estimated the $H^{s-\mez}$-norm of 
$B$ (see Lemma~\ref{ecBV}), we conclude that 
$$
\| B\|_{H^{s} }\leq \mathcal{F}\bigl(\mathcal{F}(M_{s,0})
+T\mathcal{F}\bigl(M_s(T)\bigr)\bigr),
$$
and coming back to the relation~$U= V+ T_\zeta B$ we get that~$V$ satisfies the same estimate.
\end{proof}

\begin{lemm}\label{L:estpsi}
There exists a non-decreasing function~$\mathcal{F}$ such that
\begin{equation*}
\lA \psi\rA_{L^\infty([0,T];H^{s+\mez})}
\le \mathcal{F}\bigl(\mathcal{F}(M_{s,0})+T\mathcal{F}\bigl(M_s(T)\bigr)\bigr).
\end{equation*}
\end{lemm}
\begin{proof}
Since~$\nabla\psi=V+\B \nabla \eta$ and since the~$L^\infty([0,T];H^{s-\mez})$-norm of 
$(\nabla\eta,V,\B)$ has been previously estimated, it remains only to estimate 
$\lA \psi\rA_{L^\infty([0,T];L^{2})}$. 

Since 
$$
B\defn \frac{\partialx \eta \cdot\partialx \psi+ G(\eta)\psi}{1+|\partialx  \eta|^2},
$$
the equation for $\psi$ can be written under the form
$$
\partial_{t}\psi+g \eta
+ \frac{1}{2}\la\partialx \psi\ra^2  -\frac{1}{2}(1+\la \nabla\eta\ra^2)B^2=0.
$$
Therefore, since $V=\nabla\psi-B\nabla\eta$,
\begin{equation}\label{psiVB}
\begin{aligned}
\partial_t \psi+V\cdot \nabla \psi&=\partial_t\psi+\la \nabla\psi\ra^2-B\nabla\eta\cdot\nabla\psi\\
&=-g\eta+\mez \la\nabla\psi\ra^2+\frac{1}{2}(1+\la \nabla\eta\ra^2)B^2-B\nabla\eta\cdot\nabla\psi\\
&=-g\eta+\mez \la \nabla\psi-B\nabla\eta\ra^2-\mez B^2 \la \nabla\eta\ra^2+\frac{1}{2}(1+\la \nabla\eta\ra^2)B^2\\
&=-g\eta+\mez V^2+\mez B^2.
\end{aligned}
\end{equation}
The desired~$L^2$ estimate then follows from classical results (see Proposition~\ref{estclass}).
\end{proof}

\section{Contraction}\label{S:contraction}

In this section we prove a contraction estimate for the difference of two solutions which implies the uniqueness of solutions and a Lipschitz property in a lower norm ($\mathcal{H}^{s-1}$, compared to the $\mathcal{H}^s$ norm where the {\em a priori} estimates are established). This phenomenon is standard for quasi-linear PDE's. This choice of norm to establish the contraction property is the result of a compromise as on the one hand, the highest the norm is chosen the easiest the non linear analysis will be (as the norm controls more quantities), while some loss of derivatives are necessary (in particular as far as the Dirichlet--Neumann operator is concerned), see Remark~\ref{re.5.4}.

\begin{theo}\label{th.lipschitz}
Let~$(\eta_j,\psi_j)$,~$j=1,2$, be two solutions of~\eqref{system} such that 
$$
(\eta_j,\psi_j,V_j,B_j)\in 
C^0([0,T_0];H^{s+\mez}\times H^{s+\mez}\times H^s\times H^s),
$$
for some fixed~$T_0>0$,~$d\ge 1$ and~$s>1+d/2$. We also assume 
that the condition \eqref{hypt} holds for 
$0\le t\le T_0$ and that there exists a positive constant~$c$ 
such that for all~$0\le t\le T_0$ and for all~$x\in {\mathbf{R}}^d$, we have 
$\ma_j(t,x)\ge c$ for~$j=1,2$,~$t\in [0,T]$.
Set
$$
M_j\defn \sup_{t\in [0,T]} \lA (\eta_j,\psi_j,V_j,B_j)(t)\rA_{H^{s+\mez}\times H^{s+\mez} \times H^s\times H^s}, 
$$
$$
\deta\defn \eta_1-\eta_2, \quad \psi=\psi_1-\psi_2,\quad \dV\defn V_1-V_2, \quad \dB=\B_1-\B_2.
$$
Then we have 
\begin{multline}\label{eq.lipschitz}
\|(\eta, \psi, V, B)\| _{L^\infty((0,T); H^{s-\mez}\times H^{s-\mez}\times H^{s-1}\times H^{s-1})} \\
\leq \mathcal{K} (M_1, M_2) \|(\eta, \psi, V, B) 
\mid_{t=0}\| _{H^{s-\mez}\times H^{s-\mez}\times H^{s-1}\times H^{s-1}}.
\end{multline}
\end{theo} 
Let us recall that 
\begin{equation}\label{BVzeta0}
\left\{
\begin{aligned}
&(\partial_{t}+V_j\cdot\partialx)\B_j=\ma_j-g,\\
&(\partial_t+V_j\cdot\partialx)V_j+\ma_j\zeta_j=0,\\
&(\partial_{t}+V_j\cdot\partialx)\zeta_j=G(\eta_j)V_j+ \zeta_j G(\eta_j)\B_j+\gamma_j,\quad \zeta_j=\nabla \eta_j,
\end{aligned}
\right.
\end{equation}
where~$\gamma_j$ is the remainder term given by \eqref{eq:zeta}. Let
$$
N(T)\defn \sup_{t\in [0,T]} \lA (\eta,\psi,V,B)(t)\rA_{H^{s-\mez}\times H^{s-\mez}\times H^{s-1}\times H^{s-1}}.
$$
Our goal is to prove an 
estimate of the form 
\begin{equation}\label{NM1M2}
N(T)\le \K(M_1, M_2) N(0) +T \, \mathcal{K}(M_1,M_2) N (T),
\end{equation}
for some non-decreasing function~$\mathcal{K}$ depending only on~$s$ and~$d$. 
Then, by choosing~$T$ small enough, this implies~$N(T)\leq 2\K(M_1, M_2) N(0)~$ for~$T_1$ 
smaller than the minimum of~$T_0$ and~$1/2 \K(M_1,M_2)$, and iterating 
the estimate between~$[T_1, 2T_1]$,\dots,~$[T- T_1, T_1]$ implies Theorem~\ref{th.lipschitz}.

\subsection{Contraction for the Dirichlet-Neumann}
The first step in the proof of Theorem~\ref{th.lipschitz} is to prove a Lipschitz 
property for the Dirichlet-Neumann operator. 
This was already achieved in a very weak norm 
in Theorem~\ref{th.25}, and here we use elliptic theory to improve the result.

\begin{theo}\label{L:p60}Assume that~$s>1+\frac{d}{2}$. 
There exists a non-decreasing function~$\mathcal{F}$ such that, for all~$\eta_1,\eta_2\in H^{s+\mez}$ 
and all~$f\in H^s$, we have 
\begin{equation}
\lA \left[ G(\eta_1)-G(\eta_2)\right] f\rA_{H^{s-\tdm}}\le 
\mathcal{F} \bigl(\lA (\eta_1,\eta_2)\rA_{H^{s+\mez}}\bigr)\lA \eta_1-\eta_2\rA_{H^{s-\mez}}  \lA f\rA_{H^s}.
\end{equation}
\end{theo}
\begin{rema}\label{re.5.4}
We were unable to prove a similar estimate in a higher norm. On the other hand,  this estimate is in some sense stronger than Theorem~\ref{th.25}. Indeed, in view of Sobolev injections, the r.h.s.\ here does not control the Lipschitz norm of $(\eta_1 - \eta_2)$ which appears in Theorem~\ref{th.25}.
\end{rema}
\begin{proof}
The proof follows closely that of Theorem \ref{th.25} and we keep the notations $\rho_j$ (see \eqref{diffeo}), 
$\widetilde{\phi}_j(x,z) = \phi_j(x,\rho_j(x,z))$, 
$v= \widetilde{\phi}_1-\widetilde{\phi}_2$ introduced there. 
Recall also that we have set (see \eqref{Lambda-2})
$$
\Lambda^i=(\Lambda^i_1,\Lambda^i_2),\quad \Lambda_1^i = \frac{1}{\partial_z \rho_i}\partial_z, \quad \Lambda_2^i = \nabla_x - \frac{\nabla_x \rho_i}{\partial_z \rho_i}\partial_z.
$$

Notice that, using the smoothing property of the Poisson kernel, we have
\begin{equation}\label{diff-lambda1}
\left\{
 \begin{aligned}
 &(i) \quad \Lambda^1_k- \Lambda^2_k= \beta_k \partial_z, \quad \text{with }\supp\beta_k \subset \xR^d\times J,  \text{where }  J=[-1,0],\\
 &(ii) \quad \Vert \beta_k\Vert_{L^2(J, H^{s-1}(\xR^d))} \leq\mathcal{F}(\Vert(\eta_1,\eta_2)\vert_{H^{s+\mez}\times H^{s+\mez}})\Vert \eta_1-\eta_2\Vert_{H^{s-\mez}(\xR^d)}.
 \end{aligned}
 \right.
 \end{equation}
Recall that
\begin{equation}\label{Geta1j}
 G(\eta_j)f = U_j\arrowvert_{z=0}, \quad U_j= \Lambda_1^j \widetilde{\phi}_j - \nabla_x \rho_j\cdot \Lambda_2^j \widetilde{\phi}_j.
\end{equation}
Let us set $U = U_1 -U_2.$ According to Lemma~\ref{lem.3.11},Ê Theorem \ref{L:p60} will follow from the following estimate
\begin{equation}\label{est:U}
\Vert U \Vert_{L^2(J, H^{s-1} )} + \Vert \partial_z U \Vert_{L^2(J, H^{s-2} )} \leq \mathcal{F} \bigl(\lA (\eta_1,\eta_2)\rA_{H^{s+\mez} \times H^{s+\mez} }\bigr)\lA f\rA_{H^s} \lA \eta_1-\eta_2\rA_{H^{s-\mez}}.
 \end{equation}
According to \eqref{formeU} and \eqref{formedzU}, \eqref{est:U} will be a consequence of the following one
\begin{equation}\label{est:Bj}
\begin{aligned}
&\sum_{k=1}^5 \Vert B_k \Vert_{L^2(J, H^{s-1}) }  \leq    \mathcal{F} \bigl(\lA (\eta_1,\eta_2)\rA_{H^{s+\mez}\times H^{s+\mez}}\bigr) \lA f\rA_{H^s}\lA \eta_1-\eta_2\rA_{H^{s-\mez}} \quad \text{where}\\
&B_1  = \Lambda^1_1v,\quad B_2 = (\nabla_{x,z}\rho_2) \Lambda_2^2v, \quad B_3 = (\Lambda_1^1- \Lambda_1^2)\widetilde{\phi}_2, \quad B_4 = \nabla_{x,z}(\rho_1 - \rho_2)\Lambda_2^1\widetilde{\phi}_1,\\
& B_5 =  (\nabla_{x,z}\rho_2)  (\Lambda_1^1- \Lambda_1^2)\widetilde{\phi}_1.
 \end{aligned}
  \end{equation}
Since $\widetilde{\phi}_j$ is a variational solution,   Proposition \ref{p1} with $\sigma = s-1$ show that
$$ \Vert \nabla_{x,z}\widetilde{\phi}_j \Vert_{L^\infty(J, H^{s-1} )} + \Vert \Lambda ^j_k\widetilde{\phi}_j \Vert_{L^\infty(I, H^{s-1} )}\leq \mathcal{F}(\Vert \eta_j \Vert_{H^{s+\mez} }) \Vert f \Vert _{H^s }. $$
 Since $s>1 + \frac{d}{2},$ it follows the from \eqref{diff-lambda} that 
 \begin{equation}\label{B35}
  \sum_{l=3}^5 \Vert B_l\Vert_{L^2(J,H^{s-1} )} \leq\mathcal{F}(\Vert(\eta_1,\eta_2)\Vert_{H^{s+\mez}\times H^{s+\mez}})\Vert \eta_1-\eta_2\Vert_{H^{s-\mez}}\Vert f \Vert _{H^s}. 
\end{equation}
 Since 
 $$\Vert B_1\Vert_{L^2(J,H^{s-1})} + \Vert B_2\Vert_{L^2(J,H^{s-1})} \leq \mathcal{F}(\Vert(\eta_1,\eta_2)\vert_{H^{s+\mez}\times H^{s+\mez}})\Vert \nabla_{x,z} v \Vert_{L^2(I,H^{s-1})}$$
using the estimate \eqref{B35},  we see that \eqref{est:Bj} will be a consequence of the following lemma. Therefore  Theorem \ref{L:p60} will be proved if we prove the following result.

\begin{lemm}\label{est:B12}
We have
$$\Vert \nabla_{x,z} v \Vert_{L^2(J,H^{s-1})} \leq \mathcal{F}(\Vert(\eta_1,\eta_2)\Vert_{H^{s+\mez}\times H^{s+\mez}})\Vert \eta_1-\eta_2\Vert_{H^{s-\mez}}\Vert   f \Vert _{H^s}.$$
\end{lemm}
\begin{proof}
Notice that $v = \widetilde{\phi}_1-\widetilde{\phi}_2$ is a solution of the problem
 \begin{equation}\label{eq-v}
\partial_z^2 v +\alpha_1\Delta v + \beta_1 \cdot\partialx\partial_z v 
- \gamma_1 \partial_z v=F, \quad v\arrowvert_{z=0} =0
\end{equation}
where
$$
F=(\alpha_2-\alpha_1)\Delta \widetilde{\phi}_2+(\beta_2-\beta_1)\cdot\nabla\partial_z \widetilde{\phi}_2 -(\gamma_2-\gamma_1)\partial_z \widetilde{\phi}_2
$$
and~$\alpha_j$ are given by~\eqref{alpha}. We would like to apply Proposition \ref{p1} with $\sigma = s - \frac{3}{2}.$ To this end, according to \eqref{XY},  we shall  estimate the $L^2(J, H^{s-2}(\xR^d))$  norm of $F$ and the $X^{-\mez}(J)$ norm of $\nabla_{x,z} v.$ 

{\it{Estimate on $F$}}:
 Since $s>1+\frac{d}{2}$ (thus $2s-3>0$) we may apply \eqref{pr} with~$s_1 = s-2, s_2= s-1, s_0= s-2$. We get
\begin{align*}
\lA (\alpha_1-\alpha_2)\Delta \widetilde{\phi}_2\rA_{L^2 (J,H^{s-2} )}&\le K 
\lA \alpha_1-\alpha_2\rA_{L^2 (J,H^{s-1} )}\lA \Delta \widetilde{\phi}_2\rA_{L^\infty (J,H^{s-2} )},\\
\lA (\beta_1-\beta_2)\cdot \nabla\partial_z \widetilde{\phi}_2\rA_{L^2 (J,H^{s-2} )}
&\le K \lA \beta_1-\beta_2\rA_{L^2 (J,H^{s-1} )} 
\lA \nabla\partial_z \widetilde{\phi}_2\rA_{L^\infty (J,H^{s-2} )},\\
\lA (\gamma_1-\gamma_2)\partial_z \widetilde{\phi}_2\rA_{L^2 (J,H^{s-2} )}&\le K 
\lA \gamma_1-\gamma_2\rA_{L^2 (J,H^{s-2} )}\lA \partial_z \widetilde{\phi}_2\rA_{L^\infty (J,H^{s-1} )}.
\end{align*}
Then, using the product rule in Sobolev space~\eqref{pr}, 
and~\eqref{eq.rho1}, we obtain
\begin{multline}\label{eq.6.4}
\lA \alpha_1-\alpha_2\rA_{L^2 (J,H ^{s-1})}+\lA \beta_1-\beta_2\rA_{L^2 (J,H ^{s-1})}+
\lA \gamma_1-\gamma_2\rA_{L^2 (J,H ^{s-2})}\\
\le 
\mathcal{F}\bigl(\lA (\eta_1,\eta_2)\rA_{H^{s+\mez} \times H^{s+\mez}}\bigr) \lA\eta_1-\eta_2\rA_{H^{s-\mez}}.
\end{multline}
Moreover from  Proposition~\ref{p1} with~$\sigma=s-1$ we have
$$
\lA \nabla_{x,z} \widetilde{\phi}_j\rA_{L^\infty(J,H^{s-1})}
\le C(\lA \eta_j\rA_{H^{s+\mez}})\lA f\rA_{H^{s}}.
$$
It follows that
\begin{equation}\label{est:F}
\lA F\rA_{L^2_z(J,H^{s-2}_x)}\le \mathcal{F} \bigl(\lA (\eta_1,\eta_2)\rA_{H^{s+\mez}}\bigr) \lA \eta_1-\eta_2\rA_{H^{s-\mez}}\lA f\rA_{H^s}.
\end{equation}
 
{\it{Estimate of $ \lA \nabla_{x,z} v \rA_{X^{-\mez}(J)}, \, J= (-1,0).$}}
  
We claim that
\begin{equation}\label{est:nablav}
\lA \nabla_{x,z} v \rA_{X^{-\mez}(J)}\leq \mathcal{F} \bigl(\lA (\eta_1,\eta_2)\rA_{H^{s+\mez} \times H^{s+\mez}}\bigr) \lA \eta_1-\eta_2\rA_{H^{s-\mez}}\lA f\rA_{H^s}.
\end{equation}
 Since $\widetilde\phi_j = \widetilde{u}_j + \underline{\widetilde{f}}$ we have $v=\widetilde{u}_1 -\widetilde{u}_2.$ We begin by proving the following estimate.

There exists a non decreasing function $\mathcal{F}\colon \xR_+\rightarrow \xR_+$ such that
\begin{equation}\label{est-v1}
\Vert \nabla_{x,z} v\Vert_{L^2(J, L^2)} \leq  \mathcal{F} \bigl(\lA (\eta_1,\eta_2)\rA_{H^{s+\mez} \times H^{s+\mez}}\bigr) \lA \eta_1-\eta_2\rA_{H^{s-\mez}}\lA f\rA_{H^s}.
  \end{equation} 
 For this purpose we use the variational characterization of the solutions $u_i$. Setting  $X=(x,z)$ we have
\begin{equation}\label{equi-bis}
 \int_{\widetilde{\Omega}}\Lambda ^i\widetilde{u}_i\cdot\Lambda ^i\theta \, J_i\, dX = - \int_{\widetilde{\Omega}}\Lambda ^i \underline{\widetilde{f}} \cdot\Lambda ^i\theta \, J_i\, dX 
 \end{equation}
 for all $\theta \in H^{1,0}(\widetilde{\Omega}),$ where $J_i = \vert \partial_z \rho_i \vert.$
 
Making the difference between the two equations \eqref{equi-bis}, 
using \eqref{est-rho}  and taking $\theta = v = \widetilde{u}_1 -\widetilde{u}_2$  one can find a positive constant $C$ such that 
\begin{equation*}
\int_{\widetilde{\Omega}}\vert \Lambda ^1v\vert^2\, dX \leq C(A_1+\cdots +A_6)
\end{equation*}
where
\begin{equation*}
\left\{
\begin{aligned}
&A_1=   \int_{\widetilde{\Omega}}\vert (\Lambda^1 - \Lambda^2) 
\widetilde{u}_2 \vert \vert \Lambda^1 v\vert \,J_1\,  dX, \qquad   
&&A_2
=\int_{\widetilde{\Omega}}\vert (\Lambda^1 - \Lambda^2) v \vert \vert \Lambda^2 \widetilde{u}_2 \vert \,J_1 dX,  \\
&A_3 = \int_{\widetilde{\Omega}}\vert \Lambda^2\widetilde{u}_2\vert \vert  
\Lambda^2v \vert \,\vert J_1-J_2\vert\, dX, \qquad 
&&A_4 = \int_{\widetilde{\Omega}}
\vert (\Lambda^1 - \Lambda^2) \underline{\widetilde{f}}  
\vert \vert \Lambda^1 \widetilde{u}\vert \,J_1\, dX, \\
& A_5 =  \int_{\widetilde{\Omega}}\vert (\Lambda^1 - \Lambda^2) v 
\vert \vert \Lambda^2 \underline{\widetilde{f}}  \vert \,J_1\, dX,
\qquad 
&& A_6 =  
 \int_{\widetilde{\Omega}}\vert \Lambda^2  \widetilde{f} \vert \vert  \Lambda^2v \vert \,\vert J_1-J_2\vert\, dX .
\end{aligned}
\right.
 \end{equation*}
It follows from  the elliptic regularity theorem that 
\begin{equation*}
\begin{aligned}
 A_1 &\leq \Vert \Lambda^1 v\Vert_{L^2(\widetilde{\Omega})}\Vert \beta \Vert_{L^2( \widetilde{\Omega})} \Vert \partial_z \widetilde{u}_2 \Vert_{L^\infty(J, L^\infty(\xR^d))} \\
 &\leq \Vert \Lambda^1v \Vert_{L^2(\widetilde{\Omega})} \mathcal{F} (\Vert \eta_1\Vert_{H^{s+\mez}(\xR^d)} \Vert \eta_2\Vert_{H^{s+\mez}(\xR^d)}, \Vert \psi \Vert_{H^{s}(\xR^d)})\Vert \eta_1-\eta_2\Vert_{H^\mez (\xR^d)}.
 \end{aligned}
 \end{equation*}
 Noticing that $ \Lambda^1 - \Lambda^2  = \beta (\partial_z\rho_1)\Lambda^1_1$ where $\beta$ satisfies the estimate in \eqref{diff-lambda} we obtain
 $$A_2 \leq \Vert \partial_z \rho_1 \Vert_{L^\infty(\widetilde{\Omega})} \Vert \beta \Vert_{L^2(\widetilde{\Omega})}\Vert \Lambda^2 \widetilde{u}_2 \Vert_ {L^\infty(\widetilde{\Omega})}\Vert \Lambda^1v\Vert_{L^2(\widetilde{\Omega})}.$$
Using \eqref{est-rho}, \eqref{diff-lambda} and the elliptic regularity we obtain
$$
A_2 \leq  \Vert \Lambda^1v\Vert_{L^2(\widetilde{\Omega})}\mathcal{F} \bigl(\lA (\eta_1,\eta_2)\rA_{H^{s+\mez} \times H^{s+\mez}}\bigr) \lA \eta_1-\eta_2\rA_{H^{s-\mez}}\lA f\rA_{H^s}.
$$
Now we estimate $A_3$ as follows. We have
$$
A_3 \leq \Vert \Lambda^2\widetilde{u}_2 \Vert_ {L^\infty(\widetilde{\Omega})}\Vert \Lambda^2 v \Vert_ {L^ 2(\widetilde{\Omega})}\Vert J_1-J_2\Vert_{L^ 2(\widetilde{\Omega})}.
 $$
Then we observe that 
\begin{align*}
&\Vert J_1-J_2\Vert_{L^ 2(\widetilde{\Omega})}\leq C\Vert \eta_1-\eta_2\Vert_{H^\mez (\xR^d)}\\
& \Vert \Lambda^2 v  \Vert_ {L^ 2(\widetilde{\Omega})}\leq C \Vert \Lambda^1v\Vert_ {L^ 2(\widetilde{\Omega})}
 \end{align*}
and we use the elliptic regularity. 
To estimate $A_4$ and $A_5$ we recall 
that $\underline{\widetilde{f}} = e^{z\langle D_x \rangle}f.$ 
Then we have 
$$
\Vert \beta \partial_z \underline{\widetilde{f}} \Vert_{L^2(J\times \xR^d)}  \leq  \Vert \beta  \Vert_{L^2(I\times \xR^d) }\Vert \partial_z\underline{\widetilde{f}} \Vert_{L^\infty(J \times \xR^d) }.
 $$
Since $\Vert \partial_z\underline{\widetilde{f}} \Vert_{L^\infty(J\times \xR^d)} 
\leq \Vert   \partial_z  \underline{\widetilde{f}} \Vert_{L^\infty(J,H^{s-1}(\xR^d))} 
\leq \Vert f\Vert_{H^s(\xR^d)},$ using \eqref{diff-lambda} 
we obtain
$$
A_4+A_5 \leq \Vert \Lambda^1v\Vert_{L^2(\widetilde{\Omega})} 
\mathcal{F} \bigl(\lA (\eta_1,\eta_2)\rA_{H^{s+\mez}}\bigr) 
\lA \eta_1-\eta_2\rA_{H^{s-\mez}}\lA f\rA_{H^s}.
$$
The term $A_6$ is estimated like $A_3.$ Since $\mez< s-\mez$ 
this proves   \eqref{est-v1} . 
    
To complete the proof of \eqref{est:nablav} 
we have to estimate $\Vert \nabla_{x,z} v \Vert_{L^\infty(I, H^{-\mez})}.$ 
The estimate of 
 $\Vert \nabla_{x} v \Vert_{L^\infty(J, H^{-\mez})} $ follows 
 from \eqref{est-v1}Êand from Lemma \ref{lem.3.11}.
To estimate   $\Vert \partial_z v \Vert_{L^\infty(J, H^{-\mez})} $ 
we have to use \eqref{est-v1} and the equation satisfied by $v$. 
If we prove that 
 \begin{equation}\label{est-v2}
 \Vert \partial_z^2 v\Vert_{L^2(J, H^{-1})} 
 \leq \mathcal{F} \bigl(\lA (\eta_1,\eta_2)\rA_{H^{s+\mez} \times H^{s+\mez}}\bigr) \lA \eta_1-\eta_2\rA_{H^{s-\mez}}\lA f\rA_{H^s}.
  \end{equation} 
 the result will follow again from Lemma \ref{lem.3.11}. 
 Recall that $v$ satifies the equation \eqref{eq-v}.
  
It follows that we have 
\begin{equation}\label{est-v3}
\begin{aligned}
\Vert \partial_z^2 v  \Vert_{L^2(J, H^{-1})}
&\leq \Vert \alpha_1\Delta v \Vert_{L^2(J, H^{-1})} + \Vert \beta_1 \cdot\partialx\partial_z v\Vert_{L^2(J, H^{-1})}  \\
&\quad+ \Vert  \gamma_1 \partial_z v\Vert_{L^2(J, H^{-1})}  +\Vert  F\Vert_{L^2(J, H^{-1})}.
\end{aligned}
\end{equation}

Since $-1<s-2$ \eqref{est:F} yields 
$$
 \lA F\rA_{L^2 (J,H^{-1} )}\le\lA F\rA_{L^2(J,H^{s-2})}\le  \mathcal{F} \bigl(\lA (\eta_1,\eta_2)\rA_{H^{s+\mez} \times H^{s+\mez}}\bigr) \lA \eta_1-\eta_2\rA_{H^{s-\mez}}\lA f\rA_{H^s}.
 $$
On the other hand, since $s-\mez -1>0$ and $ -1< s-\mez -1- \frac{d}{2} $  \eqref{pr} show that we have
 \begin{align*}
 \Vert \alpha_1 \Delta v\Vert_{L^2(J, H^{-1})} &\leq \Vert \alpha_1 \Vert _{L^\infty(J, H^{s-\mez})} \Vert \nabla_x v \Vert_{L^2(J, L^2)}\\
 \Vert \beta_1 \cdot\partialx\partial_z v\Vert_{L^2(J, H^{-1})} &\leq \Vert \beta_1 \Vert _{L^\infty(J, H^{s-\mez})} \Vert \partial_z v \Vert_{L^2(J, L^2)}\\
\Vert  \gamma_1 \partial_z v\Vert_{L^2(J, H^{-1})} &\leq \Vert \gamma_1 \Vert _{L^\infty(J, H^{s-\frac{3}{2}})} \Vert \partial_z v \Vert_{L^2(J, L^2)}.
  \end{align*}
Using Lemma \ref{lem.3.25} and \eqref{est-v1} we obtain eventually \eqref{est-v2}.  

Now Lemma  \ref{est:B12} follows from \eqref{est:F}, \eqref{est:nablav} and Proposition \ref{p1} with $\sigma = s-\frac{3}{2} .$
 \end{proof}
Lemma \ref{est:B12} together with \eqref{est:Bj} prove \eqref{est:U} 
which in turn proves Proposition \ref{L:p60}.
\end{proof}

\subsection{Paralinearization of the equations}
We begin by noticing that, as in the proof of Lemma~\ref{L:estpsi}, it is enough to estimate 
$\deta,\dB,\dV$. Indeed, 
the estimate of the~$L^\infty([0,T];H^{s-1/2}_x)$-norm of~$\dpsi$ is in two elementary steps. 
Firstly, since~$V_j =\nabla \psi_j-B_j\nabla \eta_j$, one can estimate the~$L^\infty([0,T];H^{s-3/2})$-norm of 
$\nabla\dpsi$
from the identity
\begin{align*}
\nabla \dpsi = \dV +\dB \nabla \eta_1 +\B_2 \nabla \deta.
\end{align*}
On the other hand, 
the estimate of the~$L^\infty([0,T];L^2_x)$-norm of~$\dpsi$ follows from the equation \eqref{psiVB}.

An elementary calculation shows that the functions
$$
\zeta= \zeta_1 - \zeta_2, \quad V= V_1-V_2, \quad B=B_1- B_2
$$
satisfy the system of equations
\begin{equation}\label{BVzeta2}
\left\{
\begin{aligned}
&\partial_{t}\dB+V_1\cdot\partialx\dB+V\cdot \nabla B_2=\da ,\\
&\partial_t \dV+V_1\cdot\partialx\dV+\dV \cdot\nabla V_2 +\ma_2 \dzeta 
+\da \zeta_1=0,\\
&\partial_{t}\dzeta+V_2\cdot\partialx\dzeta+V\cdot \nabla\zeta_1=G(\eta_1)\dV+ \zeta_1 G(\eta_1)\dB+\dzeta G(\eta_2)B_2+R+\gamma,
\end{aligned}
\right.
\end{equation}
where
\begin{equation}\label{defi:Runi}
R=\left[ G(\eta_1)-G(\eta_2)\right] V_2 +\zeta_1 \left[ G(\eta_1)-G(\eta_2)\right] B_2,
\end{equation}
and~$\gamma= \gamma_1 - \gamma_2$,~$\gamma_j$ are given by~\eqref{eq:zeta}.

\begin{lemm}\label{lemm:symmind}
The differences~$\zeta,B,V$ satisfy a system of the form
\begin{equation}\label{syst:delta}
\left\{
\begin{aligned}
&(\partial_t +V_1\cdot\partialx)(\dV+\zeta_1 B)+\ma_2 \dzeta =f_1 ,\\
&(\partial_{t}+V_2\cdot\partialx)\dzeta-G(\eta_1)\dV- \zeta_1 G(\eta_1)\dB=f_2,
\end{aligned}
\right.
\end{equation}
for some remainders such that
$$
\lA (f_1,f_2)\rA_{L^{\infty}([0,T];H^{s-1}\times H^{s-\tdm})} \le \K(M_1,M_2)N(T).
$$
\end{lemm}
\begin{proof}
We begin by rewriting System~\eqref{BVzeta2} under the form
\begin{equation*}
\left\{
\begin{aligned}
&\partial_{t}\dB+V_1\cdot\partialx\dB=\da +R_1 ,\\
&\partial_t \dV+V_1\cdot\partialx\dV+\ma_2 \dzeta+ \da \zeta_1=R_2,\\
&\partial_{t}\dzeta+V_2\cdot\partialx\dzeta =G(\eta_1)\dV+ \zeta_1 G(\eta_1)\dB+R+\gamma+R_3,
\end{aligned}
\right.
\end{equation*}
where~$R$ is given by \eqref{defi:Runi},~$\gamma=\gamma_1-\gamma_2$ and
$$
R_1=-V\cdot \nabla B_2,\quad R_2=-\dV \cdot\nabla V_2 ,\quad 
R_3=- V\cdot \nabla\zeta_1+ \dzeta G(\eta_2)B_2.
$$
 From  Theorem \ref{L:p60} one has 
$$
\lA R\rA_{L^\infty(0,T;H^{s-\tdm})} \le \K(M_1,M_2) N(T).
$$
Similarly, proceeding as in the end of the proof of Proposition~\ref{prop:newS}, we have
$$
\lA \gamma\rA_{L^\infty(0,T;H^{s-\tdm})} \le \K(M_1,M_2) N(T).
$$
On the other hand, since~$s-1>d/2$,~$H^{s-1}$ is an algebra and
$$
\lA V\cdot \nabla B_2\rA_{H^{s-1}}\le K \lA V\rA_{H^{s-1}}\lA \nabla B_2\rA_{H^{s-1}}
\le K \lA V\rA_{H^{s-1}}\lA B_2\rA_{H^{s}}
$$
and similarly
$$
\lA V\cdot \nabla V_2\rA_{H^{s-1}}\le K \lA V\rA_{H^{s-1}}\lA V_2\rA_{H^{s}}.
$$
On the other hand, according to Theorem~\ref{coro:estiDN} we have 
$$
\lA G(\eta_2)B_2\rA_{H^{s-1}}
\le C(\lA \eta_2\rA_{H^{s+1/2}}) \lA B_2\rA_{H^s},
$$
and hence
$$
\lA \dzeta G(\eta_2)B_2\rA_{H^{s-\tdm}}\le C(\lA \eta_2\rA_{H^{s+1/2}}) \lA B_2\rA_{H^s} \lA \zeta\rA_{H^{s-\tdm}}.
$$
To estimate~$V\cdot \nabla \zeta_1$ we use the product rule \eqref{pr} to deduce
$$
\lA V\cdot \nabla \zeta_1\rA_{H^{s-\tdm}}\le K \lA V\rA_{H^{s-1}}\lA \nabla \zeta_1\rA_{H^{s-\tdm}}
\le  K \lA V\rA_{H^{s-1}}\lA \eta_1\rA_{H^{s+\mez}}.
$$
Therefore we have,
$$
\lA R_1\rA_{H^{s-1}}+\lA R_2\rA_{H^{s-1}}+\lA R_3\rA_{H^{s-\tdm}}\le 
C \left\{ \| \eta \|_{H^{s-\mez}}+\lA B\rA_{H^{s-1}}+\lA V\rA_{H^{s-1}}\right\},
$$
for some constant~$C$ depending only on~$\lA \eta_j\rA_{H^{s+\mez}},\lA B_j\rA_{H^s},\lA V_j\rA_{H^s}$.
The next step consists in transforming again the equation. 
We want to replace~$\ma\zeta_1$ in the second equation by 
$$
(\partial_{t}\dB+V_1\cdot\partialx\dB - R_1)\zeta_1.
$$
The idea is that this allows to factor out the convective derivative~$\partial_t +V_1\cdot \nabla~$. 
Writing
$$
(\partial_{t}\dB+V_1\cdot\partialx\dB)\zeta_1
=(\partial_{t}+V_1\cdot\partialx)(\dB \zeta_1)-\dB(\partial_{t}+V_1\cdot\partialx) \zeta_1
$$
we thus end up with
\begin{equation}\label{eq.p644}
(\partial_t +V_1\cdot\partialx)(\dV+\zeta_1 B)+\ma_2 \dzeta =R_1 \zeta_1 + \dB(\partial_{t}+V_1\cdot\partialx) \zeta_1+R_2.
\end{equation}
Since 
$$
(\partial_{t}+V_1\cdot\partialx)\zeta_1=G(\eta_1)V_1+ \zeta_1 G(\eta_1)\B_1+ \gamma_1,
$$
we have
$$
\lA (\partial_{t}+V_1\cdot\partialx)\zeta_1\rA_{H^{s-1}}\le \mathcal{F}\bigl(\lA (\eta_1,B_1,V_1)\rA_{H^{s+\mez}\times H^s\times H^s}\bigr).
$$
By using this estimate and our previous bounds for~$R_1,R_2$, we find
$$
\lA R_1 \zeta_1 + \dB(\partial_{t}+V_1\cdot\partialx) \zeta_1+R_2\rA_{H^{s-1}}\le 
C \left\{ \| \eta \|_{H^{s-\mez}}+\lA B\rA_{H^{s-1}}+\lA V\rA_{H^{s-1}}\right\},
$$
for some constant~$C$ depending only on 
$\lA \eta_j\rA_{H^{s+\mez}},\lA B_j\rA_{H^s},\lA V_j\rA_{H^s}$. 
Notice that here, as we used the equation satisfied by~$\zeta_1$, 
it was important to have~$(\partial_t+ V_1 \cdot \nabla )$ in the l.h.s. of~\eqref{eq.p644} 
and not~$(\partial_t+ V_2 \cdot \nabla )$, and this algebraic reduction required 
some care in the previous step.  
\end{proof}
\subsection{Estimates for the good unknown}
We now symmetrize System~\eqref{syst:delta}.  We   set~$I = [0,T].$
\begin{lemm} \label{symetr}
Set
$$
\lambda_1\defn \sqrt{(1+|\nabla \eta_1|^2)|\xi|^2-(\nabla\eta_1\cdot\xi)^2},
$$
and
$$
\ell \defn \sqrt{ \lambda_1 \ma_2},\quad \varphi\defn T_{\sqrt{\lambda_1}}(\dV+\zeta_1 B),\quad \vartheta\defn T_{\sqrt{\ma_2}} \dzeta.
$$
Then 
\begin{align}
&(\partial_t +T_{V_1}\cdot\partialx)\varphi + T_{\ell}\vartheta =g_1 ,\label{syst:delta5}\\
&(\partial_{t}+T_{V_2}\cdot\partialx)\vartheta -T_{\ell }\varphi =g_2,\label{syst:delta5bis}
\end{align}
where
$$
\lA (g_1,g_2)\rA_{L^{\infty}(I;H^{s-\tdm}\times H^{s-\tdm})} \le \K(M_1,M_2)N(T).
$$
\end{lemm}
\begin{proof}
We start from Lemma \ref{lemm:symmind}. By using Proposition~\ref{lemPa},  
one can replace~$V_1\cdot \partialx~$ by~$T_{V_1}\cdot \partialx$ and~$\ma_2\dzeta$ by~$T_{\ma_2}\dzeta$, modulo admissible remainders. 
It is found that 
\begin{equation}\label{syst:delta2}
(\partial_t +T_{V_1}\cdot\partialx)(\dV+\zeta_1 B)+T_{\ma_2} \dzeta =f'_1 ,\end{equation}
for some remainder~$f_1'$ such that
$$
\lA f'_1\rA_{L^{\infty}( I;H^{s-1})} \le \K(M_1,M_2)N(T).
$$
Similarly, one can replace 
$V_2\cdot\partialx$ by~$T_{V_2}\cdot \partialx$. 
According to Proposition~\ref{coro:paraDN1}, with~$\eps = \mez$, we have
$$
\|G(\eta_1) V - T_{\lambda_1}V \|_{L^\infty(I; H^{s-\frac 3 2})} 
+  \|G(\eta_1) B - T_{\lambda_1}B \|_{L^\infty(I; H^{s-\frac 3 2})}\leq \K(M_1) N(T),
$$
and according to Proposition~\ref{lemPa}, with~$\gamma = r = s -\frac 3 2, \mu= s- \mez$,
$$
\| \zeta_1 G( \eta_1) B- T_{\zeta_1} T_{\lambda_1}B \|_{L^\infty(I; H^{s-\frac 3 2})}\leq \K (M_1) N(T).
$$
We deduce 
\begin{equation}\label{syst:delta3}
(\partial_{t}+T_{V_2}\cdot\partialx)\dzeta-T_{\lambda_1}\dV- T_{\zeta_1}T_{\lambda_1}\dB=f'_2,
\end{equation}
where 
$$
\lA f'_2\rA_{L^{\infty}(I;H^{s-\tdm})} \le \K(M_1,M_2)N(T).
$$
Now, according to Lemma~\ref{lemm:Dt1},~\eqref{eq.lambda} and ~\eqref{eq:zeta}  and we find that
\begin{multline}
\| [ T_{\sqrt{ \lambda_1}}, ( \partial_t + T_{V_1} \cdot \nabla)]\|_{ H^{s-1} \rightarrow H^{s - \tdm}} \\
\leq \K (M_1) \bigl(\mathcal{M}_0^{\mez} ( \sqrt{ \lambda_1} )+ \mathcal{M}_0^{\mez} ( (\partial_t + V_1\cdot \nabla) \sqrt{ \lambda_1} )\bigr) \leq \K ' (M_1) 
\end{multline}
and similarly,  according to Lemma~\ref{lemm:Dt1} and~\eqref{esti:a2}, 
\begin{multline}
\| [ T_{\sqrt{ a_2}}, ( \partial_t + T_{V_2} \cdot \nabla)]\|_{ H^{s-1} \rightarrow H^{s - \tdm}} \\
\leq \K (M_2) \bigl(\mathcal{M}_0^{0} ( \sqrt{a_2} )+ \mathcal{M}_0^{0} ( (\partial_t + V_2\cdot \nabla) \sqrt{ a_2} )\bigr) \leq \K ' (M_2), 
\end{multline}
which implies 
\begin{align}\label{syst:delta4}
&(\partial_t +T_{V_1}\cdot\partialx)T_{\sqrt{\lambda_1}}(\dV+\zeta_1 B)+ T_{\sqrt{\lambda_1}}T_{\ma_2}\dzeta =f''_1 ,\\
&(\partial_{t}+T_{V_2}\cdot\partialx)T_{\sqrt{\ma_2}} \dzeta-T_{\sqrt{\ma_2}}\bigl(T_{\lambda_1}\dV- T_{\zeta_1}T_{\lambda_1}\dB\bigr)=f''_2,\label{syst:delta4bis}
\end{align}
where 
$$
\lA (f''_1,f''_2)\rA_{L^{\infty}(I;H^{s-\tdm})} \le \K(M_1,M_2)N(T).
$$
According to~\eqref{esti:quant2},~\eqref{eq.lambda} and~\eqref{esti:a1}, 
since~$s> 1 + \frac d 2$,
$$
T_{\sqrt{\lambda_1}} T_{a_2}- T_{\sqrt{\lambda_1a_2}}T_{\sqrt{ a_2}}\quad\text{is of order }0,
$$
which implies~\eqref{syst:delta5}. On the other hand, 
according to~\eqref{esti:quant2} and~\eqref{eq.lambda} 
the operators~$T_{\zeta_1}T_{\lambda_1}- T_{\lambda_1 \zeta_1}$ 
and~$T_{\lambda_1}T_{\zeta_1}- T_{\lambda_1 \zeta_1}$ 
are of order~$1/2$ (with norm controlled by~$\K (M_1)$, 
which allows to commute~$T_{\sqrt{\lambda_1}}$ and~$T_{\ma_2}$ 
in~\eqref{syst:delta4bis}). 
Now, according to ~Proposition~\ref{lemPa} (with~$\gamma =r = s - \mez, \mu= s-1$) 
$$ \| T_{\zeta_1} B- \zeta_1 B\|_{ H^{s- \mez}} \leq \K( M_1) \| B\|_{H^{s-1}}.
$$  
Which  implies~\eqref{syst:delta5bis} (using again~\eqref{esti:quant2}).
\end{proof}

Recall that we have set
\begin{equation}
N(T)\defn \sup_{t\in I} \lA (\eta,\psi,V,B)(t)\rA_{H^{s-\mez}\times H^{s-\mez}\times H^{s-1}\times H^{s-1}}.
 \end{equation}

\begin{lemm}\label{L63.a}
Set
$$
N'(T)\defn  \sup_{t\in I} \big\{\lA \vartheta(t)\rA_{H^{s-\tdm}}
+ \lA \varphi(t)\rA_{H^{s-\tdm}}\big \}.
$$
We have
\begin{equation}\label{esti:N'}
N'(T)\le \mathcal{K}(M_1,M_2)\bigl( N(0)+TN(T)\bigr).
\end{equation}
\end{lemm}
\begin{proof}
We first prove that 
\begin{equation}\label{esti:N'-1}
N'(T)\le \mathcal{K}(M_1,M_2)\bigl( N'(0)+TN(T)+TN'(T)\bigr).
\end{equation}
The desired estimate~\eqref{esti:N'} then follows from the fact that 
\begin{equation}\label{eq.N'}
N'(0) \le \K (M_1, M_2) N(0), \qquad N'(T)\le \K(M_1,M_2)N(T). 
\end{equation}
which follows from the continuity of 
paradifferential operators in the Sobolev spaces 
(see Theorem \eqref{theo:sc0}) and 
the fact that~$H^{s-1}({\mathbf{R}}^d)$ is an algebra since~$s>1+ \frac{d}{2}$.

The proof of \eqref{esti:N'-1} is based on a classical argument : we 
commute $\lDx{s-3/2}$ to \eqref{syst:delta5}--\eqref{syst:delta5bis} 
and perform an $L^2$ estimate. Then the key points are that (see point~$(iii)$ in Theorem~\ref{theo:sc0})
\begin{equation}\begin{aligned}
\lA (T_{V_j}\cdot \partialx)^* + T_{V_j}\cdot \partialx \rA_{L^2\rightarrow L^2}&\les \lA V_j\rA_{W^{1,\infty}},\\
\lA T_\ell -(T_\ell)^*\rA_{L^2\rightarrow L^2}&\le
\mathcal{F}\bigl( \lA (\eta_1,\eta_2)\rA_{W^{3/2,\infty}}\bigr)
\end{aligned}
\end{equation}
and that the commutators~$[T_{V_j} \cdot \nabla, \lDx{s-3/2}]~$ are,  according to~\eqref{esti:quant2}, of order~$s- \tdm$. 

Notice that since 
$(\varphi,\vartheta)\in C^1([0,T_0];H^{s-\tdm})$, we do not need to regularize the equations to justify the computations.
\end{proof}

Finally, let us notice that an elementary argument allows to control lower norms of~$(V,B)$ (and hence also of~$V+ \zeta_1B$): 
\begin{equation}\label{p64.8}
\lA (\dV,\dB)\rA_{L^\infty(I; H^{s-\frac{3}{2}})}\le \K(M_1,M_2)\bigl(N(0)+TN (T)\bigr).
\end{equation}
Indeed,  (the proof of)  Theorem \ref{L:p60} implies that  (with ~$\da=\ma_1-\ma_2$)
\begin{equation}
\lA \da \rA_{L^\infty(I; H^{s-\frac{3}{2}})}\le \K(M_1,M_2)N(T).
\end{equation}
Since~$\partial_{t}\dB+V_1\cdot\partialx\dB=\da-V\cdot \nabla B_2$, we have 
\begin{multline}
 \| B\|_{L^\infty(I;H^{s-2})} \leq \|B(0)\|_{H^{s-2}} + \int_0^T \Bigl(\| V_1 \cdot \nabla B\|_{H^{s-2}} + \| a\|_{H^{s-2}} + \| V \cdot \nabla B_2\|_{H^{s-2}}\Bigr) dt' \\
 \leq \|B(0)\|_{H^{s-2}}+ T \K (M_1, M_2) N(T).  
 \end{multline}
 Similarly, we have 
 \begin{equation}\label{p64.9} \| V\|_{L^\infty(I;H^{s-2})} \leq 
  \|V(0)\|_{H^{s-2}}+ T \K (M_1, M_2) N(T). 
 \end{equation}
 Now we have 
~$$V+ \zeta_1 B=  T_{\sqrt{ \lambda_1}^{-1}} \varphi + ( \text{Id} - T_{\sqrt{ \lambda_1}^{-1}}T_{\sqrt{\lambda_1}} ) (V+ \zeta_1 B),$$ where according to~\eqref{esti:quant2}, the operator~$\text{Id} - T_{\sqrt{ \lambda_1}^{-1}}T_{\sqrt{\lambda_1}}$ is of order$- 1/2$. Hence, we deduce from~\eqref{eq.N'}, \eqref{esti:N'}, ~\eqref{p64.8},~\eqref{p64.9} and a bootstrap argument 
 \begin{equation}\label{eq.p65.5}
  \| V+ \zeta_1 B\|_{L^\infty(I; H^{s-1})} \leq \K(M_1, M_2)\{ N(0) + T   N(T)\}.
  \end{equation}
\subsection{Back to the original unknowns}
Recall that $I = [0,T]$ (resp. $J= (-1,0)$) is an interval in the $t$ variable (resp. in the $z$ variable).   
\begin{lemm}\label{lem.N}
There holds
\begin{equation}\label{p64.10}
\| \eta \|_{L^\infty(I; H^{s-\mez})}
\le  \K(M_1, M_2) \{N(0)+  T N(T)\}.
\end{equation}
\end{lemm}
\begin{proof}
>From the equation~$\partial_t \eta_j = G(\eta_j)\psi_j$ we have,
$$
\eta(t) = \eta (0)+ \int_0^t G(\eta_1) \psi(t') dt' 
+ \int_0^t \big( G(\eta_1) - G(\eta_2)\big)\psi_2(t') dt' ,
$$
from which we deduce according to  Theorem \ref{L:p60},
\begin{equation}\label{eq.ap} \| \eta\|_{L^\infty(I; H^{s- 2})} \leq \|\eta(0) \|_{H^{s- 2}} + T \K (M_1, M_2) \| \eta\|_{L^\infty(I; H^{s- \mez})}.
\end{equation}
Let~$R = Id - T_{\frac{1}{\sqrt{a_2}}} T_{\sqrt{a_2}}$, which, according to~\eqref{esti:quant2} and~\eqref{esti:a1} is an operator of order~$- \mez$ (with norm estimated by~$\K ( M_2)$).  We have 
$$ \nabla \eta = R \nabla \eta + T_{\frac{1}{\sqrt{a_2}}} \vartheta.$$
Therefore we deduce from~\eqref{eq.ap}, \eqref{esti:N'} and a bootstrap argument,
$$
\lA \nabla \eta\rA_{L^\infty(I; H^{s-\frac{3}{2}})}
\leq \K(M_1, M_2) \bigl(N(0) + T N(T)\bigr).
$$
Combining with~\eqref{eq.ap} gives Lemma~\ref{lem.N}.
\end{proof}

We are now ready to estimate $(V,B)$.
\begin{prop}\label{lem.NN}
There holds
\begin{equation}\label{p64.10bis}
\| (V,B) \|_{L^\infty(I; H^{s-1})}
\le  \K(M_1, M_2) \big\{N(0)+  T N(T)\big\}.
\end{equation}
\end{prop}

The proof will require several preliminary Lemmas. 
We begin by noticing that it is enough to estimate $B$. Indeed, if
$$
\| B \|_{L^\infty(I; H^{s-1})}
\le  \K(M_1, M_2) \big\{N(0)+  T N(T)\big\},
$$
then, by using the triangle inequality, the estimate \eqref{eq.p65.5}Ê
for $V+\zeta_1 B$ implies that $V$ satisfies the desired estimate.

 Let~$v=\widetilde{\phi}_1-\widetilde{\phi}_2$, where $\widetilde{\phi}_j$ 
 is the harmonic extension in $\tilde\Omega$ of the function  $\psi_j$ and set 
$$ b_2\defn\frac{\partial_z \widetilde{\phi}_2}{\partial_z \rho_2},\quad 
w= v - T_{b_2}\rho.$$

 We notice that 
 \begin{equation}\label{trace:w}
 w\arrowvert_{z=0} = \psi - T_{B_2} \eta.
\end{equation}
We first state the following result.
  \begin{lemm}\label{est:trace}
We have 
\begin{equation}
\| \psi - T_{B_2} \eta \|_{L^\infty(I; H^{s})}
\le  \K(M_1, M_2)\big\{N(0)+  T N(T)\big\}.
\end{equation}
\end{lemm}
\begin{proof}
Indeed, the low frequencies are estimated by \eqref{p64.8}, 
while for the high frequencies, we write
\begin{align*}
\partialx (\psi-T_{B_2}\eta)&=\nabla \psi -T_{B_2}\nabla \eta-T_{\nabla B_2}\eta\\
&=\nabla \psi_1-\nabla\psi_2 -T_{B_2}\nabla \eta-T_{\nabla B_2}\eta\\
&=V_1+B_1\nabla\eta_1-V_2-B_2\nabla\eta_2-T_{B_2}\nabla\eta-T_{\nabla B_2}\eta\\
&=V+(B_1-B_2)\nabla\eta_1+B_2(\nabla\eta_1-\nabla\eta_2)-T_{B_2}\nabla\eta-T_{\nabla B_2}\eta\\
&=V + \zeta_1 B + ( B_2 - T_{B_2}) \nabla \eta - T_{\nabla B_2} \eta,
\end{align*}
where we used that, by definition, $\nabla\psi_j=V_j+B_j\nabla\eta_j$ and $\zeta_1=\nabla\eta_1$.

The main term $ V + \zeta_1 B$ is estimated 
using~\eqref{eq.p65.5}, while the two other terms are 
estimated using~\eqref{p64.10}, the {\em a priori} 
estimate on $B_2$ and the product rules~\eqref{Bony}  and~\eqref{niS}. 
\end{proof}

We next relate $w$, $\rho$ and $B$.
\begin{lemm}We have
$$
B=\Bigl[ \frac{1}{\partial_z\rho_1}\Bigl( \partial_z w -(b_2-T_{b_2})\partial_z\rho
+T_{\partial_z b_2}\rho\Bigr)\Bigr]\Big\arrowvert_{z=0}.
$$
\end{lemm}
\begin{proof}
Write
\begin{align*}
B_1-B_2&=\frac{\partial_z \widetilde{\phi}_1}{\partial_z \rho_1}
-\frac{\partial_z \widetilde{\phi}_2}{\partial_z \rho_2}\Big\arrowvert_{z=0}\\
&=\frac{1}{\partial_z\rho_1}\bigl(\partial_z\widetilde{\phi}_1-\partial_z\widetilde{\phi}_2\bigr)
+\bigr(\frac{1}{\partial_z\rho_1}-\frac{1}{\partial_z\rho_2}\Bigr)\partial_z\widetilde{\phi}_2\Big\arrowvert_{z=0}\\
&=\frac{1}{\partial_z\rho_1}\partial_z v
-\frac{1}{\partial_z\rho_1}\frac{\partial_z \widetilde{\phi}_2}{\partial_z \rho_2}\partial_z\rho\Big\arrowvert_{z=0}
\end{align*}
and replace $v$ by $w+T_{b_2}\rho$ in the last expression.
\end{proof}

\begin{lemm}\label{esti:b2012}
Recall that $ b_2\defn\frac{\partial_z \widetilde{\phi}_2}{\partial_z \rho_2}.$ For $k=0,1,2$, we have
$$
\lA \partial_z^k b_2\rA_{C^0([-1,0], L^\infty(I, H^{s-\mez-k}))}\le C \lA \psi_2\rA_{H^{s+\mez}},
$$
for some constant $C$ depending only on $\lA \eta_2\rA_{H^{s+\mez}}$. 
\end{lemm}
\begin{proof}
We estimate $\nabla_{x,z}\widetilde{\phi}_2$ in 
$C^0([-1,0], L^\infty(I, H^{s-\mez}))$ 
by using the elliptic regularity 
(see Proposition~\ref{p1} and Remark~\ref{rema:321}). 
Now, using the equation satisfied by $\widetilde{\phi}_2$ and the product rule in Sobolev spaces, we successively estimate $\partial_z^2\widetilde{\phi}_2$ and $\partial_z^3\widetilde{\phi}_2$. 
This proves the lemma since the derivatives of $\rho_2$ are estimated directly from the definition of $\rho_2$.
\end{proof}

Notice that $\eta$ and hence $\rho$ are estimated in $L^\infty(I; H^{s-\mez})$ (see~\eqref{p64.10}). 
Now, use Lemma~\ref{esti:b2012} and  
Proposition~\ref{lemPa} (applied with $s>1+d/2$, $\gamma=s-1$, $r=s-1/2$, $\mu=s-3/2$) to obtain
$$
\lA (b_2-T_{b_2})\partial_z\rho\rA_{H^{s-1}}\les \lA b_2\rA_{H^{s-\mez}}\lA \eta\rA_{H^{s-1}}.
$$
Now, \eqref{boundpara} implies that
$$
\lA T_{\partial_z b_2}\rho\rA_{H^{s-1}}\les \lA b_2\rA_{H^{s-1/2}}\lA \eta\rA_{H^{s-1}},
$$
and hence, to complete the proof of the Proposition \ref{lem.NN}, it remains only to estimate $\partial_z w\arrowvert_{z=0}$ 
in $L^\infty(I,H^{s-1})$. This is the purpose of the following result.

 \begin{lemm}\label{reg:w}
 For $t\in [0,T]$ we have   
 \begin{equation}
\| \nabla_{x,z} w  \|_{C^0([-1,0],   H^{s-1}) }
\le  \K(M_1, M_2) \big\{N(0)+  T N(T)\big\}.
\end{equation}
\end{lemm}
\begin{proof} To prove this estimate, we are going to show that $w$ 
satisfies an elliptic equation in the variables $(x,z)$ to which we may 
apply the results of Proposition~\ref{p1}.  We have 
\begin{equation*}
\partial_z^2 v +\alpha_1\Delta v + \beta_1 \cdot\partialx\partial_z v 
- \gamma_1 \partial_z v=(\gamma_1-\gamma_2)\partial_z \widetilde{\phi}_2+F_1,
\end{equation*}
where (see \eqref{eq-v})
$$
F_1=(\alpha_2-\alpha_1)\Delta \widetilde{\phi}_2+(\beta_2-\beta_1)\cdot\nabla\partial_z\widetilde{\phi}_2.
$$
We claim that for $t\in [0,T]$
\begin{equation}\label{est::F1}
\Vert F_1(t,\cdot) \Vert_{ L ^2(J,H^{s-\frac{3}{2}} )} \leq   \K(M_1, M_2) \{N(0)+  T N(T)\}. 
\end{equation}
The two terms in $F_1$ are estimated by the same way. 
We will only consider the first one. 
Using the product rule \eqref{pr} with $s_0 = s-\frac{3}{2}, s_1 = s-1, s_2 = s-\frac{3}{2}$ 
we can write for fixed $t$,
$$\Vert(\alpha_2-\alpha_1) \Delta \widetilde{\phi}_2\Vert_{L^2(J, H^{s-\frac{3}{2}}) }  \leq C\Vert  \alpha_2-\alpha_1   \Vert_{L^2(J, H^{s-1})}  \Vert  \Delta \widetilde{\phi}_2\Vert_{L^\infty(J, H^{s-\frac{3}{2}})}.
$$
Then we use \eqref{eq.6.4}, Proposition \ref{p1} 
with $\sigma = s-\mez$ and Lemma \ref{lem.N} 
to conclude that the term above is estimated by the right hand side of \eqref{est::F1}.
 
Now we introduce the operators
$$
P_j\defn \partial_z^2+\alpha_j\Delta+\beta_j\cdot\partialx\partial_z, \qquad L_j = P_j - \gamma_j \partial_z, \qquad (j=1,2).
$$
With these notations we have $\gamma_j = \frac{1}{\partial_z \rho_j}P_j \rho_j$ and 
\begin{equation}\label{L1vG1}
L_1 v=(\gamma_1-\gamma_2)\partial_z \widetilde{\phi}_2+F_1.
\end{equation}
Moreover
\begin{align*}
\gamma_1-\gamma_2&=\frac{1}{\partial_z\rho_1}P_1 \rho_1   - \frac{1}{\partial_z\rho_2}P_2 \rho_2  = \frac{1}{\partial_z\rho_2}P_1 \rho_1  + \Bigl(\frac{1}{\partial_z\rho_1}-\frac{1}{\partial_z\rho_2}\Bigr)P_1 \rho_1 - \frac{1}{\partial_z\rho_2}P_2 \rho_2 \\
&=\frac{1}{\partial_z\rho_2} P_1  \rho + \frac{1}{\partial_z \rho_2}(P_1- P_2) \rho_2 + \Bigl(\frac{1}{\partial_z\rho_1}-\frac{1}{\partial_z\rho_2}\Bigr)
P_1\rho_1 \\
&=\frac{1}{\partial_z\rho_2}P_1\rho
+\Bigl(\frac{1}{\partial_z\rho_1}-\frac{1}{\partial_z\rho_2}\Bigr)
P_1\rho_1+F_2,
\end{align*}
where
\begin{equation}\label{eq::F2}
F_2=\frac{1}{\partial_z\rho_2}\Bigl((\alpha_1-\alpha_2)\Delta \rho_2
+(\beta_1-\beta_2)\cdot\partialx\partial_z\rho_2\Bigr).
\end{equation}
 Now we observe that
$$
\Bigl(\frac{1}{\partial_z\rho_1}-\frac{1}{\partial_z\rho_2}\Bigr)
P_1\rho_1=-\frac{\partial_z \rho}{\partial_z\rho_2} \frac{P_1\rho_1}{\partial_z\rho_1}
=-\frac{\partial_z \rho}{\partial_z\rho_2}\gamma_1,
$$
which implies  
$$
\gamma_1-\gamma_2=\frac{1}{\partial_z\rho_2}P_1\rho-\frac{\partial_z \rho}{\partial_z\rho_2}\gamma_1
+F_2=\frac{1}{\partial_z\rho_2}L_1\rho+F_2.
$$
Plugging this into \eqref{L1vG1} yields
\begin{equation}\label{eq::1v}
L_1 v -b_2(L_1 \rho)=F_1+(\partial_z\widetilde{\phi}_2) F_2.
\end{equation}
We claim that for fixed $t$ we have
\begin{equation}\label{est::F2}
\lA (\partial_z\widetilde{\phi}_2) F_2 (t,\cdot) \rA_{ L ^2(J,H^{s-\frac{3}{2}} )} \leq   \K(M_1, M_2) \{N(0)+  T N(T)\}. 
\end{equation}
Indeed we first use the product rule \eqref{pr} to write
$$
\lA (\partial_z\widetilde{\phi}_2) F_2 (t,\cdot) \rA_{ L ^2(J,H^{s-\frac{3}{2}} )} 
\leq \Vert (\partial_z\widetilde{\phi}_2)   (t,\cdot) \Vert_{ L ^2(J,H^{s-\mez} )}\lA   F_2 (t,\cdot) \rA_{ L ^2(J,H^{s-\frac{3}{2}} )}.
$$
By the elliptic regularity the first term in the right hand side is bounded by $\mathcal{K}(M_2).$ 
It is therefore sufficient to bound the second one. 
We have, for fixed $t$
 $$\lA \frac{1}{\partial_z\rho_2} (\alpha_1-\alpha_2)\Delta \rho_2\rA_{L^2(J,H^{s-\frac{3}{2}})} 
 \leq \mathcal{K}(M_2) \Vert  \alpha_1-\alpha_2  \Vert_{L^2(J,H^{s-1})} \Vert  \Delta \rho_2\Vert_{L^\infty(J,H^{s-\frac{3}{2}})}.$$
 Using \eqref{eq.6.4} and \eqref{eq.rho1} we see that the 
 right hand side is bounded by the right hand side of \eqref{est::F2}. 
 The second term in $F_2$ is estimated by the same way.
 
To estimate ~$v-T_{b_2}\rho$ we paralinearize in writing
\begin{equation}\label{eq::2v}
 b_2(L_1 \rho)=T_{b_2}L_1 \rho
+T_{L_1\rho}b_2 +F_3.
 \end{equation}
 
We claim that for $t\in [0,T]$
\begin{equation}\label{est::F3}
\Vert F_3(t,\cdot) \Vert_{ L ^2(J,H^{s-\frac{3}{2}} )} \leq   \K(M_1, M_2) \{N(0)+  T N(T)\}. 
\end{equation}
To prove it we shall use \eqref{Bony} 
with $\alpha = s-\mez, \beta = s- 2.$ Then $\alpha + \beta -\frac{d}{2} > s-\frac{3}{2}.$ 
It follows that, for fixed $z$ and $t$ we have
$$
\Vert F_3(t, \cdot, z)\Vert_{H^{s- \frac{3}{2}}} \leq C \Vert b_2 \Vert_{H^{s-\mez}} \Vert L_1 \rho \Vert_{H^{s-2}}. 
$$
Therefore
$$
\Vert F_3(t, \cdot)\Vert_{L^2(J,H^{s- \frac{3}{2}})} 
\leq C \Vert b_2 \Vert_{L^\infty(J,H^{s-\mez})} \Vert L_1 \rho \Vert_{L^2(J,H^{s-2})}.
$$
Now as we have seen before we have $\Vert b_2(t\cdot) \Vert_{L^\infty(J,H^{s-\mez})}\leq \mathcal{K}(M_2)$ 
and due to the smoothing of the Poisson kernel $\Vert L_1 \rho \Vert_{L^2(J,H^{s-2})} \leq \mathcal{K}(M_1)\Vert \eta \Vert_{H^{s-\mez}}.$ 
The estimate \eqref{est::F3} thus follows from \eqref{p64.10}.

Setting $F_4 =  T_{L_1\rho}b_2$ we claim that for fixed $t$ we have 
\begin{equation}\label{est::F4}
\Vert F_4(t,\cdot) \Vert_{ L ^2(J,H^{s-\frac{3}{2}} )} \leq   \K(M_1, M_2) \{N(0)+  T N(T)\}. 
\end{equation}
To see this we use \eqref{boundpara} with $s_0 = s-\frac{3}{2}, s_1 = s-2,s_2 = s-\mez.$ 
We get
$$
\Vert F_4(t,\cdot) \Vert_{ L ^2(J,H^{s-\frac{3}{2}})}
\leq \Vert L_1 \rho(t,\cdot) \Vert_{L^2(J, H^{s-2})} \Vert b_2(t,\cdot) \Vert_{L^\infty(J, H^{s-\mez})}
$$
and \eqref{est::F4} follows from estimates used above.
 
Now according to \eqref{eq::1v}, \eqref{eq::2v} we have 
$$
L_1v -T_{b_2} L_1 \rho = F_1 +(\partial_z\widetilde{\phi}_2)F_2 + F_3 +F_4.
$$
We claim that we have 
$$
L_1 T_{b_2}\rho = T_{b_2}L_1 \rho - F_5
$$
with 
\begin{equation*}
\Vert F_5(t,\cdot) \Vert_{ L ^2(J,H^{s-\frac{3}{2}} )} \leq   \K(M_1, M_2) \{N(0)+  T N(T)\}. 
\end{equation*}
To see this we use \eqref{esti:b2012} and \eqref{boundpara}. It follows then that we have
$$
L_1 w = L_1(v -T_{b_2}\rho) =  F_1 +(\partial_z\widetilde{\phi}_2)F_2 + F_3 +F_4 +F_5 := F
$$
where $\Vert F(t,\cdot) \Vert_{ L ^2(J,H^{s-\frac{3}{2}})}$ is bounded by the right hand side of \eqref{est::F4}. 
 
Using \eqref{trace:w} and Lemma \ref{est:trace} we may then apply to $w$ 
Proposition \ref{p1} with $\sigma = s-1$ to conclude the proof of Lemma \ref{reg:w} and thus that of Proposition \ref{lem.NN}.
\end{proof}
 
\section{Well-posedness of the Cauchy problem}\label{S:regul}

Here we conclude the proof of Theorem~\ref{theo:Cauchy} about the 
Cauchy theory 
for the system
\begin{equation}\label{system:A1}
\left\{
\begin{aligned}
&\partial_t \eta+G(\eta)\psi=0,\\
&\partial_{t}\psi+g \eta
+ \frac{1}{2}\la\partialx  \psi\ra^2  -\frac{1}{2}
\frac{\bigl(\partialx   \eta\cdot\partialx \psi +G(\eta) \psi \bigr)^2}{1+|\partialx  \eta|^2} = 0.
\end{aligned}
\right.
\end{equation}

We previously proved the uniqueness of solutions (see~Theorem~\ref{th.lipschitz}). 
To complete the proof of Theorem~\ref{theo:Cauchy}, 
we prove here the existence of solutions to the water waves system 
as limits of smooth solutions to 
approximate systems. 
This approach has been detailed in \cite{ABZ1}, where 
we considered the problem with surface tension. 


To explain the scheme of the proof, we first consider the case without bottom ($\Gamma=\emptyset$). 
Then 
we know, 
from previous results (see Wu~\cite{WuInvent,WuJAMS}, Lannes~\cite{LannesJAMS}, 
Lindblad~\cite{LindbladAnnals}), that the Cauchy problem 
is well-posed for smooth initial data. 
Then, one can obtain the existence of smooth approximate 
solutions in a straightforward way : by smoothing the initial data. 
Namely, denote by~$J_\eps$ 
the usual Friedrichs mollifiers, defined by 
$J_\eps = \jmath(\eps D_x)$ where~$\jmath \in C^\infty_0({\mathbf{R}}^d)$,~$0\le \jmath\le 1$, is such that 
$$
\jmath(\xi)=1\quad \text{for }\la \xi\ra\le 1,\qquad 
\jmath(\xi)=0\quad\text{for }\la \xi\ra\ge 2.
$$
Set~$\psi_0^\eps=J_\eps\psi_0$ and~$\eta_0^\eps=J_\eps\eta_0$. Then 
$(\psi_0^\eps,\eta_0^\eps)\in H^\infty({\mathbf{R}}^d)^2$ and the Cauchy problem for~\eqref{system:A1} 
has a unique smooth solution~$(\psi^\eps,\eta^\eps)$ defined on some time interval~$[0,T_\eps^*)$. 
It follows from 
Proposition~\ref{prop:apriori} that there exists a function~$\mathcal{F}$ such that, for all 
$\eps\in (0,1]$ and all~$T<T_\eps$, we have
\begin{equation}\label{apriorieps}
M_s^\eps(T)\le \mathcal{F}\bigl(\mathcal{F}(M_{s,0})+T\mathcal{F}\bigl(M_s^\eps(T)\bigr)\bigr),
\end{equation}
with obvious notations. Then by standard argument, 
we infer that the lifespan of~$(\eta_\eps,\psi_\eps)$ is 
bounded from below by a positive time~$T_0$ 
independent of~$\eps$ 
and that we have uniform estimates 
on~$[0,T_0]$. 
The fact that one can pass to the limit in the equations follows 
from the previous contraction estimates (see~\eqref{eq.lipschitz}), 
which allows us to prove that~$(\eta_\eps,\psi_\eps,B_\eps,V_\eps)$ 
 is a Cauchy sequence (this argument has been explained in~\cite{ABZ1}). Notice that these estimates were proved under the assumption $a(t)>a_0/2$. This actually follows from the {\em a priori} bound 
 \eqref{apriorieps}, \eqref{esti:a1}, \eqref{esti:a2} and a bootstrap method. Then, it remains to prove that the limit solution 
has the desired regularity properties. Again, this follows from the analysis in~\cite{ABZ1}. 


In the case with a general bottom, to apply the strategy explained above, 
the only remaining point is to prove that, for smooth enough initial data, 
the Cauchy problem 
has a smooth solution. 
Namely, for our purposes it is enough to prove the following weak well-posedness result 
(where we 
allow an arbitrarily large loss of~$N$ derivatives).

\begin{prop}\label{HFS}
For all~$s_0>0$ there exists~$N>0$ such that the following result holds. 
Consider an initial data~$(\eta_0,\psi_0)\in H^{s_0+N}({\mathbf{R}}^d)^2$ 
such that, initially, the Taylor coefficient 
$\ma$ is bounded from below by $a_0>0$.  Then, there exists~$T>0$ 
and a solution~$(\eta,\psi)\in C^1([0,T];H^{s_0}({\mathbf{R}}^d))^2$ to~$\eqref{system:A1}$ 
satisfying~$(\eta(0),\psi(0))=(\eta_0,\psi_0)$ and $a\ge  a_0 /2$.
\end{prop}
\begin{proof}
We use a parabolic regularization of the equations 
and seek solutions of the Cauchy problem~\eqref{system:A1} as limits of solutions of the following 
approximate systems:
\begin{equation}\label{A2}
\left\{
\begin{aligned}
&\partial_{t}\eta-G(\eta)\psi=\eps \Delta\eta,\\
&\partial_{t}\psi+g \eta
+ \frac{1}{2}\la\partialx  \psi\ra^2  -\frac{1}{2}
\frac{\bigl(\partialx   \eta\cdot\partialx \psi +G(\eta) \psi \bigr)^2}{1+|\partialx  \eta|^2}
= \eps \Delta\psi,
\\
& (\eta,\psi)\arrowvert_{t=0}=(J_\eps \eta_{0},J_\eps \psi_0),
\end{aligned}
\right.
\end{equation}
where~$\eps \in ]0,1]$ is a small parameter and 
\begin{equation}\label{defi:Beps}
B=\frac{\partialx \eta \cdot\partialx \psi+G(\eta) \psi}{1+|\partialx  \eta|^2}.
\end{equation}
Write \eqref{A2} under the form 
$$
u=e^{\eps t\Delta}u_0+\int_0^t e^{\eps (t-\tau)\Delta} \mathcal{A}[u(\tau)]\, d\tau,\quad 
u=(\eta,\psi).
$$
To solve this Cauchy problem we use two ingredients. Firstly we have 
$$
\lA e^{\eps t\Delta} f\rA_{L^2([0,T];H^{s+1})}\le K(\eps) \lA f\rA_{H^s},
$$
together with a dual estimate, and secondly (for $s$ large enough)
$$
\lA \mathcal{A}[u]\rA_{L^2([0,T];H^{s-1})}\le \sqrt{T}
\mathcal{F}(\lA u\rA_{L^\infty([0,T];H^{s})})\lA u\rA_{L^\infty([0,T];H^s)},
$$
which follows from Theorem~\ref{coro:estiDN} and the product rule in Sobolev spaces. Then, 
given $\eps>0$ and $(\eta_0,\psi_0)\in (H^{s}(\xR^d))^2$ with $s$ large enough, 
by applying the Banach fixed point theorem in 
$L^\infty([0,T];H^s)\cap L^2([0,T];H^{s+1})$, one gets that, 
for $T$ small enough possibly depending on $\eps$, the 
Cauchy problem for~\eqref{A2} 
has a smooth solution. Moreover, the maximal time of existence $T_\eps$ 
satisfies: either $T_\eps=+\infty$ or
$$
\limsup _{t\rightarrow T_\eps} \lA (\eta,\psi)(t,\cdot)\rA_{H^{s}\times H^s} = +\infty.
$$
We have to prove that these solutions exist for a time 
interval independent of $\eps\in ]0,1]$. To do so, we begin by computing the equations satisfied by $B$ and 
$V=\nabla\psi-B\nabla \eta$. 

We first observe that, 
since $G(\eta)\psi=B-V\cdot \nabla\eta$, we have
\begin{equation}\label{eq:epseta}
 \partial_t \eta+V\cdot\nabla \eta =B+\eps\Delta\eta.
\end{equation}
On the other hand, as in~\eqref{psiVB}, we have
\begin{equation}\label{eq:epspsi}
\partial_t \psi+V\cdot \nabla \psi
=-g\eta+\mez |V|^2-\mez B^2+\eps\Delta\psi.
\end{equation}
We next write that
\begin{align*}
\partial_t V +V\cdot \nabla V
&=(\partial_t +V\cdot \nabla)(\nabla\psi-B\nabla\eta)\\
&=\nabla (\partial_t\psi +V\cdot \nabla\psi)-(\partial_t B+V\cdot \nabla B)\nabla\eta
-B\nabla( \partial_t \eta+V\cdot\nabla \eta)+R
\end{align*}
where $R=(R_1,\ldots,R_d)$ with 
$$
R_k=-\sum_j \partial_k V_j \partial_j \psi+B\sum_j\partial_k V_j\partial_j\eta.
$$
Then, it follows from \eqref{eq:epseta} and \eqref{eq:epspsi} that
$$
\partial_t V+V\cdot \nabla V=\nabla \Bigl( -g\eta+\mez |V|^2-\mez B^2+\eps\Delta\psi\Bigr)
-(\partial_t B+V\cdot \nabla B)\nabla\eta+B\nabla \bigl( B+\eps\Delta\eta\bigr)+ R.
$$
Now, observing that $R+\mez \nabla |V|^2=0$ and simplifying,
$$
\partial_t V+V\cdot\nabla V 
+(g+\partial_t B+V\cdot\nabla B)\nabla \eta
=\eps\nabla \Delta\psi- \eps B\nabla\Delta\eta
$$
which in turn implies that
$$
\partial_t V+V\cdot\nabla V +\tilde a \zeta =\eps\Delta V+2\eps(\nabla\B\cdot\nabla)\zeta,
$$
where $\zeta=\nabla\eta$ and
\begin{equation}\label{defi:aeps}
\tilde a=g+\partial_t B+V\cdot\nabla B-\eps\Delta B.
\end{equation}
Now, for $s>5/2+d/2$, from the proof of \eqref{parastep1a} (using Lemma~\ref{lemm:Dt1} instead 
of Lemma~\ref{lemm:Dt} as in the proof of Lemma~\ref{lemm:symmind}), we obtain that 
$U=V+T_{\zeta}B$ satisfies
\begin{equation}\label{eq:Ureg}
\partial_t U+T_V\cdot\nabla U+T_{\tilde a}\zeta-\eps\Delta U -\eps T_{\nabla B}\cdot\nabla\zeta
=F_1
\end{equation}
where
$$
\lA F_1\rA_{L^\infty_t (H^s)}\le \mathcal{F}(\lA (\psi,\eta,V,B,\tilde{a})\rA_{L^\infty_t(H^{s+1/2}\times 
H^{s+1/2}\times H^s\times H^s\times C^{1/2})}).
$$
Also, as in the proof of Proposition~\ref{psym}, we find that $\theta=T_{q}\zeta$ (where $q=\sqrt{\tilde{a}/\lambda}$) satisfies
$$
\partial_t\theta+T_V\cdot \nabla\theta-T_\gamma U-\eps\Delta\theta
-\eps \bigl[ T_q ,\Delta\bigr] T_{q^{-1}}\theta =F_2,
$$
where $F_2$ is estimated by means of Theorem~\ref{theo:sc0} and \eqref{lemm:Dt1}:
$$
\lA F_2\rA_{L^\infty_t (H^s_x)}\le
\mathcal{F}(\lA (\psi,\eta,V,B,\tilde{a},\partial_t\tilde a+V\cdot\nabla\tilde a)
\rA_{L^\infty_t(H^{s+1/2}\times H^{s+1/2}\times H^s\times H^s\times C^{1/2}\times L^\infty)}).
$$
Using the unknown $\theta$, one can rewrite \eqref{eq:Ureg} as
\begin{equation}\label{eq:Ureg2}
\partial_t U+T_V\cdot\nabla U+T_{\tilde a}\zeta-\eps\Delta U -\eps T_{\nabla B}\cdot\nabla
T_{q^{-1}}\theta
=\tilde F_1
\end{equation}
where $\tilde{F}_1=F_1+\eps T_{\nabla B}\cdot\nabla(T_q T_{q^{-1}}-I)\zeta$ and it follows from 
Theorem~\ref{theo:sc0} that
$$
\lA T_{\nabla B}\cdot\nabla(T_q T_{q^{-1}}-I)\zeta\rA_{H^s}
\les \lA \nabla B\rA_{L^\infty}M^{1/2}_{3/2}(q)M^{-1/2}_{3/2}(q^{-1})
\lA \eta\rA_{H^{s+1/2}}.
$$ 
Since $M^{1/2}_{3/2}(q)M^{-1/2}_{3/2}(q^{-1})\le C(\lA \eta\rA_{C^{5/2}},\lA \tilde{a}\rA_{C^{3/2}})$ we obtain
$$
\big\| \tilde F_1\big\|_{L^\infty_t (H^s)}\le \mathcal{F}(\lA (\psi,\eta,V,B,\tilde{a})\rA_{L^\infty_t(H^{s+1/2}\times 
H^{s+1/2}\times H^s\times H^s\times C^{3/2})}).
$$
Now, since the operator
$$
\begin{pmatrix} \Delta  & T_{\nabla B}\cdot\nabla T_{q^{-1}} \\
0 & \Delta+\bigl[ T_q ,\Delta\bigr] T_{q^{-1}} \end{pmatrix}
$$
is elliptic (being a perturbation of $\Delta$ of order $3/2$), 
we estimate $(U,\theta)$ in $L^\infty_t(H^s\times H^s)$ 
by commuting the equation with $(I-\Delta)^{s/2}$ and performing an $L^2$ estimate. 
To do so it remains only to prove that
$$
\lA (\tilde{a},\partial_t\tilde a+V\cdot\nabla \tilde a)
\rA_{L^\infty_t( C^{3/2}\times L^\infty)}\le 
\mathcal{F}\Bigl(\lA (\psi,\eta,V,B)\rA_{L^\infty_t(H^{s+1/2}\times H^{s+1/2}\times H^s\times H^s)}\Bigr).
$$
Here this is much easier than in the above analysis. 
Indeed, since we do not need to work with critical regularity, 
it is enough to prove that there exist $s_2\ge s_1\ge 3/2+s_0$ with $s_0>d/2$ and $s_2\le s$
(possibly $s_0\ll s_1\ll s_2$) such that
\begin{align}
\lA \tilde{a}-g
\rA_{L^\infty_t( H^{s_1})}&\le 
\mathcal{F}\Bigl(\lA (\eta,\psi)\rA_{L^\infty_t(H^{s_2}\times H^{s_2})}\Bigr),\label{esti:tildea-s1}\\
\lA \partial_t\tilde a\rA_{L^\infty_t(H^{s_0})}&\le 
\mathcal{F}\Bigl(\lA (\eta,\psi)\rA_{L^\infty_t(H^{s_2}\times H^{s_2})}\Bigr).\label{esti:tildea-s0}
\end{align}
To prove \eqref{esti:tildea-s1} we express $\tilde a -g$ in terms of $\eta,\psi$ (as noticed in~\cite{LannesLivre}). 
To do so, we use two observations. 
Firstly, by definition of $B$ (see~\eqref{defi:Beps}), we have
$$
\partial_t B=
\frac{1}{1+|\partialx  \eta|^2}\Bigl(\partialx \partial_t\eta \cdot\partialx \psi
+\partialx \eta \cdot\partialx \partial_t\psi+\partial_t G(\eta) \psi\Bigr)
-\frac{2\nabla\partial_t \eta\cdot\nabla\eta}{1+|\partialx  \eta|^2}B.
$$
Secondly, we have
\begin{equation}\label{DNshape}
\partial_tG(\eta)\psi=G(\eta)(\partial_t\psi-B\partial_t \eta)-\cnx(V\partial_t\eta).
\end{equation}
The formula \eqref{DNshape} is proved by Lannes 
for smooth bottoms (see~\cite{LannesJAMS}) or infinite depth~(\cite{LannesLivre}). 
In the case with a general bottom considered here, this follows from~\cite{ABZ-Bertinoro}. 
Then, replacing $\partial_t \psi$ and $\partial_t\eta$ by their expressions computed using System~\eqref{A2}, 
we are in position to express $\partial_t B$ in terms of $\eta,\psi$ only. 
Setting the result thus obtained in~\eqref{defi:aeps} 
we get an expression of $\tilde{a}-g$ in terms of $\eta,\psi$ only. 
The estimate~\eqref{esti:tildea-s1} then follows from the 
nonlinear estimates in Sobolev spaces 
(see~\eqref{prS2} and~\eqref{esti:F(u)}) and repeated uses 
of the estimate 
for the Dirichlet-Neumann proved in Theorem~\ref{coro:estiDN}. 
Repeating this reasoning, we get an expression of $\partial_t\tilde{a}$ 
in terms of $\eta,\psi$ only and 
hence~\eqref{esti:tildea-s0}. 

Up to now, this gives uniform estimates in $\eps$ for~$(U,\theta)$ in~$L^\infty([0,T];H^s\times H^s)$, as long as $\tilde a(t)$ is bounded from below by $a_0 /2$.  
Moreover, as in the proof of Lemma~\ref{ecBV} and Lemma~\ref{ecBVeta}, it follows from 
\eqref{esti:tildea-s0} that one can estimate 
$(\eta,B,V)$ in $L^\infty([0,T]; H^s\times H^{s_0}\times H^{s_0})$. 
Now we can recover estimates for the unknowns 
$(\eta,\psi,\B,V)$ in~$L^\infty([0,T];H^{s+\mez}\times H^{s+\mez}\times H^s\times H^{s})$ 
by a bootstrap argument as in the proof of Lemma~\ref{ecBVs}.  

Using these uniform estimates, we get by a bootstrap argument that the solutions to~\eqref{A2} exist on 
a time interval independent of $\eps$ (and satisfy $a(t) > a_0/2$ according to~\eqref{esti:tildea-s1} \eqref{esti:tildea-s0}) on this time interval uniform estimates. 
To conclude it remains to pass to the limit in~\eqref{A2}. For this we use again 
the uniform estimates which hold as long as the Taylor coefficient $a(t)$ is bounded from below by $a_0/2$ (and the equation to bound the time derivatives) 
to extract subsequences converging weakly and the fact that the Dirichlet-Neumann operator, though 
nonlocal, passes to the limit. 
Eventually, we use again a bootstrap argument to control the Taylor coefficient using \eqref{esti:a1} and \eqref{esti:a2}.
\end{proof}

\end{document}